\documentclass[a4paper, 11pt, twoside, leqno]{article}

\usepackage[T1]{fontenc}
\usepackage[utf8x]{inputenc}
\usepackage[english]{babel}
\usepackage{lmodern}               
\usepackage{amsmath,amsthm,amssymb}
\usepackage{mathrsfs,esint,bbm,stmaryrd}
\usepackage[shortlabels]{enumitem}
\usepackage{graphicx,xcolor,subcaption}
\usepackage[textwidth=16.1cm,centering,headheight=24pt]{geometry}
\usepackage{authblk}

\newtheorem{theorem}{Theorem}           
\newtheorem{corollary}[theorem]{Corollary}
\newtheorem{lemma}[theorem]{Lemma}
\newtheorem{prop}[theorem]{Proposition}

\newtheorem{mainthm}{Theorem}           

\theoremstyle{definition}              
\newtheorem{definition}{Definition}

\newtheorem*{notation}{Notation}

\theoremstyle{remark}                  
\newtheorem{step}{Step}

\newtheorem{remark}{Remark}

\DeclareMathOperator{\dist}{dist}                                   
\DeclareMathOperator{\spt}{spt}                                     
\let\div\relax
\DeclareMathOperator{\div}{div}                                     

\DeclareMathOperator{\SBV}{SBV}

\DeclareMathOperator{\diam}{diam}

\newcommand{\abs}[1]{\left| #1 \right|}                             
\newcommand{\norm}[1]{\left\| #1 \right\|}                          
\newcommand{\one}{\mathbbm{1}}                                      

\newcommand{\csubset}{\subset\!\subset}                             
\newcommand{\ip}[2]{\left\langle #1, #2 \right\rangle}

\DeclareMathAlphabet{\mathpzc}{OT1}{pzc}{m}{it}

\newcommand{\D}{\mathrm{D}}       
\newcommand{\T}{\mathrm{T}}
\renewcommand{\d}{\mathrm{d}}
\renewcommand{\o}{\mathrm{o}}

\newcommand{\R}{\mathbb{R}}

\newcommand{\C}{\mathbb{C}}

\newcommand{\Q}{\mathbf{Q}}
\newcommand{\M}{\mathbf{M}}
\newcommand{\I}{\mathbf{I}}
\newcommand{\X}{\mathbf{X}}
\newcommand{\N}{\mathbf{N}}

\newcommand{\n}{\mathbf{n}}
\newcommand{\m}{\mathbf{m}}

\newcommand{\e}{\mathbf{e}}

\renewcommand{\u}{\mathbf{u}}
\newcommand{\NN}{\mathscr{N}}     

\newcommand{\F}{\mathscr{F}}
\renewcommand{\H}{\mathscr{H}}

\newcommand{\eps}{\varepsilon}

\newcommand{\Cpot}{{\rm C}_{\rm pot}}

\SetSymbolFont{stmry}{bold}{U}{stmry}{m}{n}  
\newcommand{\nnu}{{\boldsymbol{\nu}}}

\newcommand{\ttau}{{\boldsymbol{\tau}}}

\newcommand{\Sz}{\mathscr{S}_0^{2\times 2}}
\newcommand{\Qb}{\Q_{\mathrm{bd}}}
\newcommand{\Mb}{\M_{\mathrm{bd}}}

\renewcommand{\v}{\mathbf{v}}

\renewcommand{\S}{\mathrm{S}}

\newcommand{\mres}{\mathbin{\vrule height 1.6ex depth 0pt width
0.13ex\vrule height 0.13ex depth 0pt width 1.3ex}}

\title{The formation of gradient-driven singular structures of
codimension one and two in two-dimensions: \\
The case study of ferronematics \\
Part~{II}: Refined structure of the energy-concentration set}

\date{}
\author{Giacomo~Canevari, Federico~Luigi~Dipasquale, Bianca~Stroffolini}

\newcommand{\Addresses}{{
  \bigskip
  \footnotesize

  Giacomo~Canevari \\
  \textsc{Universit\`{a} di Verona} \\
  Strada Le Grazie 15, 37134 Verona, Italy \\
  \textit{E-mail address}: \texttt{giacomo.canevari@univr.it}

  \medskip

  Federico~Luigi~Dipasquale \\
  \textsc{Scuola Superiore Meridionale}\\
  Via Mezzocannone 4, 80138 Napoli, Italy\\
  \textit{E-mail address}: \texttt{f.dipasquale@ssmeridionale.it}

    \medskip

    Bianca~Stroffolini \\
    \textsc{Dipartimento di Matematica e Applicazioni ``Renato Caccioppoli'',}\\
    \textsc{Universit\`{a} degli studi di Napoli ``Federico II''}\\
    Via Cintia, Monte S. Angelo, 80126 Napoli, Italy\\ 
  \textit{E-mail address}: \texttt{bstroffo@unina.it}
}}

\begin{document}
\maketitle

\begin{abstract}
In this paper, we continue our study, started in~\cite{CDS1}, of a two-dimensional variational model for ferronematics --- composite materials formed by dispersing magnetic nanoparticles into a liquid crystal matrix. The model features two coupled order parameters: a Landau-de Gennes~$\Q$-tensor for the liquid crystal component and a magnetisation vector field~$\M$, both of them governed by a Ginzburg-Landau-type energy. The energy includes a singular coupling term favouring alignment between~$\Q$ and~$\M$. We analyse the asymptotic behaviour of (not necessarily minimizing) critical points as a small parameter~$\eps$ tends to zero. 
While in~\cite{CDS1} we showed that the (rescaled) energy density for the~$\Q$-component concentrates, to leading order, on a finite number of singular points, in this paper we prove the energy density for the~$\M$-component concentrates along a one-dimensional rectifiable set.
Moreover, we prove that the curvature of the singular set for the $\M$-component (technically, the first variation of the associated varifold) is concentrated on a finite number of points, i.e.~the singular set for the~$\Q$-component. 
Crucial to our arguments will be the energy estimates and compactness results proved in~\cite{CDS1}.


\end{abstract}

In this paper, we continue our study, started in~\cite{CDS1}, of 
the asymptotic behaviour of critical points of the energy functional
$\F_\eps$, defined in~\eqref{eq:Feps} below, that 
has been proposed in the physical literature~\cite{bisht2019} as a simplified model for  
two-dimensional ferronematics. Ferronematics are
a class of composite materials obtained as suspensions
of magnetic nanoparticles in a nematic liquid crystal host~\cite{Brochard,MLDC}.
According to the approach proposed in~\cite{bisht2019}, ferronematics are described by two order parameters.
The orientation of the liquid crystal molecules
is described by the Landau-de Gennes $\Q$-tensor,
which is a map from the physical domain~$\Omega\subseteq\R^2$
to the space~$\Sz$ of~$2\times 2$, symmetric, real matrices
with trace equal to zero. Nonzero values of~$\Q$ correspond
to liquid crystal configurations with a well-defined direction of
molecular alignment, represented by the
eigenspace of~$\Q$ associated with the positive eigenvalue,
while~$\Q = 0$ indicates an isotropic state,
where all the directions of molecular alignment are equally likely.
The distribution of magnetic nanoparticles is
described by the average magnetisation vector,
$\M\colon\Omega\to\R^2$. The system is governed by a
free energy functional which depends on both~$\Q$ and~$\M$:
\begin{equation}\label{eq:Feps}
	\F_\eps(\Q, \M) :=  \int_\Omega \left\{ \frac{1}{2} \abs{\nabla \Q}^2 + \frac{\eps}{2} \abs{\nabla \M}^2 + \frac{1}{\eps^2} f_\eps(\Q,\M)\right\} {\d}x,
\end{equation}
where~$\eps$ is a non-dimensional parameter.
The interaction between the liquid crystal host and
the magnetic inclusions is mediated by the potential~$f_\eps$,
which takes the form
\begin{equation}\label{eq:potential}
	f_\eps(\Q, \M) := \frac{1}{4}\left(1-\abs{\Q}^2\right)^2
	+ \frac{\eps}{4}\left(1-\abs{\M}^2\right)^2 - \eps \beta \, \Q \M \cdot \M +  \kappa_\eps.
\end{equation}
Here~$\beta> 0$ is given and~$\kappa_\eps$ is an additive
constant, depending on~$\eps$ and~$\beta$ only,
uniquely determined by imposing that~$\inf f_\eps = 0$.
For positive values of~$\beta$, the potential~$f_\eps$
promotes alignment between the liquid crystal molecules
and the magnetisation vector. 
As in \cite{CDS1}, we consider two alternative sets of boundary conditions
for~$\Q_\eps$ and~$\M_\eps$.
One option is to impose
Dirichlet boundary conditions for both~$\Q_\eps$ and~$\M_\eps$:
\begin{equation} \label{bc}
 \Q_\eps = \Qb, \quad \M_\eps = \Mb \qquad \textrm{on } \partial\Omega.
\end{equation}
In this case, the boundary data~$\Qb\in C^1(\partial\Omega, \, \Sz)$,
$\Mb\in C^1(\partial\Omega, \, \R^2)$ are~$\eps$-independent
maps that satisfy
\begin{equation} \label{hp:bc}
 \abs{\Mb} = \left(\sqrt{2}\beta + 1\right)^{1/2},
 \qquad \Qb = \sqrt{2}\left(\frac{\Mb\otimes\Mb}{\sqrt{2}\beta + 1}
- \frac{\I}{2}\right)
\end{equation}
at any point of~$\partial\Omega$.

Alternatively to~\eqref{bc}-\eqref{hp:bc}, we consider 
`mixed' boundary conditions, i.e. Dirichlet boundary
conditions for~$\Q_\eps$ and homogeneous Neumann boundary
conditions for~$\M_\eps$:
\begin{equation} \label{bcbis}
 \Q_\eps = \Qb, \quad \partial_\nnu\M_\eps = 0
 \qquad \textrm{on } \partial\Omega,
\end{equation}
where~$\nnu$ is the exterior unit normal to~$\partial\Omega$.
We assume the boundary datum~$\Qb$
(independent from~$\eps$ and) of the form
\begin{equation} \label{hp:bcbis}
 \Qb = \sqrt{2}\left(\n_{\mathrm{bd}}\otimes\n_{\mathrm{bd}} - \frac{\mathbf{I}}{2}\right)
\end{equation}
on~$\partial\Omega$, for some
map~$\n_{\mathrm{bd}}\in C^1(\partial\Omega, \, \R^2)$.
Regardless of our choice of boundary conditions, we always \emph{assume} that there
exists a constant~${\rm C}_{\rm pot} > 0$, independent of~$\eps$, such that
\begin{equation} \label{hp:potential_bound}
 \frac{1}{\eps^2}\int_\Omega f_\eps(\Q_\eps, \, \M_\eps) \leq {\rm C}_{\rm pot}
\end{equation}
for any~$\eps$ small enough. 
One can provide sufficient conditions for~\eqref{hp:potential_bound} in terms of the domain and the boundary datum, see e.g.~\cite[Remark~1]{CDS1}.

By a \emph{critical point} of the energy functional $\F_\eps$, we mean 
a finite-energy solution~$(\Q_\eps, \, \M_\eps)$ of the
Euler-Lagrange system of equations
\begin{align}
  -&\Delta\Q_\eps + \dfrac{1}{\eps^2}(\abs{\Q_\eps}^2 - 1)\Q_\eps
  - \dfrac{\beta}{\eps}\left(\M_\eps\otimes\M_\eps
  - \dfrac{\abs{\M_\eps}^2}{2}\I\right)  = 0  \label{EL-Q} \\
  -&\Delta\M_\eps + \dfrac{1}{\eps^2}(\abs{\M_\eps}^2 - 1)\M_\eps
  - \dfrac{2\beta}{\eps^2}\Q_\eps\M_\eps = 0. \label{EL-M}
\end{align}
As show in~\cite{CDS1}, under boundary conditions as in~\eqref{bc}-\eqref{hp:bc} or as in~\eqref{bcbis}-\eqref{hp:bcbis} and upon assuming~\eqref{hp:potential_bound}, for any critical pair~$(\Q_\eps,\,\M_\eps)$ we have the fundamental energy bounds  
\begin{align}
	& \F_\eps(\Q_\eps,\,\M_\eps) \leq C \abs{\log \eps}, \label{eq:Feps-bdd-intro} \\
	& \int_\Omega \left\{ \eps \abs{\nabla \M_\eps}^2 + \frac{1}{\eps^2}f_\eps(\Q_\eps,\,\M_\eps) \right\}\,{\d}x \leq C \label{eq:LOT-bdd-intro},
\end{align} 
where the constant $C$ does not depend on $\eps$.
We consider the functions
\begin{align}
	\mu_\eps &:= \frac{1}{\abs{\log\eps}}
	\left( \frac{1}{2}\abs{\nabla \Q_\eps}^2 + \frac{\eps}{2} \abs{\nabla \M_\eps}^2 + \frac{1}{\eps^2}f_\eps(\Q_\eps,\,\M_\eps) \right), \label{eq:def-mu-eps} \\
	\nu_\eps &:=  \frac{\eps}{2} \abs{\nabla \M_\eps}^2 + \frac{1}{\eps^2}f_\eps(\Q_\eps,\,\M_\eps). \label{eq:def-nu-eps}
\end{align}
In view of~\eqref{eq:Feps-bdd-intro},~\eqref{eq:LOT-bdd-intro}, the families~$(\mu_\eps)_{\eps > 0}$, $(\nu_\eps)_{\eps > 0}$ are bounded in~$L^1(\Omega)$.

The following convergence theorem has been obtained in~\cite{CDS1}.
\setcounter{mainthm}{0}
\begin{mainthm}[{\cite{CDS1}}]\label{mainthm:asymp}
	Let~$\Omega\subseteq\R^2$ be a bounded, simply connected domain of class~$C^2$.
	Let $\{(\Q_\eps, \,\M_\eps)\}\subset W^{1,2}(\Omega, \, \Sz)\times W^{1,2}(\Omega, \, \R^2)$ be a sequence of critical
	points of $\F_\eps$ subject to either~\eqref{bc}--\eqref{hp:bc}
	or to~\eqref{bcbis}--\eqref{hp:bcbis}.
	Assume that the condition~\eqref{hp:potential_bound}
	is satisfied. Then, as $\eps \to 0$,
	\begin{enumerate}[(i)]
		\item\label{item:mainthm-asymp-conv-QM} $\Q_\eps \to \Q_\star$ strongly in $W^{1,p}(\Omega)$ for any $p < 2$
		and $\M_\eps \to \M_\star$ strongly in $L^p(\Omega)$ for any $p < +\infty$.
		\item\label{item:mainthm-asymp-Q*M*-j} The limiting maps~$\Q_\star$, $\M_\star$
		satisfy
		\[
			\abs{\Q_\star(x)} = 1, \quad
			\abs{\M_\star(x)} = \left(\sqrt{2}\beta + 1\right)^{1/2}, \quad
			\Q_\star(x) = \sqrt{2} \left( \frac{\M_\star(x) \otimes \M_\star(x)}{\sqrt{2}\beta+1} - \frac{\I}{2} \right)
		\]
		for a.e.~$x\in\Omega$ and
		\begin{equation}\label{eq:Q*-harm-intro}
			\div (\Q_\star \times \nabla \Q_\star) = 0
		\end{equation}
		in the sense of distributions in~$\Omega$.
		\item\label{item:mainthm-asymp-conv-mu*nu*} $\mu_\eps \rightharpoonup \mu_\star$ and
		$\nu_\eps \rightharpoonup \nu_\star$ weakly*
		as measures in $\Omega$.
		\item\label{item:mainthm-asymp-supports} $\spt\mu_\star$ is a finite set of points. 
		\item\label{item:mainthm-asymp-strong-conv} $\Q_\eps \to \Q_\star$ strongly in
		$W^{1,p}_{\rm loc}(\Omega \setminus \spt\mu_\star)$ for any $p < +\infty$.
	\end{enumerate}
\end{mainthm}
We recall that, by~\ref{item:mainthm-asymp-strong-conv} and~\eqref{eq:Q*-harm-intro}, 
$\Q_\star$ is a smooth harmonic map in $\Omega \setminus \spt \mu_\star$. 

While Theorem~\ref{mainthm:asymp} provides detailed information on the asymptotic 
convergence of the $\Q$-component of a sequence of critical pairs and on the structure of the limiting energy concentration set $\spt\mu_\star$, it gives much weaker information on the behaviour of the $\M$-component and no information on the structure of the limiting set $\spt\nu_\star$. 
In this paper, we characterise precisely the set $\spt\nu_\star$, we clarify 
its relationship with the set $\spt\nu_\star$, and we obtain stronger convergence 
properties for the $\M$-component outside the union of $\spt\mu_\star$ and $\spt\nu_\star$.

As the coupling term promotes 
alignment between $\M$ and the eigenvectors of $\Q$, the 
energy of $\M$ should concentrate on singular lines, corresponding 
to jumps in the eigenvector frame. Furthermore, keeping into account the 
results of~\cite{CanevariMajumdarStroffoliniWang, CDS3}, where the  
asymptotic behaviour of minimiser under boundary conditions 
as in~\eqref{bc}-\eqref{hp:bc} or as in~\eqref{bcbis}-\eqref{hp:bcbis} has been studied, 
we expect that the endpoints of these singular lines should be either points of 
$\spt\mu_\star$ or, possibly, points of $\partial \Omega$.

\subsection*{Main results}

At first sight, both~\eqref{EL-Q} and~\eqref{EL-M} look like
perturbed Ginzburg-Landau systems. In \cite{CDS1}, we have 
largely exploited the Ginzburg-Landau-type structure of the 
system~\eqref{EL-Q} in order to study the asymptotic behaviour of 
the $\Q$-component. In the present paper, we focus more closely 
on the system~\eqref{EL-M}. 

The starting point of our analysis is the fact that~\eqref{EL-M} 
can be regarded as the Euler-Lagrange equation of
a `perturbed' vectorial Allen-Cahn-type energy functional, i.e.
the functional $E_\eps$ defined in~\eqref{eq:Eeps-intro} below. 
This observation allows us to pursue the approach 
introduced in the recent paper~\cite{Bethuel-AC-Acta} for 
non-minimising solutions of the Allen-Cahn system. 
As we are going to see, a major difficulty is provided by 
the coupling term. To deal with this issue, we adapt  
the energy methods developed in~\cite{Bethuel-AC-Acta} 
and we combine them with 
the refined energy estimates and convergence results for the 
$\Q$-component obtained in~\cite{CDS1}.

\medskip

In order to make explicit the Allen-Cahn-type structure 
of the system~\eqref{EL-M}, we start by introducing the function 
$\ell\colon \Sz \times \R^2 \to \R$ given by
\begin{equation} \label{ell_eps-intro}
	\ell(\Q,\,\M) :=
	\frac{1}{4}\left(\abs{\M}^2 - 1 \right)^2 - \beta \Q \M \cdot \M
	+ \frac{1}{2}\left( \beta^2 + \sqrt{2}\beta  \right).
\end{equation}
for any $(\Q,\,\M) \in \Sz \times \R^2$.
It is easily checked (see Lemma~\ref{lemma:f-fixedQ})
that for every $\Q \in \Sz \setminus \{0\}$,
the function $\M \mapsto \ell(\Q,\,\M)$ has only two minimisers, given by
\[
	\M_\pm = \pm\left(\sqrt{2}\beta \abs{\Q} + 1\right)^{1/2} \n,
\]
where $\n$ is a unit eigenvector of $\Q$ relative to
its positive eigenvalue.
Now, given any sequence~$\{(\Q_\eps,\,\M_\eps)\}$
of critical points of $\F_\eps$, we define
\begin{equation}\label{eq:V-intro}
	V(\M_\eps) :=
	\ell(\Q_\eps, \, \M_\eps) - \inf_{y \in \R^2} \ell(\Q_\eps,\,y),
\end{equation}
and
\begin{equation}\label{eq:Eeps-intro}
	E_\eps(\M_\eps;\,G) := \int_G \left\{\frac{\eps}{2}\abs{\nabla \M_\eps}^2 +
	\frac{1}{\eps} V(\M_\eps) \right\}\,{\d}x,
\end{equation}
for any measurable set~$G\subseteq\Omega$.
Although not explicitly stressed in the notation,
the function~$V$ and the functional~$E_\eps$ actually depend on~$\Q_\eps$, too.
In particular, the wells of $V(\M_\eps)$ are not fixed, 
but \emph{move}, dependingly on both $x$ and $\eps$.

If $G \subset \Omega$ is a subset where $\abs{\Q_\eps} \geq 1/2$, say, for any
sufficiently small $\eps > 0$, then for any such $\eps$ the maps
\begin{equation}\label{eq:Mpm-intro}
	(\M_\pm)_\eps = \pm\left(\sqrt{2}\beta \abs{\Q_\eps} + 1\right)^{1/2} \n_\eps
\end{equation}
provide the only zeroes of $V(\M_\eps)$ in $G$. Note that, whenever
$(\M_+)_\eps$, $(\M_-)_\eps$ are defined, there holds
$\abs{(\M_+)_\eps - (\M_-)_\eps} > 2$.
We also notice that, basically by~\eqref{hp:potential_bound} and~\eqref{eq:LOT-bdd-intro} (see Proposition~\ref{prop:Eeps-bounded} for more details), the $\eps$-independent bound 
\begin{equation}\label{eq:Eeps-bdd-intro}
	E_\eps(\M_\eps;\,G) \leq \mathcal{E}_0
\end{equation}
is in force.

With the notation introduced above,
the Euler-Lagrange equation~\eqref{EL-M} can be recast in the form
\begin{equation}\label{eq:AC-M-intro}
	- \eps \Delta \M_\eps + \frac{1}{\eps} \nabla_\M V(\M_\eps) = 0,
\end{equation}
which is exactly the Euler-Lagrange equation of the functional $E_\eps$. 
Formally, the Equation~\eqref{eq:AC-M-intro} looks like an Allen-Cahn system.  
Although the fact the wells are not fixed prevents us 
from applying directly the results of~\cite{Bethuel-AC-Acta}, we can 
still follow the same line of thought. 
In particular, exactly as in \cite{Bethuel-AC-Acta}, 
it turns out that, besides the energy densities $\nu_\eps$ 
and their limit $\nu_\star$, another fundamental object is involved in the study for obtaining the structure properties of the set $\spt\nu_\star$: 
the \emph{limiting potential energy density}, $\zeta_\star$.

Proceeding as in~\cite{Bethuel-AC-Acta},
we define the \emph{potential energy densities}
\begin{equation}\label{eq:def-zeta-eps-intro}
	\zeta_\eps := \frac{1}{\eps} V(\M_\eps),
\end{equation}
As a consequence of the uniform energy bound~\eqref{eq:Eeps-bdd-intro}, 
the sequence $(\zeta_\eps)_{\eps > 0}$ is bounded in $L^1(\Omega)$, 
and thus  
we can extract a subsequence and find $\zeta_\star$,
a non-negative Radon measure in $\Omega$, so that
\[
	\zeta_\eps \rightharpoonup^* \zeta_\star
\]
weakly* as measures in $\Omega$ in the limit as~$\eps\to 0$. 
As we discuss in more detail in the last part of this 
introduction, it turns out that
\[
	\zeta_\star \leq \nu_\star \lesssim \zeta_\star
\]
in the sense of measures away from $\spt\mu_\star$, so that, loosely speaking, 
$\nu_\star$ and $\zeta_\star$ retain the same amount 
of information. However, exactly as in~\cite{Bethuel-AC-Acta}, 
$\zeta_\star$ is a much more convenient object to study than 
$\nu_\star$.

We are now ready to state our main theorems, which are best 
expressed using the language of the theory of varifolds 
(see, e.g., \cite{Simon} for a detailed account --- 
a brief review of basic terminology, as needed in this paper, 
is included in Section~\ref{sec:S*}.)
\begin{mainthm}\label{mainthm:B+C}
	The set $\spt\nu_\star$ is $\H^1$-rectifiable, with locally 
	finite measure. Upon setting 
	\[
		\mathfrak{S}_\star := \spt\nu_\star \setminus \spt\mu_\star,
	\]	
	the following holds.
	\begin{enumerate}[(i)]
	\item\label{item:density} 
	The limiting potential energy measure $\zeta_\star$ is absolutely 
	continuous with respect to the measure $\H^1 \mres \mathfrak{S}_\star$. 
	In other words, there exists a function $\mathfrak{v}_\star : \mathfrak{S}_\star \to \R^+$, 
	locally integrable with respect to the measure $\H^1 \mres \mathfrak{S}_\star$, 
	such that
	\[
		\zeta_\star = \mathfrak{v}_\star \,\H^1 \mres \mathfrak{S}_\star.
	\] 
	The function $\mathfrak{v}_\star$ is locally bounded both from above and 
	from below in any compact set $K \subset \Omega\setminus\spt\mu_\star$.
	\item\label{item:zeta*-weight-measure} The measure $\zeta_\star$ is the weight measure of the $\H^1$-rectifiable varifold 
	$\mathbb{V}_\star$ carried by $\mathfrak{S}_\star$ with density function $\mathfrak{v}_\star$. 
	\item\label{item:balance-law} The varifold $\mathbb{V}_\star$ is stationary in $\Omega \setminus \spt\mu_\star$ and 
	its first variation as a varifold in $\Omega$ is concentrated on $\spt\mu_\star$.
	\item\label{item:segments} The set $\spt\nu_\star$ is locally a union of segments, open relative to $\Omega$, 
	each of which having constant density, given by the value of $\mathfrak{v}_\star$ at any 
	point of the segment. In addition, apart from an exceptional, $\H^1$-null set, around any point $x_0$ of 
	$\mathfrak{S}_\star$, the singular set $\spt\nu_\star$ consists of exactly one segment, 
	with constant density $\mathfrak{v}_\star(x_0)$.
	\item\label{item:unif-conv-M} As $\eps \to 0$, $\M_\eps \to \M_\star$ in $L^\infty_{\rm loc}(\Omega \setminus (\spt\mu_\star \cup \spt\nu_\star))$.
	\end{enumerate}
\end{mainthm}
Theorem~\ref{mainthm:B+C} provides detailed information on the 
structure of the set $\spt\nu_\star$, obtained through the properties 
of the limiting measure $\zeta_\star$. 
 
Statement~\ref{item:balance-law} of Theorem~\ref{mainthm:B+C} is an abridged 
version of our second main theorem, Theorem~\ref{mainthm:balance-law} below, 
which establishes the precise form of the relationship between $\spt\mu_\star$ and $\mathbb{V}_\star$ (and thus, $\spt\nu_\star$). 
Before stating Theorem~\ref{mainthm:balance-law}, we 
need to introduce another bit of notation. 

Given distinct points~$a_1, \, \ldots, \, a_n$ in~$\Omega$,
we denote by~$\mathbb{W}_\star(a_1, \, \ldots, \, a_n)$ their Ginzburg-Landau renormalised energy, defined exactly as in~\cite{BBH} (we recall its characterising property in~\eqref{eq:ren-energy-BBH} below) and representing 
their interaction energy.
In particular, we will be interested in the renormalised
energy $\mathbb{W}_\star(a_1,\,\ldots,\,a_{n_\star})$
of the (distinct) points $a_1, \, \ldots, \, a_{n_\star}$ in
$\spt\mu_\star\cap\Omega$.
\begin{mainthm}\label{mainthm:balance-law}
	The first variation  $\delta \mathbb{V}_\star$ of the 
	varifold $\mathbb{V}_\star := \v( \mathfrak{S}_\star,\,\mathfrak{v}_\star)$ 
	satisfies
	\begin{equation}\label{eq:first-variation-intro}
		\delta \mathbb{V}_\star(\X) = \sum_{j=1}^{n_\star} z_j \cdot \X(a_j)
	\end{equation}
	for any $\X \in C^1_c(\Omega,\,\R^2)$, where
	\begin{equation}\label{eq:zj-intro}
		z_j = -\frac{1}{2} \nabla_{a_j} \mathbb{W}_\star\left(a_1,\dots,a_{n_\star}\right)
	\end{equation}
	for any $j\in\{1,\dots,n_\star\}$. Here
	$n_\star$ and~$a_1, \dots, a_{n_\star}$ denote, respectively,
	the cardinality and the points of $\spt \mu_\star \cap\Omega$.
\end{mainthm}
We can look at~\eqref{eq:first-variation-intro}--\eqref{eq:zj-intro}
as a \emph{balance law} which sheds light on the
relationship between $\spt\mu_\star$ and $\spt\nu_\star$:
the interaction energy between the point singularities
$a_1, \dots,a_{n_\star}$, encoded in the Ginzburg-Landau
renormalised energy $\mathbb{W}_\star(a_1,\dots,a_{n_\star})$,
is compensated by the first variation of $\mathbb{V}_\star$.
In particular, we see from~\eqref{eq:first-variation-intro} that $\mathbb{V}_\star$ is
stationary in $\Omega \setminus \spt\mu_\star$. 
(Notice that this implies that the `boundary' of $\spt \nu_\star$
must consist of points of $\partial \Omega$ or points of $\spt \mu_\star$.)
In the pure Allen-Cahn
setting (both in the scalar case~\cite{HutchinsonTonegawa} and in the
vectorial case on two-dimensional domains~\cite{Bethuel-AC-Acta}),
the stationarity of the limiting varifold is inherited from
Equation~\eqref{eq:hopf-intro}, whose validity ultimately
goes back to the fact that $\M_\eps$ solves the Allen-Cahn equation,
so that 
$\delta \mathbb{V}_\star = 0$ can be seen as a criticality condition
for~$\spt\nu_\star$. Similarly, in
the pure Ginzburg-Landau theory, the condition
$\nabla_{a_j} \mathbb{W}_\star(a_1,\dots,a_{n_\star}) = 0$ for any
$j \in {1,\dots,n_\star}$ represents the criticality condition for
$\spt\mu_\star \cap \Omega$
(cf. e.g. \cite[Theorem~VII.4]{BBH} and \cite[Theorem~IX.1]{BBO}).
It is therefore natural that the interaction between
the two components~$\Q_\eps$ and~$\M_\eps$
induces a balance between the first variation of~$\mathbb{V}$ and the gradient of the renormalised energy, rather than separate criticality of these two objects.
 
\subsection*{Structure of the proofs}

The proofs of the Theorem~\ref{mainthm:B+C} and Theorem~\ref{mainthm:balance-law} are substantially more complex than the analogous results for minimising solutions. Many of the arguments of~\cite{CanevariMajumdarStroffoliniWang, CDS3}
do not carry over to this context, because they are based on energy minimality.
Studying the properties of the limiting measure~$\nu_\star$ 
--- in particular, proving rectifiability of its support --- 
is a much more delicate task. A substantial difficulty is the lack of relevant monotonicity formulae for~$\nu_\eps$. Indeed, since we expect concentration of~$\nu_\eps$ on a $1$-dimensional set, we seek bounds on quantities of the form~$r^{-1}\int_{B_r(x_0)} \nu_\eps \, \d x$. However, a direct computation (based on a Pohozaev identity, see Lemma~\ref{lemma:pohozaev} below) gives
\begin{equation} \label{nonmonotonicity}
 \begin{split}
  r^2 \frac{\d}{\d r}\left(\frac{1}{r}
   \int_{B_r(x_0)} \nu_\eps \, \d x\right)
  &= \int_{B_r(x_0)} \left(
   \frac{1}{\eps^2} f_\eps(\Q_\eps, \, \M_\eps)
   - \frac{\eps}{2}\abs{\nabla\M_\eps}^2\right)\d x \\
  &\qquad + r\int_{\partial B_r(x_0)}
   \left(\eps\abs{\partial_\nnu\M_\eps}^2 + \frac{1}{2} \abs{\partial_\nnu\Q_\eps}^2 - \frac{1}{2}\abs{\partial_\ttau\Q_\eps}^2 \right) \d\H^1,
 \end{split}
\end{equation}
where $\nnu$ and $\ttau$ denote, respectively, the normal and the tangent field to
$\partial B_r(x_0)$.

Leaving aside the terms that depend on~$\Q_\eps$ for the moment, the first integral in the right-hand side of~\eqref{nonmonotonicity} is problematic, because it does not have a definite sign. While in the scalar case similar ``discrepancy'' terms can be bounded from below based on the maximum principle~\cite{Modica, HutchinsonTonegawa}, there is no reason why such bound should extend to the vectorial case.
(For instance, an example in \cite{Smyrnelis} shows that in the vectorial case
the pointwise discrepancy associated with entire solutions may not have a sign, in
sharp contrast with the scalar case \cite{Modica}.)
To circumvent this issue, we follow the strategy of the recent paper~\cite{Bethuel-AC-Acta}, which is ultimately based on refined energy estimates. Once again, we have to adapt this approach to make it compatible with the coupling terms.

The key ingredients to this purpose are the following facts proved in \cite{CDS1}: 
\begin{enumerate}[(a)]
 \item\label{item:intro-unif-conv-Qeps} 
 The clearing-out theorem for $\Q_\eps$, which ensures that,  
 for every compact set $K \subset \Omega \setminus \spt\mu_\star$, 
 there holds $\abs{\Q_\eps} \geq 1/2$ uniformly on $K$, for every $\eps$ 
 small enough, depending only on $K$. 
 In fact, for every compact set $K \subset \Omega \setminus \spt\mu_\star$, we have also 
 \begin{equation}
	\abs{\Q_\eps} \to 1 \qquad \mbox{ uniformly on } K \mbox{ as }\eps \to 0.\label{eq:unif-conv-mod-Qeps}
\end{equation}
 \item For every compact set $K \subset \Omega \setminus \spt\mu_\star$ and every 
 $p$ with $1 \leq p < +\infty$, there holds   
 \begin{equation}\label{eq:strong-p-conv-rhoeps}
 	\int_K \left\{ \abs{\nabla \abs{\Q_\eps}}^p + \left(\frac{\abs{\Q_\eps} - 1}{\eps} - \kappa_\star\right)^p \right\}\,{\d}x \to 0 
 	\end{equation}
 	as $\eps \to 0$, where $\kappa_\star = \frac{\beta}{2\sqrt{2}}\left(\sqrt{2}\beta + 1 \right)$ is a constant that could be seen as the ground level of the `modified' Ginzburg-Landau potential (the term in round brackets above) arising because of the interaction between the two order parameters.\label{item:intro-crucial-lemma} 
 \item The strong convergence $\Q_\eps \to \Q_\star$ in $W^{1,p}_{\rm loc}(\Omega \setminus \spt\mu_\star$) from item~\ref{item:mainthm-asymp-strong-conv} of 
 Theorem~\ref{mainthm:asymp}.\label{item:intro-strong-conv}
 \item $\Q_\star$ is a smooth harmonic map in $\Omega \setminus \spt\mu_\star$.\label{item:intro-harm}
\end{enumerate}
The first part of~\ref{item:intro-unif-conv-Qeps} is a straightforward consequence 
of \cite[Proposition~2.5]{CDS1} and~\eqref{eq:Feps-bdd-intro}. It ensures that, away 
from $\spt\mu_\star$, the Allen-Cahn structure is well defined, making our approach 
viable. The uniform convergence~\eqref{eq:unif-conv-mod-Qeps} comes as a consequence 
of uniform bounds on the $p$-Ginzburg-Landau energy of $\Q_\eps$, for every $p$ 
with $1 \leq p < +\infty$ (see \cite[Proposition~2.6 and Lemma~2.8]{CDS1}).

Item~\ref{item:intro-crucial-lemma} follows from Lemma~3.7 and Remark~3.6 
in \cite{CDS1}. 
It is actually crucial to the proof of the strong convergence in~\ref{item:intro-strong-conv} and it also entails that
\begin{equation}\label{eq:f-V-intro}
	\frac{1}{\eps^2} f_\eps(\Q_\eps,\,\M_\eps) - \frac{1}{\eps} V(\M_\eps)
	\to 0,
\end{equation}
strongly in $L^p(K)$, for every compact $K \subset \Omega \setminus \spt\mu_\star$
and any $p$ with $1 \leq p < +\infty$, see Lemma~\ref{lemma:V-f-limit-Lp} below. 

Concerning~\ref{item:intro-strong-conv}, it implies that, 
if $B(x_0,\,r) \csubset \Omega \setminus \spt\mu_\star$, then, as $\eps \to 0$,
\begin{align*}
	& \int_{\partial B_r(x_0)} \abs{\partial_\nnu\Q_\eps}^2 \,{\d}\H^1 \to \int_{\partial B_r(x_0)} \abs{\partial_\nnu\Q_\star}^2 \,{\d}\H^1 \\
	& \int_{\partial B_r(x_0)} \abs{\partial_\ttau\Q_\eps}^2 \,{\d}\H^1 \to \int_{\partial B_r(x_0)} \abs{\partial_\ttau\Q_\star}^2\,{\d}\H^1
\end{align*}
for almost every radius. 
Finally, the harmonicity of $\Q_\star$ away from $\spt\mu_\star$ implies that 
\[
	\int_{\partial B_r(x_0)} \abs{\partial_\nnu\Q_\star}^2\,{\d}\H^1 
	= \int_{\partial B_r(x_0)} \abs{\partial_\ttau\Q_\star}^2\,{\d}\H^1.
\]

Back to~\eqref{nonmonotonicity}, we observe that, by the above, 
the terms depending
on~$\Q_\eps$ on the second line actually
vanish in the limit as $\eps \to 0$, whenever
$B_r(x_0)$ is well-contained in $\Omega \setminus \spt\mu_\star$. 
Furthermore, by~\eqref{eq:f-V-intro}, as far as we are concerned with the limit 
$\eps \to 0$, we may replace the potential 
$\frac{1}{\eps^2} f_\eps(\Q_\eps,\,\M_\eps)$ with the \emph{simpler} 
potential $\frac{1}{\eps} V(\M_\eps)$ in~\eqref{nonmonotonicity} and, 
instead of the energy densities $\nu_\eps$, we may consider the 
\emph{simpler} energy densities
\[
	\widetilde{\nu}_\eps := \frac{\eps}{2}\abs{\nabla \M_\eps}^2 +
	\frac{1}{\eps} V(\M_\eps).
\]
In view of the uniform energy bound~\eqref{eq:Eeps-bdd-intro},
the sequence $(\widetilde{\nu}_\eps)_{\eps > 0}$ 
is bounded in $L^1(\Omega)$ and so, up to extraction of a subsequence,
$\widetilde{\nu}_\eps$ convergence weakly$^*$ in the sense of measures
as $\eps \to 0$. 
By~\eqref{eq:f-V-intro}, there holds 
\begin{equation}\label{eq:nu-tilde-eps-nu-star-intro}
	\widetilde{\nu}_\eps \rightharpoonup^* \nu_\star
	\qquad {\textrm{weakly$^*$ as measures in } K,}
\end{equation}
for every compact set $K \subset \Omega \setminus \spt\mu_\star$, as $\eps \to 0$.
We stress that here $\nu_\star$ is \emph{exactly} the limit of the measures $\nu_\eps$
that appears in Theorem~\ref{mainthm:asymp}. The convergence
in~\eqref{eq:nu-tilde-eps-nu-star-intro} says essentially
that, at first order,
all the relevant information on the limiting measure $\nu_\star$
is already contained in the simpler energy measures
$\widetilde{\nu}_\eps$. What is more, on the account of the results
in~\cite{Bethuel-AC-Acta}, we may expect ``clearing-out'' properties
and energy decay estimates for $\widetilde{\nu}_\eps$
(see Section~\ref{sec:clout-M} for the precise statements), leading to lower density bounds on~$\nu_\star$ (see Section~\ref{sec:lower-bounds-nu*}).

As already mentioned,
one major consequence of the presence of $\Q_\eps$ in the definition of $V(\M_\eps)$
is the
fact that the wells of $V(\M_\eps)$ are not fixed but \emph{move},
both with $x$ and $\eps$, as shown by~\eqref{eq:Mpm-intro}.
Most arguments
in~\cite{Bethuel-AC-Acta} are based on testing~\eqref{eq:AC-M-intro}
against maps involving the difference $\M_\eps - (\M_\pm)_\eps$, so that
the fact that $(\M_\pm)_\eps$ is not constant gives rise to terms
depending on $\Q_\eps$ and its derivatives. 
We manage them by using the above ingredients.
Combined with the
the strategy in \cite{Bethuel-AC-Acta} and through a
rather delicate analysis, they allow us
to obtain the estimates and the compactness
results for the $\M_\eps$-component
in Section~\ref{sect:Eeps}
and Section~\ref{sect:nustar}, respectively.
The main result at the $\eps$-level is the clearing-out 
property for $\M_\eps$ provided by Theorem~\ref{thm:clearing-out-M},  
which, passing to the limit as $\eps \to 0$, yields a corresponding 
clearing-out property for $\nu_\star$, item~\ref{item:cl-o} of 
Theorem~\ref{thm:properties-nu*} below. 

As in~\cite{Bethuel-AC-Acta}, 
the clearing-out theorem for $\nu_\star$ is a fundamental tool towards 
the analysis of the structure properties of the set $\spt\nu_\star$. First, it 
yields density lower bounds for $\nu_\star$ and the fact that the $\H^1$-measure 
of $\spt\nu_\star$ is locally finite (see Theorem~\ref{thm:properties-nu*}). 
Second, together with an appropriate Pohozaev-type inequality 
(Proposition~\ref{prop:analogue-of-Bethuel-4.6}), which yields the 
non-existence of `islands' of singular set (Theorem~\ref{thm:analogue-of-Bethuel-8}), 
it entails that $\spt\nu_\star$ is locally path connected, and thus an 
$\H^1$-rectifiable set. By exploiting the rectifiability of $\spt\nu_\star$ and, again, 
the clearing-out property, we prove in Theorem~\ref{thm:unif-conv-Meps} 
the locally uniform convergence of $\M_\eps$ towards $\M_\star$ 
away from $\spt\mu_\star \cup \spt \nu_\star$.

Continuing following~\cite{Bethuel-AC-Acta}, the crucial observation to overcome the aforementioned lack of monotonicity 
formulae for the rescaled Allen-Cahn energy is the fact 
$\zeta_\star$ solves the equation
\begin{equation}\label{eq:hopf-intro}
	\frac{\partial \omega_\star}{\partial {\bar{z}}}
	= 2 \frac{\partial \zeta_\star}{\partial z}
\end{equation}
in the sense of distributions in $\Omega \setminus \spt\mu_\star$.
Here, $\omega_\star$ is the \emph{limiting Hopf differential},
defined by
\begin{equation} \label{omega_star}
\begin{split}
	\omega_\star = \omega^{\rm Q} + \omega^{\rm M} := &\mbox{(weak*-)}\lim_{\eps \to 0}
	\left( \abs{\partial_1 \Q_\eps}^2 - \abs{\partial_2 \Q_\eps}^2 - 2i \partial_1 \Q_\eps \cdot \partial_2 \Q_\eps \right) \\
	& + \mbox{(weak*-)}\lim_{\eps \to 0} \eps \left( \abs{\partial_1 \M_\eps}^2 - \abs{\partial_2 \M_\eps}^2 - 2i \partial_1 \M_\eps \cdot \partial_2 \M_\eps \right).
\end{split}
\end{equation}
Because of the strong convergence~$\Q_\eps \to \Q_\star$ in 
$W_{\rm loc}^{1,2}(\Omega \setminus \spt\mu_\star)$ and since 
$\Q_\star$ is smooth harmonic in $\Omega \setminus \spt\mu_\star$, 
we have $\frac{\partial \omega_\star^{\rm Q}}{\partial \bar{z}} = 0$ and, therefore,~\eqref{eq:hopf-intro} can equivalently be written
(see Proposition~\ref{prop:hopf}) as
\[
	\frac{\partial \omega_\star^{\rm M}}{\partial {\bar{z}}}
	= 2 \frac{\partial \zeta_\star}{\partial z}
\]
in the sense of distributions in $\Omega \setminus \spt\mu_\star$.
According to~\cite{Bethuel-AC-Acta}, the system~\eqref{eq:hopf-intro}
could be interpreted as a sort of `modified Cauchy-Riemann conditions'.

A remarkable consequence of~\eqref{eq:hopf-intro} is
that, for any $x_0 \in \Omega$, the function
$r \mapsto \frac{\zeta_\star(B(x_0,\,r))}{r}$ is monotone non-decreasing
(see Proposition~\ref{prop:analogue-of-Bethuel-Prop-6}).
Neither an analogous monotonicity property nor any equation similar
to~\eqref{eq:hopf-intro} are known to hold for~$\nu_\star$.
This motivates us to focus our attention
on~$\zeta_\star$ instead of~$\nu_\star$.
Luckily, from the analysis of~$\zeta_\star$ we
still recover information on the set
$\mathfrak{S}_\star := \spt \nu_\star \setminus \spt \mu_\star$, 
because $\mathfrak{S}_\star$ coincides exactly with
the support of $\zeta_\star$.
Indeed, on the one hand the convergence~\eqref{eq:nu-tilde-eps-nu-star-intro}
and the equality~$\zeta_\star(\spt\mu_\star) = 0$
(which follows from the the monotonicity of $\frac{1}{r}\zeta_\star$)
imply that
\[
	\zeta_\star \leq \nu_\star
\]
in the sense of measures in $\Omega$ (see Proposition~\ref{prop:analogue-of-Bethuel-Prop-6}).

On the other hand, by monotonicity, 
$\zeta_\star$ has a density at every point $x_0 \in \Omega$, which yields 
the first part of statement~\ref{item:density} of Theorem~\ref{mainthm:B+C}. 
The monotonicity of $\zeta_\star$ and, once again, a suitable Pohozaev inequality, 
allow us to prove that the density of $\nu_\star$ exists a $\H^1$-a.e. point of $\mathfrak{S}_\star$ and is controlled by $\mathfrak{v}_\star$ through a constant depending only 
on $\Omega$ and $\beta$. Together with the density lower bounds for $\nu_\star$ 
coming from the clearing-out theorem, it gives the density lower bounds on $\mathfrak{v}_\star$ mentioned in the second part of~\ref{item:density}, while the upper bounds come 
as a consequence of~\eqref{eq:Eeps-bdd-intro} and, again, the monotonicity of $\zeta_\star$. As a consequence, 
\[
	\nu_\star \lesssim \zeta_\star \quad \mbox{as measures in } \Omega \setminus \spt\mu_\star \qquad \mbox{and} \qquad 
	\spt\zeta_\star = \mathfrak{S}_\star.
\]
Once statement~\ref{item:density} is proved, statement~\ref{item:zeta*-weight-measure} 
of Theorem~\ref{mainthm:B+C} is an immediate consequence of the definitions 
(see Section~\ref{sec:S*}). 
Given~\ref{item:zeta*-weight-measure} and assuming for a moment~\ref{item:balance-law} 
(i.e., Theorem~\ref{mainthm:balance-law}), 
item~\ref{item:segments} follows from a classical
structure theorem for stationary $1$-varifolds with bounded
density~\cite{AllardAlmgren}, 
together with~\cite[Theorem~1.3]{Bethuel-AC-Acta}.
Contrary to the case of minimisers, we do not know if
the density $\mathfrak{v}_\star$ is constant, in general,
although we suspect that the limiting
varifold $\mathbb{V}_\star$ has integer multiplicity, see Remark~\ref{rk:integrality}.

Finally, concerning Theorem~\ref{mainthm:balance-law}, 
relying on~\eqref{eq:hopf-intro} exactly as in \cite[Theorem~1.3]{Bethuel-AC-Acta},
we first prove in Theorem~\ref{thm:stationary-varifold} that  
$\mathbb{V}_\star$ is stationary in $\Omega \setminus \spt\mu_\star$, which 
implies that the first variation of $\mathbb{V}_\star$ as a varifold in $\Omega$ 
must be supported on the finite set $\spt\mu_\star$. 
Then, in Theorem~\ref{thm:first-variation}, we compute explicitly 
$\delta \mathbb{V}_\star$. The computation starts from the Pohazaev 
identity at the $\eps$-level, given by Lemma~\ref{lemma:pohozaev} below,
and it exploits once again the above toolkit
\ref{item:intro-unif-conv-Qeps}--\ref{item:intro-harm} in 
a decisive way to obtain the balance law~\eqref{eq:first-variation-intro}--\eqref{eq:zj-intro}.

\subsection*{Plan of the paper}
The paper is organised as follows. Section~\ref{sect:prelim} contains some notation and preliminary results 
imported from the companion paper~\cite{CDS1}. In Section~\ref{sec:clout-M} we develop refined energy 
estimates for $\M_\eps$, along the lines of~\cite{Bethuel-AC-Acta}, that lead in the end to the clearing-out 
theorem for $\M_\eps$, Theorem~\ref{thm:clearing-out-M}. 
Section~\ref{sect:nustar} is devoted to the analysis of the limit measure~$\nu_\star$, based on an adapting the arguments in~\cite{Bethuel-AC-Acta}. (More details on the
structure of the arguments can be found at the beginning of the section.)
The proof of Theorem~\ref{mainthm:B+C} and of Theorem~\ref{mainthm:balance-law} 
are given in Section~\ref{sec:S*}.

\begin{notation}
We use the symbol $\{X_\eps\}$ to denote a family of
objects indexed by the parameter $\eps > 0$, almost always
used as a shorthand to denote
the sequence $\left\{X_{\eps_k}\right\}_k$, where $\eps_k \to 0$ as
$k \to +\infty$.
Usually, for the sake of a lighter notation, we do not relabel subsequences.

In inequalities like $A \lesssim B$, the symbol $\lesssim$ means that there
exists a constant $C$, independent of $A$ and $B$, such that $A \leq C B$.
In particular,
dealing with sequences indexed by $\eps$, we use
$\lesssim$ to denote inequality up to a constant independent of $\eps$.
Whenever it is relevant, we keep track of the dependences
of the implicit constants. 
Whenever possible without inducing 
ambiguities, for the sake of a lighter notation, we avoid writing 
explicitly the measure of integration in the integrals. 
\end{notation}

\tableofcontents

\section{Preliminary results}
\label{sect:prelim}

\setcounter{equation}{0}
\numberwithin{equation}{section}
\numberwithin{definition}{section}
\numberwithin{theorem}{section}
\numberwithin{remark}{section}
\numberwithin{example}{section}

In this Section, we collect the results from the companion 
paper~\cite{CDS1} that will be needed in the following course.

\subsection{Properties of the potential~$f_\eps$.}

For the reader's convenience, we summarise below some properties of the potential~$f_\eps$. All these properties have been already proved either 
in~\cite{CanevariMajumdarStroffoliniWang} or in \cite{CDS1}, to which the 
reader will be referred for detailed proof. 
 
The constant $\kappa_\eps$ in~\eqref{eq:potential} is uniquely determined by imposing 
$\inf f_\eps = 0$. It turns out that (\cite[Lemma~B.3]{CanevariMajumdarStroffoliniWang})
\[
	\kappa_\eps = \frac{1}{2} \left(\beta^2 + \sqrt{2} \beta\right) + \kappa_\star^2 \eps^2 + \o(\eps^2), 
\] 
where
\begin{equation}\label{eq:k*}
	\kappa_\star := \frac{1}{2 \sqrt{2}}\beta\left( \sqrt{2}\beta + 1 \right).
\end{equation}
Upon using the function $\ell(\Q,\,\M)$ defined in~\eqref{ell_eps}, 
we can rewrite
\begin{equation} \label{f-ell}
 \frac{1}{\eps^2} f_\eps(\Q, \, \M)
  = \frac{1}{4\eps^2}\left(\abs{\Q}^2 - 1\right)^2
   + \frac{1}{\eps} \ell(\Q, \, \M) + \chi_\eps,
\end{equation} 
where
\begin{equation} \label{ell_eps}
 \ell(\Q, \, \M) := \frac{1}{4}\left(\abs{\M}^2 - 1\right)^2
  - \beta \, \Q \M \cdot \M
  + \frac{1}{2} \left(\beta^2 + \sqrt{2} \beta\right)
\end{equation}
and where
\begin{equation} \label{chi_eps}
 \chi_\eps := \frac{\kappa_\eps}{\eps^2}
  - \frac{1}{2\eps} \left(\beta^2 + \sqrt{2} \beta\right) \! 
\end{equation}
satisfies
\begin{equation}\label{eq:chi-to-k*}
	\chi_\eps \to \kappa_\star^2 \qquad \mbox{as } \eps \to 0.
\end{equation}

For later reference, we recall the following characterisation of minimisers of
$\ell(\Q, \, \cdot\,)$, for a given~$\Q \neq 0$.

\begin{lemma}[{\cite[Lemma~1.2]{CDS1}}] \label{lemma:f-fixedQ}
 For any~$\Q\in\Sz\setminus\{0\}$, the function~$\ell(\Q, \, \cdot\,)$
 has exactly two minimisers, given by
 \begin{equation} \label{minimisers_l}
  \M_{\pm} := \pm\left(\sqrt{2}\beta\rho + 1\right)^{1/2} \n
 \end{equation}
 where~$\rho:= \abs{\Q}$ and~$\n$ is 
 a unit eigenvector of~$\Q$ corresponding to its positive eigenvalue.
 Moreover, there holds
 \begin{equation} \label{ell-minimum}
  \begin{split}
   \min\ell(\Q, \, \cdot\,)
   = \ell(\Q, \, \M_{\pm})
   = \frac{\beta}{2} (1 - \rho)\left(\sqrt{2} + \beta + \beta\rho\right) \! .
  \end{split}
 \end{equation}
\end{lemma}


For~$\Q\in\Sz\setminus\{0\}$, let~$\Sigma(\Q)
:= \{\M_+, \, \M_-\}\subseteq\R^2$ be the set of minimisers
of~$\ell(\Q, \, \cdot\,)$, as given by~\eqref{minimisers_l}.
If~$\Q\neq 0$ and~$\dist(\M, \, \Sigma(\Q)) \leq 1$, we define
\begin{equation} \label{projection}
 \pi(\Q, \, \M) :=
 \begin{cases}
  \M_+ &\textrm{if } \abs{\M - \M_+} \leq 1 \\
  \M_- &\textrm{if } \abs{\M - \M_-} \leq 1.
 \end{cases}
\end{equation}
$\pi(\Q, \, \M)$ is the closest projection of~$\M$
to~$\Sigma(\Q)$. It is defined in a non-ambiguous way,
because~$\abs{\M_+ - \M_-} = 2(\sqrt{2}\beta\rho + 1)^{1/2} > 2$.
%

%
%
\begin{lemma}[{\cite[Lemma~1.4]{CDS2}}] \label{lemma:ell2}
 For any~$\Lambda > 1$, there exist positive
 constants~$\delta_0(\Lambda)$ and~$c_0(\Lambda)$
 (depending on~$\Lambda$ and~$\beta$ only)
 such that the following statement holds:
 for any~$(\Q, \, \M)\in\Sz\times\R^2$ such
 that~$\Lambda^{-1} \leq \abs{\Q} \leq \Lambda$
 and~$\dist(\M, \, \Sigma(\Q)) \leq \delta_0(\Lambda)$, there holds
 \[
  \nabla_{\M} \ell(\Q, \, \M)\cdot(\M - \pi(\Q, \, \M))
   \geq c_0(\Lambda)\Big(\ell(\Q, \, \M) - \ell(\Q, \, \pi(\Q, \, \M))\Big)
   \geq 0.
 \]
\end{lemma}

\subsection{The maximum principle}
A first consequence of the Euler-Lagrange 
equations~\eqref{EL-Q},~\eqref{EL-M} is the 
following maximum principle, the proof of 
which is detailed in \cite[Lemma~1.6]{CDS1}.
\begin{lemma} \label{lemma:max}
 The maps~$\Q^*_\eps$, $\M^*_\eps$ are smooth inside~$\Omega$
 and of class~$C^1$ up to the boundary of~$\Omega$.
 Moreover, there exist 
 a constant~$C_\beta$, depending only on $\beta$,
 and a constant $C_{\beta,\Omega}$, depending only on $\beta$ and $\Omega$,
 such that
 \begin{align}
  \norm{\Q^*_\eps}_{L^\infty(\Omega)}
   + \norm{\M^*_\eps}_{L^\infty(\Omega)} &\leq {C_\beta} \label{max-QM} \\
  \norm{\nabla\Q^*_\eps}_{L^\infty(\Omega)}
   + \norm{\nabla\M^*_\eps}_{L^\infty(\Omega)}
   &\leq \frac{{C_{\beta,\Omega}}}{\eps}. \label{max-gradients}
 \end{align}
\end{lemma}

\begin{remark} \label{rk:max}
 As observed in~\cite[Remark~1.6]{CDS1}, the proof of Lemma~\ref{lemma:max} 
 shows a little bit more. Indeed, we have, more precisely, 
 \begin{equation}\label{eq:Q-M-s*}
  \abs{\Q_\eps} \leq s_*(\beta, \, \eps), \quad
  \abs{\M_\eps}^2 \leq 1 + {{\sqrt 2} }\beta \, s_*(\beta, \, \eps)
  \qquad \textrm{in } \Omega,
 \end{equation}
 where~$s_*(\beta, \, \eps) > 1$ is the largest root of the polynomial
 $P(X) = X^3 - (1 + {\beta^2\eps})X - {\frac{\beta}{\sqrt{2}}} \eps$.
 Elementary calculus shows that~$s_*(\beta, \, \eps)\to 1$
  as~$\eps\to 0$, for any given value of~$\beta$. Then, by differentiating
  the constraint $P(s_*(\beta, \, \eps)) = 0$ with respect to~$\eps$,
  we deduce
  \begin{equation}\label{eq:s_*}
   s_*(\beta, \, \eps) = 1 + {\eps \kappa_\star + {\rm O}(\eps^2)}
  \end{equation}
  as~$\eps\to 0$.
\end{remark}

\subsection{Summary of the energy estimates}

The following theorem summarises the main energy estimates 
proved in~\cite{CDS1}.
\begin{theorem}[{\cite[Theorem~2.11]{CDS1}}]\label{lemma:energy-est}
	Let $\{(\Q_\eps,\,\M_\eps)\}$ be a sequence of
	critical points of the functional $\F_\eps$, subject to boundary
	conditions as in~\eqref{bc},~\eqref{hp:bc} or
	as in~\eqref{bcbis},~\eqref{hp:bcbis}. Assume
	that~\eqref{hp:potential_bound} holds. Then,
	\begin{align}
		&\F_\eps(\Q_\eps,\,\M_\eps) \lesssim \abs{\log\eps}, \label{eq:log-bound} \\
		&\int_\Omega \abs{\nabla \Q_\eps}^2 \lesssim \abs{\log\eps}, \\
		&\int_\Omega \eps \abs{\nabla \M_\eps}^2 \lesssim 1, \label{eq:M_bound} \\
		&\frac{1}{\eps^2}\int_\Omega \left(1-\abs{\Q_\eps}^2\right)^2 \lesssim 1 \label{eq:bdd-GL-pot},
	\end{align}
	where the implicit constants on the right-hand side depend only
	on $\beta$, $\Cpot$, and
	the $L^1(\partial \Omega)$- and the $L^2(\partial\Omega)$-norm of
	$\Qb \times \partial_\ttau \Qb$.

	Moreover, on any ball $B = B(x_0,\,R) \csubset \Omega$ on which
	$\abs{\Q_\eps} \geq 1/2$,
	there hold
	\begin{gather}
		\int_B \abs{\nabla \Q_\eps}^2\,{\d}x \leq C(x_0,\,R,\,\beta,\,\Cpot), \\
		\F_\eps(\Q_\eps,\,\M_\eps;\,B) \leq C(x_0,\,R,\,\beta,\,\Cpot) \label{eq:bounded-F-clearing-out},
	\end{gather}
	where the constant
	$C(x_0,\,R,\,\beta,\,\Cpot,\,\Qb)$ depends only on $x_0$, $R$,
	$\beta$, $\Cpot$, and the $L^1(\partial \Omega)$- and the $L^2(\partial\Omega)$-norm of $\Qb \times \partial_\ttau \Qb$.
	Consequently, if $K \subset \Omega$ is any compact set such
	that $\abs{\Q_\eps} \geq 1/2$ on $K$ for any $\eps$ small enough,
	then there holds
	\begin{equation}
		\lim_{\eps \to 0}\frac{\F_\eps(\Q_\eps,\,\M_\eps;\,K)}{\abs{\log\eps}}
		= 0.
	\end{equation}
\end{theorem}
\subsection{Pohozaev identity}
In several points of this work, we shall make use of the Pohozaev identity 
satisfied by critical points $(\Q_\eps,\,\M_\eps)$ of $\F_\eps$, both at 
the $\eps$-level and the `$\star$-level', i.e., after taking the limit 
$\eps \to 0$. This identity has been already obtained (and exploited) 
in~\cite{CDS1}, so we just recall the result, referring the reader to~\cite{CDS1} 
for a detailed proof.

As in \cite{CDS1}, we define the \emph{stress-energy tensor} associated
to~$(\Q_\eps, \, \M_\eps)$ as
\begin{equation} \label{stressenergy}
 T_{jk}^\eps := \partial_j \Q_\eps\cdot\partial_k\Q_\eps
  + \eps \, \partial_j \M_\eps\cdot\partial_k\M_\eps
  - e_\eps(\Q_\eps, \, \M_\eps) \, \delta_{jk}
\end{equation}
for~$(j, \, k)\in\{1, \, 2\}^2$.
Here~$e_\eps(\Q_\eps, \, \M_\eps)$ is the energy density,
defined by
\begin{equation} \label{energydensity}
 e_\eps(\Q_\eps, \, \M_\eps) := \frac{1}{2}\abs{\nabla\Q_\eps}^2
 + \frac{\eps}{2}\abs{\nabla\M_\eps}^2 + \frac{1}{\eps^2} f_\eps(\Q_\eps, \, \M_\eps)
\end{equation}
Then, the following identity holds.
\begin{lemma}[{\cite[Lemma~1.7]{CDS1}}] \label{lemma:stressenergy}
Let $G \subseteq \Omega$ be any open set with boundary $\partial G$ of 
class $C^1$ and let $X \in C^1_c(\R^2;\,\R^2)$. Then, for any solution 
$(\Q,\,\M)$ to~\eqref{EL-Q},~\eqref{EL-M} there holds
  \begin{equation} \label{stren4}
  \begin{split}
   \int_G T_{jk}^\eps \, \partial_j X_k
   &= \int_{\partial G} \left( (\nnu\cdot\nabla\Q)\cdot(\X\cdot\nabla\Q)
    + \eps \, (\nnu\cdot\nabla\M)\cdot(\X\cdot\nabla\M)\right) \d s \\
   &\hspace{1.5cm} - \int_{\partial G} (\X\cdot\nnu)\, e_\eps(\Q, \, \M) \, {\d}s.
  \end{split}
 \end{equation}
In particular, 
the stress-energy tensor satisfies $\partial_j T_{jk}^\eps = 0$ in~$\Omega$.
\end{lemma}
For later reference, we point out an immediate consequence of Lemma~\ref{lemma:stressenergy}.
\begin{lemma} \label{lemma:pohozaev}
 For each ball~$B = B(x_0, \, R)\subseteq\Omega$,
 there holds
 \begin{equation}\label{eq:pohozaev-ball}
  \begin{split}
   \frac{2}{\eps^2} \int_{B} f_\eps(\Q_\eps, \, \M_\eps)\,{\d}x
    &+ \frac{R}{2}\int_{\partial B} \left( \abs{\partial_{\nnu}\Q_\eps}^2
     + \eps \abs{\partial_{\nnu}\M_\eps}^2\right) \d s\\
   &= \frac{R}{2}\int_{\partial B}
    \left( \abs{\partial_{\ttau}\Q_\eps}^2
	+ \eps \abs{\partial_{\ttau}\M_\eps}^2
	+ \frac{2}{\eps^2} f_\eps(\Q_\eps, \, \M_\eps)\right) \d s
  \end{split}
 \end{equation}
 where~$\nnu$ is the outward unit normal and~$\ttau$
 is the unit tangent field on~$\partial B$, oriented in such a way
 that $(\nnu, \, \ttau)$ is a positive basis.
\end{lemma}
\begin{proof}
 The lemma follows by taking~$G = B(x_0, \, R)$
 and~$\X(x) := \varphi(x)(x - x_0)$ in~\eqref{stren4}, 
 where $\varphi : \R^2 \to \R$ is a smooth cut-off function 
 such that $\varphi \equiv 1$ in a neighbourhood of $\Omega$.
\end{proof}
Finally, we recall that, being a smooth harmonic map in 
$\Omega \setminus \spt \mu_\star$, $\Q_\star$ satisfies 
an independent Pohozaev identity. More precisely, if $\X$ is any smooth vector field with compact support in $G \subseteq \Omega \setminus \spt \mu_\star$, then we have
\begin{equation}\label{eq:pohozaev-Q*}
	\int_G \left\{ \partial_j \Q_\star \cdot \partial_k \Q_\star \, \partial_j X_k - \frac{1}{2} (\div \X) \abs{\nabla \Q_\star}^2 \right\}\,{\d}x = 0.
\end{equation}
Equation~\eqref{eq:pohozaev-Q*} is classical and it is obtained by 
multiplying the harmonic map equation $-\Delta \Q_\star = \abs{\nabla \Q_\star}^2 \Q_\star$ by $\X \cdot \nabla \Q_\star$ and integrating by parts twice.

\section{Refined energy estimates for $\M_\eps$}
\label{sect:Eeps}

In this section, we take advantage of the `Allen-Cahn type' structure of 
the energy functional $E_\eps$ introduced in~\eqref{eq:Eeps-intro}
to derive several refined energy estimates
for the $\M_\eps$-component of a sequence $\{(\Q_\eps,\,\M_\eps)\}$
of critical points of the functional $\F_\eps$.

The key point of our argument is the fact that
the Euler-Lagrange equation of $E_\eps$ is precisely~\eqref{EL-M},
which, as we have already seen, can be recast in the more convenient form~\eqref{eq:AC-M-intro}. 
However, although $E_\eps$ looks like an Allen-Cahn functional
and, formally,~\eqref{eq:AC-M-intro} looks like an Allen-Cahn system,
the coupling between $\Q_\eps$ and $\M_\eps$
make the wells of the potential function $V$ in~\eqref{eq:V-intro}
dependent on $\Q_\eps$, putting our problem well outside the context
considered in \cite{Bethuel-AC-Acta}.
Nonetheless, we manage so as to exploit the improved bounds on $\Q_\eps$
from the companion paper \cite{CDS1} \emph{away from
the set $\spt \mu_\star$} given by item~\ref{item:mainthm-asymp-supports} of Theorem~\ref{mainthm:asymp}
to keep the perturbation under sufficiently precise control.
This allows us to parallel the path
traced in \cite{Bethuel-AC-Acta} \emph{away from $\spt\mu_\star$},
so that we are able to establish several a priori
estimates for critical points at the $\eps$-level starting from Equation~\eqref{EL-M}.
These estimates 
lead, ultimately, to a crucial \emph{clearing-out property} for $\M_\eps$,
expressed by Theorem~\ref{thm:clearing-out-M}.
In turn, Theorem~\ref{thm:clearing-out-M} is key to analyse
the limiting situation as $\eps \to 0$ and, indeed, it is the cornerstone of the
further developments worked out in Section~\ref{sect:nustar}.

\subsection{The `perturbed Allen-Cahn energy' $E_\eps$}
\label{sec:Eeps}
We start by studying the properties of the energy functional
$E_\eps$, already defined in~\eqref{eq:Eeps-intro} in the introduction,
and the associated potential $V$, defined in~\eqref{eq:V-intro}.
(For convenience, all the relevant definitions are recalled below.)
The potential $V$ depends on both $\Q$ and $\M$, though,
we shall see, $\Q$ enters essentially as a `parameter'.
The critical points of $E_\eps$
satisfy the Euler-Lagrange equation~\eqref{EL-M}, which is
more conveniently recast in the form~\eqref{eq:M}.
As already mentioned in the Introduction, 
whenever $\Q \neq 0$, $E_\eps$ can be
seen as a `perturbed Allen-Cahn energy'. 
The `perturbation'
has to be sought in the terms depending on $\Q$ in the
potential $V$, which make the wells \emph{move} with the
point $x \in \Omega$ instead of staying fixed, which is instead
the case in the usual Allen-Cahn theory.
In this section, we point out a few general properties
of $E_\eps$ and $V$. In particular, we prove a global,
$\eps$-independent bound for the energy of critical points
(see Proposition~\ref{prop:Eeps-bounded} below). Further,
we introduce some notation and terminology that will be
extensively used in the next sections.

\vskip5pt

Given any pair $(\Q,\,\M) \in \Sz \times \R^2$, we let
\[
	V(\M) := \ell(\Q, \M) - \inf_{y \in \R^2} \ell(\Q,\,y),
\]
where $\ell$ is the function defined in~\eqref{ell_eps}.
Notice that $V \geq 0$. 
%
If $G \subset \Omega$, $x \in G$,
and $x \mapsto \Q(x)$ is a $\Q$-tensor field and $x \mapsto \M(x)$
is a vector field, we define
\begin{equation}\label{eq:V}
	V(x,\,\M(x)) := \ell(\Q(x), \M(x))
	- \inf_{y \in \R^2} \ell(\Q(x),\,y).
\end{equation}
Note that, since the infimum of $\ell(\Q(x),\cdot\,)$ does not depend on
$\M$ (see~\eqref{ell-minimum}), we have
\begin{equation}\label{eq:nablaV=nablaL}
	\nabla_{\M} V(\,\cdot\,,\,\M) = \nabla_{\M} \ell(\Q,\,\M).
\end{equation}
for any $x \in G$.

\begin{lemma}\label{lemma:pointwise-est-V}
Let $\left(\Q_\eps,\,\M_\eps \right)$ be any critical point of $\F_\eps$ subject
to boundary conditions either as in~\eqref{bc}--\eqref{hp:bc} or as
in~\eqref{bcbis}--\eqref{hp:bcbis} and assume that~\eqref{hp:potential_bound} holds.
Then, for every $x \in \Omega$ there holds the estimate
\begin{equation}\label{eq:pointwise-est-V}
\begin{split}
	\frac{1}{\eps^2} f_\eps(\Q_\eps(x),\,\M_\eps(x)) - \frac{1}{\eps}V(x,\,\M_\eps(x))
	\lesssim \frac{1}{\eps^2}\left( 1 -\rho_\eps^2 \right)^2 + \rm{O}_{\eps \to 0}(1),
\end{split}
\end{equation}
where the implicit constant on the right-hand side depends only on $\beta$.
\end{lemma}

\begin{proof}
Recalling~\eqref{ell_eps},~\eqref{chi_eps},~\eqref{ell-minimum},
\eqref{f-ell}, \eqref{eq:chi-to-k*}, and~\eqref{max-QM}, we have the pointwise
identities
\begin{equation}\label{eq:pointwise-est-V-bis}
\begin{split}
	\frac{1}{\eps^2} f_\eps(\Q_\eps(x),\,\M_\eps(x)) &- \frac{1}{\eps}V(x,\,\M_\eps(x))\\
	&=
	\frac{1}{4\eps^2}\left( 1 -\abs{\Q_\eps}^2 \right)^2 + \chi_\eps
	+\frac{1}{\eps} \min_{y \in \R^2} \ell(\Q_\eps,\,y) \\
	&= \frac{1}{4\eps^2}\left( 1 -\rho_\eps^2 \right)^2
	+ \frac{\beta}{2}\left(\frac{1-\rho_\eps}{\eps}\right)\left( \sqrt{2} + \beta +\beta \rho_\eps \right)
	+ \kappa_\star^2 + \o_{\eps \to 0}(1),
\end{split}
\end{equation}
which, after applying Young's inequality, yields estimate~\eqref{eq:pointwise-est-V}.
\end{proof}

\begin{remark}\label{rk:V-finite}
As an immediate consequence of Lemma~\ref{lemma:pointwise-est-V},~\eqref{eq:bdd-GL-pot}, 
and assumption~\eqref{hp:potential_bound}, we have
the $\eps$-independent bound
\begin{equation}\label{eq:V-finite}
	\int_\Omega V\left(x,\,\M_\eps(x)\right) \,{\d}x \lesssim 1,
\end{equation}
where the implicit constant on the right-hand side depends only on
$\beta$ and the constant $\Cpot$ in \eqref{hp:potential_bound}.
A still immediate but even more important consequence
of Lemma~\ref{lemma:pointwise-est-V} will be provided by
Lemma~\ref{lemma:limit-f-minus-V} hereafter.
\end{remark}

Next, for any measurable set $G \subseteq \Omega$, we define
\begin{equation}\label{eq:E-M}
	E_\eps(\M; \, G) :=
	\int_G\left( \frac{\eps}{2}\abs{\nabla \M}^2 + \frac{1}{\eps} V(x,\,\M) \right)\,{\d}x.
\end{equation}
The uniform energy bound for critical points below is
immediate.
\begin{prop}\label{prop:Eeps-bounded}
Let $\left(\Q_\eps,\,\M_\eps\right)$ be a critical point of $\F_\eps$, subject
to boundary conditions either as in~\eqref{bc}--\eqref{hp:bc} or as
in~\eqref{bcbis}--\eqref{hp:bcbis}. Then,
\begin{equation}\label{eq:Eeps-bounded}
	E_\eps(\M_\eps;\,\Omega) \leq \mathcal{E}_0,
\end{equation}
where $\mathcal{E}_0$ is a positive number depending only
on $\beta$, the constant $\Cpot$ in~\eqref{hp:potential_bound}, 
and the (implicit) constant $C$ on the right-hand side of~\eqref{eq:M_bound}.
\end{prop}

\begin{proof}
	The conclusion follows immediately by
	combining~\eqref{eq:M_bound} and~\eqref{eq:V-finite}.
\end{proof}

With the notation introduced above,
the equation for $\M$, i.e., Equation~\eqref{EL-M}, can be recast in the form
\begin{equation}\label{eq:M}
	-\eps \Delta \M + \frac{1}{\eps}\nabla_{\M} V(\,\cdot,\,\M) = 0
\end{equation}
in any open set $G \subset \Omega$.
This compact writing will be conveniently employed in the following
course.
\begin{remark}[Scaling properties]\label{rk:scaling}
	Let $G = B(x_0,\,r) \subset \Omega$ be a ball. Then, arguing exactly as
	in \cite[Section~1.4.1]{Bethuel-AC-Acta}, upon setting
	\begin{gather}
		\widetilde{\eps}_r := \frac{\eps}{r}, \qquad
		\widetilde{\M}_{\widetilde{\eps}_r}(x) := \M_\eps(rx + x_0), \qquad
		\widetilde{\Q}_{\widetilde{\eps}_r}(x) := \Q_\eps(rx + x_0) \qquad
		(x \in B_1), \label{eq:Q-M-scaled} \\
		e_\eps(\M_\eps)(y) :=  \frac{\eps}{2}\abs{\nabla \M_\eps}^2(y)
		+ \frac{1}{\eps} V(y,\,\M_\eps) \qquad (y \in B(x_0,\,r)),
	\end{gather}
	we have that $\widetilde{\M}_{\widetilde{\eps}_r}$ solves~\eqref{eq:M}
	with $\eps$ changed into $\widetilde{\eps}_r$ and, moreover, there holds
	\begin{equation}\label{eq:scaling-energy-density}
		e_{\widetilde{\eps}_r}(\M_{{\widetilde \eps}_r})(x)
		= r e_\eps(\M_\eps)(rx+x_0), \qquad \forall x \in B_1.
	\end{equation}
	In fact, the separate identities
	\[
	\begin{aligned}
		\frac{\widetilde{\eps}_r}{2}\abs{\nabla \widetilde{\M}_{\widetilde{e}_r}}^2(x)
		&= r \left(\frac{\eps}{2} \abs{\nabla \M_\eps}^2(rx + x_0)\right), \\
		\frac{1}{\widetilde{\eps}_r} V_{\widetilde{\eps}_r}\left(x,\,\widetilde{\M}_{\widetilde{\eps}_r} \right) &= r^2 \left( \frac{1}{\eps}V(x,\,\M_\eps) \right)
	\end{aligned}
	\]
	hold for any $x \in B_1$, whence
	\begin{align}
		E_\eps(\M_\eps;\,B(x_0,\,r)) &= r E_{\widetilde{\eps}_r}\left(\widetilde{\M}_{{\widetilde\eps}_r};\,B_1\right), \label{eq:scaling-E-0} \\
		\frac{1}{\eps}\int_{B(x_0,\,r)} V(y,\,\M_\eps)\,{\d}y
		&= r\left(\frac{1}{\widetilde{\eps}_r}\int_{B_1} V_{\widetilde{\eps}_r}(x,\,\widetilde{\M}_{\widetilde{\eps}_r})\,{\d}x\right).
	\end{align}
	Thus, dividing both sides of~\eqref{eq:scaling-E-0} by $\eps$ shows that
	the quantity $\widetilde{\eps}_r^{-1} E_{\widetilde{\eps}_r}$ is scaling invariant:
	\begin{equation}\label{eq:scaling-inv-E}
		\eps^{-1} E_\eps(\M_\eps;\,B(x_0,\,r)) = \widetilde{\eps}_r^{-1} E_{\widetilde{\eps}_r}\left(\widetilde{\M}_{{\widetilde\eps}_r};\,B_1\right).
	\end{equation}
\end{remark}

Assume now $G \subset \Omega$ is a simply connected, open set, with smooth boundary.
Then, by well-known lifting results (cf., e.g., \cite[Section~1.1]{CDS1}),
given $\Q$, a $\Q$-tensor field which does not vanish in $G$,
we can find a pair $(\n,\,\m)$ of orthonormal eigenvectors of $\Q$
defined in the whole of $G$ so that we can write
\[
	\Q = \frac{\rho}{\sqrt{2}} (\n \otimes \n - \m \otimes \m)
	\qquad \mbox{in } G,
\]
where, as always in this paper, $\rho := \abs{\Q}$ and $\n(x)$
denotes an eigenvector of $\Q(x)$
related to its positive eigenvalue.

Next, we
let $\M_\pm = \pm \left( 1 + \sqrt{2}\beta \rho \right)^{1/2} \n$
denote the maps realising the minimum of $V$, i.e., of $\ell(\Q,\,\cdot\,)$
--- see Lemma~\ref{lemma:f-fixedQ} --- so that
\[
	V(x,\M(x)) = \ell(\Q(x),\,\M(x)) - \ell(\Q(x),\,\M_\pm(x))
\]
for any $x \in G$. We also set
\begin{equation}\label{eq:Sigma-Q}
	\Sigma(\Q(x)) := \left\{ \M_+(x), \, \M_-(x) \right\}
	\qquad \mbox{for } x \in G.
\end{equation}
Recall that, for $\Q \neq 0$,
\begin{equation}\label{eq:dist-wells}
	\abs{\M_+ - \M_-} = 2\left(1 + \sqrt{2}\beta \rho \right)^{1/2} > 2
\end{equation}
and we notice that, by a straightforward computation, 
for both $\N = \M_+$ and $\N = \M_-$ there holds 
 \begin{equation}\label{eq:hess-ell-N}
  \begin{split}
   \D^2_{\M}\ell(\Q, \, \N)
   = \sqrt{2}\beta\rho \, \I
   + (\sqrt{2}\beta\rho + 2)\n\otimes\n + \sqrt{2}\beta\rho\,\m\otimes\m
   = 2\sqrt{2}\beta\rho \, \I + 2\n\otimes\n
  \end{split}
 \end{equation}
so that
\[
  \begin{split}
   \D^2_{\M} \ell(\Q, \, \N)(\M - \N)\cdot(\M - \N)
   \geq 2\sqrt{2}\,\beta \rho \abs{\M - \N}^2
  \end{split}
 \]
for both $\N = \M_+$ and $\N = \M_-$. Thus, 
\begin{equation}\label{eq:quadratic-behaviour}
	2\sqrt{2}\beta \rho \abs{\M - \N}^2
	\leq \D_\M^2 V(\,\,\cdot,\,\N)(\M - \N)\cdot (\M - \N)
	\leq 2\left(1+\sqrt{2}\beta \rho\right) \abs{\M - \M_\pm}^2.
\end{equation}
for both $\N = \M_+$ and $\N = \M_-$.
As it is immediately checked from~\eqref{eq:hess-ell-N},
$(\n, \,\m)$ is also an orthonormal
basis of eigenvectors for $\D^2_\M V(x,\,\M_\pm(x))$ at any $x \in G$,
with eigenvalues $\lambda_\n$, $\lambda_\m$ given by
\[
	\lambda_\m = 2\sqrt{2}\beta\rho, \qquad
	\lambda_\n = 2\left(1+\sqrt{2}\beta\rho\right).
\]
In particular, if $G \subset \Omega$ is a simply connected, open set,
with smooth boundary and $\abs{\Q} > 0$ on $G$, then
$\D^2_\M V(x,\,\M_\pm(x))$ is a positive-definite
quadratic form at any $x \in G$. In addition,
if $\abs{\Q}$ is bounded from above and from below on $G$,
then from the proof of Lemma~\ref{lemma:ell2},
we have better control on the behaviour of $V$ near $\Sigma(\Q)$.
More precisely, we first observe that, if $1 /2 \leq \abs{\Q_\eps} \leq 2$ on $G$,
then
\begin{equation}\label{eq:lambda-pm}
	\lambda_- := \sqrt{2}\beta \leq \lambda_\m \qquad\mbox{and}\qquad
	\lambda_\n \leq \lambda_+ :=  2\left(1+2\sqrt{2}\beta\right)
	\qquad \mbox{on } G.
\end{equation}
so that the lemma below follows easily and provides a counterpart of~\cite[Proposition~2.1]{Bethuel-AC-Acta}.

\begin{lemma}\label{lemma:V-quadratic}
There exists a positive constant $\delta_\beta \in (0,\,1/4]$, depending
only on $\beta$, such that the following holds.
If $G \subset \Omega$ is an open, simply connected set
with smooth boundary and $\Q$ is a $\Q$-tensor field with
$\frac{1}{2} \leq \abs{\Q} \leq 2$ on $G$, then for any $x \in G$
and any $y \in \R^2$ such that $\dist(y,\,\Sigma(\Q(x))) < \delta_\beta$ there holds
\begin{align}\label{eq:V-quadratic}
	& \frac{1}{4}\lambda_- \dist(y,\,\,\Sigma(\Q(x)))^2 \leq V(x,\,y)
	\leq \lambda_+ \dist(y,\,\Sigma(\Q(x)))^2, \\
	&\frac{1}{2}\lambda_- \dist(y,\,\,\Sigma(\Q(x)))^2 \leq \nabla V(x,\,y) \cdot (y - \N) \leq 2 \lambda_+ \dist(y,\,\Sigma(\Q(x)))^2,
\end{align}
where $\N = \pi(\Q,\,\M)$ is given by~\eqref{projection}.
Moreover,
\begin{equation}\label{eq:lower-bound-V}
	V(x,\,y) \geq \gamma_\beta := \frac{\sqrt{2}}{2}\beta \delta_\beta^2 \qquad
	\mbox{for any } y \in \R^2 \mbox{ such that }
	\dist(y,\,\Sigma(\Q(x))) \geq \frac{\delta_\beta}{4}.
\end{equation}
\end{lemma}

\begin{proof}
	Being the proof completely elementary, we just give a quick sketch of it.

	The first line in~\eqref{eq:V-quadratic} follows immediately by Taylor-expanding
	$V$ close to the wells and using the absolute bounds~\eqref{eq:lambda-pm}, for an
	appropriate choice of $\delta_\beta$ depending only on $\beta$.
	The first inequality in the second line is a straightforward consequence of
	Lemma~\ref{lemma:ell2} with the choice $\Lambda = 2$, and possibly reducing
	$\delta_\beta$. The second inequality is proven in a
	completely analogous way.
\end{proof}
As an elementary consequence, we obtain the lemma below
(cf.~\cite[Lemma~2.2]{Bethuel-AC-Acta}).
\begin{lemma}\label{lemma:V-quadratic-bis}
	Let $G \subset \Omega$ be any open, simply connected set,
	with smooth boundary and $\Q$ any $\Q$-tensor field
	with $\frac{1}{2} \leq \abs{\Q} \leq 2$ on $G$.
	Then, for any $x \in G$ and any $y \in \R^2$ such that
	$V(x,\,y) \leq \gamma_\beta$ there holds
	\[
		\dist(y,\,\Sigma(\Q(x))) \leq \delta_\beta
	\]
	and, moreover,
	\[
		\dist(y,\,\Sigma(\Q(x))) \leq
		\sqrt{4\lambda_-^{-1} V(x,\,y)},
	\]
	where $\lambda_-$ is given by~\eqref{eq:lambda-pm}.
\end{lemma}


Let $B = B(x_0,\,R) \subset \Omega$ be a ball
such that $\frac{1}{2} \leq \abs{\Q} \leq 2$ on $B$.
Let $\kappa \in \left(0, \delta_\beta\right)$.
In analogy with \cite[Section~4]{Bethuel-AC-Acta}, we define
	\begin{equation}\label{eq:Upsilon}
		\Upsilon := \Upsilon(R,\,\kappa) :=
		\left\{ x \in B \,: \, \dist(\M(x), \Sigma(\Q(x)) \leq \kappa \right\}
	\end{equation}
	as well as
	\begin{equation}\label{eq:Ups-Gamma-Pi}
	\begin{split}
		& \Upsilon_\pm := \left\{ x \in B : \, \dist(\M(x),\, (\M_\pm)_\eps(x) \leq \kappa \right\}, \\
		& \Gamma_\pm := \partial \Upsilon_\pm \cap B, \qquad \Pi_\pm := \Upsilon_\pm \cap \partial B.
	\end{split}
	\end{equation}
	Notice that
	\[
		\Upsilon = \Upsilon_+ \cup \Upsilon_-.
	\]
	Moreover, by the choice of $\kappa$, the sets $\Upsilon_\pm$ are disjoint,
	so that
	\[
	\begin{split}
		E_\eps(\M;\,\Upsilon) &=
		\int_\Upsilon\left( \eps \abs{\nabla \M}^2 + \frac{1}{\eps}V(x,\,\M) \right)\,{\d}x \\
		&=
		E_\eps(\M;\,\Upsilon_+) + E_\eps(\M;\,\Upsilon_-).
	\end{split}
	\]
	Besides the sets $\Upsilon_\pm$, we define, for any $R > 0$, the set
	\begin{equation}\label{eq:def-Theta}
		\Theta(\M,\,R) :=
		\left\{ x \in B(x_0,\,R) :\, \dist(\M(x), \Sigma(\Q(x))) \geq \delta_\beta \right\}.
	\end{equation}
	Clearly,
	\[
		B(x_0,\,R) = \Upsilon\left(R,\, \delta_\beta \right)
		\cup \Theta(\M,\,R).
	\]
	\begin{remark}\label{rk:Qeps-less-2}
		Let $\eps > 0$ and let $(\Q_\eps,\,\M_\eps)$ be any solution
		to~\eqref{EL-Q},~\eqref{EL-M} subject to the boundary
		conditions~\eqref{bc}--\eqref{hp:bc} or \eqref{bcbis}--\eqref{hp:bcbis}.
		Then, by~\eqref{eq:Q-M-s*} and~\eqref{eq:s_*}, the assumption
		$\abs{\Q_\eps} \leq 2$ required in the above statements is
		automatically satisfied for any $0 < \eps \leq \eps_\beta$,
		where $\eps_\beta > 0$ depends only on $\beta$.
	\end{remark}

	\begin{remark}\label{rk:control-Theta}
		Let $\eps > 0$ and
		let $(\Q_\eps,\,\M_\eps)$ be any solution to~\eqref{EL-Q},~\eqref{EL-M}
		subject to the boundary conditions~\eqref{bc}--\eqref{hp:bc} or
		\eqref{bcbis}--\eqref{hp:bcbis}.
		Then, by the uniform gradient bound~\eqref{max-gradients}, i.e.,
		\[\norm{\nabla \M_\eps}_{L^\infty(\Omega)} \lesssim \eps^{-1},\]
		where the implicit constant on the right-hand side depends
		only on the coupling parameter $\beta$ and on $\Omega$
		(see~\eqref{max-gradients}),
		it follows that in the 
		set $\Theta(\M_\eps,\,R)$ the energy density is controlled
		by the potential energy density, through a constant depending
		only on $\beta$ and $\Omega$. In particular, we have
		\begin{equation}\label{eq:control-Theta}
			\eps \abs{(\nabla \M_\eps)(x)}^2 + \frac{1}{\eps} V(x,\,\M_\eps(x))
			\lesssim \frac{1}{\eps}V(x,\,\M_\eps(x))
			\qquad \mbox{for any } x \in \Theta\left(\M_\eps,\, R\right),
		\end{equation}
		where the implicit constant on the right-hand side depends
		only on $\beta$ and $\Omega$.
	\end{remark}

\subsection{Energy estimates through the potential}
\label{sec:energy-est-pot}
In this section, we extend to our context the
local energy estimates proved
in \cite[Proposition~4.8]{Bethuel-AC-Acta}
for the pure Allen-Cahn system.
As in \cite{Bethuel-AC-Acta}, the argument is mainly based on
refined energy estimates, obtained by
multiplying Equation~\eqref{EL-M} (i.e., Equation~\eqref{eq:M})
by suitable functions and then integrating
the result on appropriate domains.
We proceed in two steps, proving first a refined energy
estimate \emph{close} to the zero set of the potential $V(\,\cdot\,,\,\M)$
and then obtaining an estimate valid in a whole ball.
Since this strategy does not depend on the precise form of the
potential at hand and since Equation~\eqref{EL-M}
looks \textbf{formally} exactly the same as in \cite{Bethuel-AC-Acta}
(this is best seen by looking at the equivalent
form of~\eqref{EL-M} given by~\eqref{eq:M}),
on a formal level the same strategy works in our situation as well,
the difference being the appearance of perturbation terms containing the
variable $\Q$, which are due to the coupling term in $V(\,\cdot\,,\,\M)$.
However, 
we will see that
such additional terms are essentially of lower order
away from $\spt\mu_\star$ and they do not alter 
the general structure of the
arguments in \cite{Bethuel-AC-Acta}. The results of this section
are at the level of a single critical point, i.e., at fixed $\eps > 0$.
In the next sections, we shall consider \emph{sequences} of critical points
and we will deal with the limit as $\eps \to 0$.

The main result of this section is Proposition~\ref{prop:analogue-of-Bethuel-4.2-R}
below, which is the analogue in our context of
\cite[Proposition~4.8]{Bethuel-AC-Acta}. In the statements below,
$\eps_\beta > 0$ is the number, depending only on $\beta$, defined
in Remark~\ref{rk:Qeps-less-2}.
\begin{prop}\label{prop:analogue-of-Bethuel-4.2-R}
	Let $\eps \in (0,\,\eps_\beta]$ and let $(\Q,\,\M) = (\Q_\eps, \, \M_\eps)$
	be any critical point of $\F_\eps$.
	Let $B = B(x_0,\,R) \subset \Omega$ be any ball such that
	$\abs{\Q} \geq \frac{1}{2}$ on $B$. Let 
	$\kappa \in (0,\,\delta_\beta)$ and
	$\varrho_\eps \in \left[\frac{1}{2}, \frac{3}{4} \right]$.
	Assume that, for the given value of $\kappa$, there holds
	\begin{equation}\label{eq:assumption-Bethuel-R}
		\dist(\M,\,\Sigma(\Q)) < \kappa \qquad
		\mbox{on } \partial B.
	\end{equation}
	Then, there exists a constant ${\rm C}_\Upsilon(\mathcal{E}_0) > 0$,
	depending only on $\mathcal{E}_0$ and $\beta$, such that
	\begin{equation}\label{eq:est-close-to-Sigma-R}
	\begin{split}
		E_\eps(\M; \, \Upsilon(\varrho_\eps R,\, \kappa)) \leq
		{\rm C}_\Upsilon(\mathcal{E}_0)&\left[ \kappa \int_{B(x_0,\,\varrho_\eps R)} \left(  \frac{1}{\eps}V(x,\,\M(x)) + {\eps}\abs{\nabla \Q}^2 \right)\,{\d}x\right.\\
		&\left. + \eps\left( E_\eps(\M; \, \partial B(x_0,\, \varrho_\eps R)) + \eps\int_{\partial B(x_0,\,\varrho_\eps R)} \abs{\nabla \Q}^2 \,{\d}\sigma \right)  \right].
	\end{split}
	\end{equation}
	Here, $\mathcal{E}_0$ is the uniform bound on $E_\eps$ provided by
	Proposition~\ref{prop:Eeps-bounded}.
\end{prop}

\begin{remark}
A crucial feature of the estimate~\eqref{eq:est-close-to-Sigma-R}
is that, exactly as in \cite{Bethuel-AC-Acta}, it is
\textbf{linear in $\kappa$}.
\end{remark}

\begin{remark}
	As seen in Proposition~\ref{prop:Eeps-bounded},
	the energy bound $\mathcal{E}_0$ (and hence ${\rm C}_\Upsilon$) depends on $\beta$,
	the constant $\Cpot$ in~\eqref{hp:potential_bound}, $\Omega$,
	$\norm{\Qb \times \partial_\ttau \Qb}_{L^1(\partial \Omega)}$, and
	$\norm{\Qb \times \partial_\ttau \Qb}_{L^2(\partial \Omega)}$. However, we preferred to
	state Proposition~\ref{prop:analogue-of-Bethuel-4.2-R} in
	the form above to emphasise that, exactly as in \cite{Bethuel-AC-Acta},
	it holds, in fact, given
	\emph{any} uniform bound on $E_\eps$.
\end{remark}

It will be convenient to obtain Proposition~\ref{prop:analogue-of-Bethuel-4.2-R}
from its scaled version, Proposition~\ref{prop:analogue-of-Bethuel-4.2} below,
and Remark~\ref{rk:scaling}.

\begin{notation}
In what follows, $\eps \in (0,\,\eps_\beta]$ is a fixed number and,
to simplify the notation,
$(\Q,\,\M)$ denotes a critical pair $(\Q_\eps,\,\M_\eps)$ after
rescaling according to~\eqref{eq:Q-M-scaled}.
We assume that $(\Q_\eps,\,\M_\eps)$
satisfies the same hypotheses as in
Proposition~\ref{prop:analogue-of-Bethuel-4.2-R}.
Furthermore, in all the results below up to the statement of
Proposition~\ref{prop:analogue-of-Bethuel-4.3},
we set $B = B_1$ and we always assume that $\abs{\Q} \geq \frac{1}{2}$
on $\overline{B}$. For notational convenience, we set $B_r := B(0,\,r)$.
\end{notation}


\begin{prop}\label{prop:analogue-of-Bethuel-4.2}
	Let 
	$\kappa \in (0,\,\delta_\beta)$ and
	$\varrho_\eps \in \left[\frac{1}{2}, \frac{3}{4}\right]$.
	Assume that, for the given value of $\kappa$, there holds
	\begin{equation}\label{eq:assumption-Bethuel}
		\dist(\M,\,\Sigma(\Q)) < \kappa \qquad
		\mbox{on } \partial B.
	\end{equation}
	Then, there exists a constant ${\rm C}_\Upsilon(\mathcal{E}_0) > 0$,
	depending only on $\mathcal{E}_0$ and $\beta$, such that
	\begin{equation}\label{eq:est-close-to-Sigma}
	\begin{split}
		E_\eps(\M; \, \Upsilon(\varrho_\eps, \kappa)) \leq
		{\rm C}_\Upsilon(\mathcal{E}_0)&\left[ \kappa \int_{B_{\varrho_\eps}} \left(  \frac{1}{\eps}V(x,\,\M(x)) + \eps \abs{\nabla \Q}^2 \right)\,{\d}x\right.\\
		&\left. + \eps\left( E_\eps(\M; \, \partial B_{\varrho_\eps}) + \eps\int_{\partial B_{\varrho_\eps}} \abs{\nabla \Q}^2 \,{\d}\sigma \right)  \right].
	\end{split}
	\end{equation}
	Here, $\mathcal{E}_0$ is the uniform bound on $E_\eps$ provided by
	Proposition~\ref{prop:Eeps-bounded}.
\end{prop}

The proof of Proposition~\ref{prop:analogue-of-Bethuel-4.2}
will be achieved through a series of lemmas, adapting to our setting
results in \cite[Section~4]{Bethuel-AC-Acta}.

\begin{lemma}\label{lemma:ineq-E-Q}
	We have
	\begin{equation}\label{eq:ineq-E-Q}
		E_\eps(\M;\, \Upsilon_\pm) \lesssim \mathfrak{Q}_\pm,
	\end{equation}
	where
	\[
		\mathfrak{Q}_\pm := \int_{\Upsilon_\pm}
		\left( \eps \abs{\nabla \M}^2 + \frac{1}{\eps}\nabla_{\M} V(x,\,\M) \cdot (\M-\M_\pm)\right)\,{\d}x
	\]
	and the implicit constant on the right-hand side in~\eqref{eq:ineq-E-Q}
	depends only on $\beta$.
\end{lemma}

\begin{proof}
	By definition of $\Upsilon_\pm$ and Lemma~\ref{lemma:V-quadratic},
	$\M$ is so close to the wells $\M_\pm$ on $\Upsilon_\pm$ that there holds
	\[
	\nabla_{\M} V(\,\cdot\,,\,\M) \cdot(\M - \M_\pm)
	\gtrsim \abs{\M - \M_\pm}^2 \gtrsim V(\,\cdot\,,\,\M)
	\]
pointwise, where the implicit constants depend only on $\beta$.
The conclusion follows immediately by dividing the above inequalities by $\eps$
and integrating the result over $\Upsilon_\pm$.
\end{proof}

\begin{lemma}\label{lemma:ineq-Qi+ineq-Gi}
	There holds
	\begin{equation}\label{eq:ineq-Qi}
	\begin{split}
		\mathfrak{Q}_\pm \lesssim & \eps\int_{\Upsilon_\pm} \abs{\nabla \M_\pm}^2 \,{\d}x
		+ \eps \int_{\partial \Upsilon_\pm} (\M - \M_\pm) \cdot \partial_{\nnu} \M_\pm \,{\d}\sigma \\
		&+ \eps \kappa \int_{\Gamma_\pm} \partial_{\nnu}\abs{\M-\M_\pm} \,{\d}\sigma + \eps E_\eps(\M;\,\Pi_\pm)
		+ \eps^2 \int_{\Pi_\pm} \abs{\nabla \M_\pm}^2 \,{\d}\sigma.
	\end{split}
	\end{equation}
	As a consequence,
	\begin{equation}\label{eq:ineq-Gi}
	\begin{split}
		E_\eps(\M; \Upsilon_\pm) \lesssim & \eps\int_{\Upsilon_\pm} \abs{\nabla \M_\pm}^2 \,{\d}x \\
		&+ \eps \int_{\Gamma_\pm} (\M - \M_\pm) \cdot \partial_{\nnu} \M_\pm \,{\d}\sigma
		+ \eps \kappa \int_{\Gamma_\pm} \partial_{\nnu}\abs{\M-\M_\pm} \,{\d}\sigma \\
		& + \eps^2 \int_{\Pi_\pm} \abs{\nabla \Q}^2 \,{\d}\sigma + \eps E_\eps(\M;\,\Pi_\pm).
		\end{split}
		\end{equation}
	The implicit constants on the right-hand sides
	of~\eqref{eq:ineq-Qi},~\eqref{eq:ineq-Gi} depend only on $\beta$.
	\end{lemma}

	\begin{proof}
	For the sake of clarity, we divide the proof into two steps.
	\setcounter{step}{0}
		\begin{step}[Proof of~\eqref{eq:ineq-Qi}]
		Testing Equation~\eqref{eq:M} against $\M - \M_\pm$ and integrating
		the result over $\Upsilon_\pm$, we obtain
		\[
		\begin{split}
			-\eps \int_{\partial \Upsilon_\pm} &(\M - \M_\pm) \cdot \partial_\nnu \M \,{\d}\sigma
			+ \eps \int_{\Upsilon_\pm} \nabla\M \cdot \nabla(\M - \M_\pm) \,{\d}x \\
			&+ \frac{1}{\eps} \int_{\Upsilon_\pm} \nabla_{\M} V(\,\cdot\,,\,\M) \cdot(\M-\M_\pm) \,{\d}x = 0.
		\end{split}
		\]
		As a consequence (cf.~\cite[Lemma~4.3]{Bethuel-AC-Acta}),
		\begin{equation}\label{eq:compu-Qi}
		\begin{split}
			\mathfrak{Q}_\pm &= \eps \int_{\Upsilon_\pm} \nabla\M\cdot\nabla \M_\pm \,{\d}x
			+ \eps \int_{\partial \Upsilon_\pm} (\M - \M_\pm) \cdot \partial_\nnu \M \,{\d}\sigma\\
			&=\eps \int_{\Upsilon_\pm} \nabla\M\cdot\nabla \M_\pm \,{\d}x
			+ \eps \int_{\partial \Upsilon_\pm} (\M - \M_\pm) \cdot \partial_\nnu \M_\pm
			+ \frac{\eps}{2} \int_{\partial \Upsilon_\pm} \partial_\nnu \abs{\M-\M_\pm}^2 \,{\d}\sigma.
		\end{split}
		\end{equation}
		As in \cite[Lemma~4.3 and Lemma~4.5]{Bethuel-AC-Acta}, we observe that
		$\partial \Upsilon_\pm = \Gamma_\pm \cup \Pi_\pm$ and that on $\Pi_\pm$
		(which is a subset of $\partial B$, cf.~\eqref{eq:Ups-Gamma-Pi})
		the normal derivative equals the radial
		derivative, so that
		\begin{equation}\label{eq:compu-norm-der-abs-M-Ni}
		\begin{split}
			\partial_\nnu \abs{\M-\M_\pm}^2 &= 2 \abs{\M-\M_\pm} \partial_\nnu \abs{\M-\M_\pm} \\
			&\lesssim \sqrt{\ell(\Q,\,\M) - \ell(\Q,\,\M_\pm)} \partial_\nnu \abs{\M-\M_\pm} \\
			&\lesssim \sqrt{V(\,\cdot\,,\M)} \partial_\nnu \abs{\M-\M_\pm} \\
			&\lesssim \eps \abs{\nabla \M - \nabla \M_\pm}^2
			+ \frac{1}{\eps}V(\,\cdot\,,\,\M) \\
			&\lesssim \eps \abs{\nabla \M}^2 + \eps \abs{\nabla \M_\pm}^2 +
			\frac{1}{\eps} V(\,\cdot\,,\M) \quad \mbox{on } \Pi_\pm,
		\end{split}
		\end{equation}
		Thus,
		\begin{equation}\label{eq:compu-norm-der-abs-M-Ni-bis}
			\frac{\eps}{2} \int_{\Pi_\pm} \partial_\nnu \abs{\M-\M_\pm}^2 \,{\d}\sigma
			\lesssim \eps E_\eps(\M; \, \Pi_\pm) + \eps^2 \int_{\Pi_\pm} \abs{\nabla \M_\pm}^2 \,{\d}\sigma,
		\end{equation}
		where the implicit constant depends only on $\beta$,
		and plugging~\eqref{eq:compu-norm-der-abs-M-Ni-bis} into~\eqref{eq:compu-Qi}
		we obtain
		\[
		\begin{split}
			\mathfrak{Q}_\pm \lesssim
			& \eps \int_{\Upsilon_\pm} \nabla\M\cdot\nabla \M_\pm \,{\d}x
			+ \eps \int_{\partial \Upsilon_\pm} (\M - \M_\pm) \cdot \partial_\nnu \M_\pm \,{\d}\sigma \\
			&+ \eps \kappa \int_{\partial \Upsilon_\pm} \partial_\nnu \abs{\M-\M_\pm} \,{\d}\sigma
			+ \eps E_\eps(\M;\,\Pi_\pm) + \eps^2 \int_{\Pi_\pm} \abs{\nabla \M_\pm}^2 \,{\d}\sigma.
		\end{split}
		\]
		Then,~\eqref{eq:ineq-Qi} follows by applying Young's inequality to
		the term $\eps \int_{\Upsilon_\pm} \nabla\M\cdot\nabla \M_\pm\,{\d}\sigma$.
		\end{step}

		\begin{step}[Proof of~\eqref{eq:ineq-Gi}]
		Once~\eqref{eq:ineq-Qi} has been obtained,
		inequality~\eqref{eq:ineq-Gi} follows easily. Indeed, first we
		separate the contributions from $\Gamma_\pm$ and $\Pi_\pm$ in
		$\int_{\partial \Upsilon_\pm} (\M - \M_\pm)\cdot \partial_\nnu \M_\pm \,{\d}\sigma$,
		and we repeat the same computations as in~\eqref{eq:compu-norm-der-abs-M-Ni},
		\eqref{eq:compu-norm-der-abs-M-Ni-bis}. Then, we observe that
		\[
			(\M - \M_\pm) \cdot \partial_\nnu \M_\pm \lesssim \eps
			\abs{\nabla \M_\pm}^2 +
			\frac{1}{\eps}V(\,\cdot\,,\M) \quad \mbox{on } \Pi_\pm
		\]
		We conclude by recalling 
		that, by~\eqref{minimisers_l}, $\abs{\nabla \M_\pm} \lesssim \abs{\nabla \Q}$
		pointwise on $\overline{B}$, where the implicit constant depends only on
		$\beta$, and applying~\eqref{eq:ineq-E-Q}.
		\qedhere
		\end{step}
	\end{proof}

	Exactly as in \cite[Lemma~4.7]{Bethuel-AC-Acta}, the coarea formula and a
	standard averaging argument yield the following intermediate lemma.
	\begin{lemma}\label{lemma:coarea+average}
		Let $\varrho_\eps \in \left[\frac{1}{2},\frac{3}{4}\right]$.
		There exists some number
		$\widetilde{\delta}_\eps \in
		\left[ \delta_\beta, 2 \delta_\beta \right]$
		such that
		\begin{equation}\label{eq:coarea+average}
		\begin{split}
			\eps
			\int_{\Gamma_{\eps,\pm}(\varrho_\eps,\,\widetilde{\delta}_\eps)}
				\partial_\nnu \abs{\M - \M_\pm} \,{\d}\sigma & \lesssim
			\eps
\int_{\Gamma_{\eps,\pm}(\varrho_\eps,\, \widetilde{\delta}_\eps)}
				\left(\abs{\nabla \M} + \abs{\nabla \M_\pm} \right) \,{\d}\sigma \\
				& \lesssim E_{\eps}(\M; \Theta(\M,\,\varrho_\eps))
				+\eps \int_{B_{\varrho_\eps}} \abs{\nabla \M_\pm}^2 \,{\d}x \\
				&\lesssim \int_{\Theta(\M,\,\varrho_\eps)}
				\frac{1}{\eps} V(x,\,\M)\,{\d}x
				+\eps \int_{B_{\varrho_\eps}} \abs{\nabla \M_\pm}^2 \,{\d}x
		\end{split}
		\end{equation}
		where the set
		$\Theta(\M,\,\varrho_\eps) \subset B_{\varrho_\eps}$
		is defined in~\eqref{eq:def-Theta} and the implicit constant on the
		right-hand side depends only on
		the coupling parameter $\beta$ and on $\Omega$.
	\end{lemma}

	With Lemma~\ref{lemma:ineq-E-Q}, Lemma~\ref{lemma:ineq-Qi+ineq-Gi}, and
	Lemma~\ref{lemma:coarea+average} at hand, we can prove both
	Proposition~\ref{prop:analogue-of-Bethuel-4.1}, providing the analogue
	of \cite[Proposition~4.1]{Bethuel-AC-Acta}, and
	Proposition~\ref{prop:analogue-of-Bethuel-4.2}.
	Recall once again that, by~\eqref{max-QM},
	whenever $(\Q,\,\M) = (\Q_\eps,\,\M_\eps)$ is a solution
	to \eqref{EL-Q},~\eqref{EL-M} in $\Omega$, subject to
	either~\eqref{bc}--\eqref{hp:bc} or~\eqref{bcbis}--\eqref{hp:bcbis},
	the uniform bound
	\[
		\norm{\M}_{L^\infty(\Omega)} \leq {\rm M}_\beta,
	\]
	holds, for some positive constant ${\rm M}_\beta$ depending only on $\beta$.
	\begin{prop}\label{prop:analogue-of-Bethuel-4.1}
		Let $\varrho_\eps \in \left[\frac{1}{2}, \frac{3}{4}\right]$. There
		exists a constant $K_\Upsilon({\rm M}_\beta)$ such that the inequality
		\begin{equation}\label{eq:ineq-G-Upsilon}
		\begin{split}
			E_\eps\left(\M; \Upsilon\left(\varrho_\eps,\delta_\beta\right)\right)
			&\leq K_\Upsilon({\rm M}_\beta)
			\left[ \int_{B_{\varrho_\eps}} \left( \eps \abs{\nabla \Q}^2 + \frac{1}{\eps} V(x,\,\M) \right) \,{\d}x \right. \\
			&\quad \left.  + \eps^2 \int_{\partial B_{\varrho_\eps}} \abs{\nabla \Q}^2 \,{\d}\sigma + \eps E_\eps\left(\M; \partial B_{\varrho_\eps}\right) \right]
		\end{split}
		\end{equation}
		holds.
	\end{prop}

\begin{proof}
	Since the sets $\Pi_+$, $\Pi_-$ are disjoint and their union is contained in $\partial B$
	(see~\eqref{eq:Ups-Gamma-Pi}), we have
	\[
		\eps^2 \int_{\Pi_\pm} \abs{\nabla \Q}^2 \,{\d}\sigma + \eps E_\eps(\M;\,\Pi_\pm)
		\leq \eps^2 \int_{\partial B} \abs{\nabla \Q}^2 \,{\d}\sigma + \eps E_\eps(\M; \,\partial B).
	\]
	Therefore, choosing $\kappa = \widetilde{\delta}_\eps$ (where
	$\widetilde{\delta}_\eps$ is the number provided by Lemma~\ref{lemma:coarea+average}),
	taking the sum of the two inequalities~\eqref{eq:ineq-Gi},
	and combining the result with~\eqref{eq:coarea+average}, the
	conclusion follows from~\eqref{eq:control-Theta} exactly as
	explained in \cite[Proposition~4.1]{Bethuel-AC-Acta}.
\end{proof}

Note that there are no assumption of closeness to the wells
in the statement of Proposition~\ref{prop:analogue-of-Bethuel-4.1},
but at the same time, it provides an estimate on a \emph{fixed} level set.
As in \cite{Bethuel-AC-Acta},
Proposition~\ref{prop:analogue-of-Bethuel-4.1} will not be used in the
proof of Proposition~\ref{prop:analogue-of-Bethuel-4.2} but rather in
the proof of Proposition~\ref{prop:analogue-of-Bethuel-4.3}, which will
be another crucial tool in
Section~\ref{sect:nustar}.

We have finally at disposal all the ingredients to
proceed to the proof of Proposition~\ref{prop:analogue-of-Bethuel-4.2}.

\begin{proof}[{Proof of Proposition~\ref{prop:analogue-of-Bethuel-4.2}}]
	The estimate~\eqref{eq:ineq-G-Upsilon} in
	Proposition~\ref{prop:analogue-of-Bethuel-4.1} is apparently close
	to~\eqref{eq:est-close-to-Sigma} but it is on a \emph{fixed} level set.
	Exactly as in \cite{Bethuel-AC-Acta},
	the crux of the proof consists in allowing the level set to vary,
	exploiting the assumption of closeness to the wells,
	i.e.,~\eqref{eq:assumption-Bethuel}.

	Fix any $\kappa \in \left(0,\,\delta_\beta \right)$, and suppose
	that~\eqref{eq:assumption-Bethuel} holds.
	By continuity of $\M$, the condition~\eqref{eq:assumption-Bethuel} can
	be satisfied by \emph{either} $\M_+$ \emph{or} $\M_-$.
	Set, for brevity,
	$\N := \M_+$ or $\N := \M_-$, according to which
	of the two options in~\eqref{eq:assumption-Bethuel} is realised,
	and denote
	\[
		\mathcal{C} := \left\{ x \in B :\, \kappa \leq \abs{ \M(x) - \N(x)} \leq 	\widetilde{\delta}_\eps  \right\},
	\]
	where $\widetilde{\delta}_\eps$ is the number provided by
	Lemma~\ref{lemma:coarea+average}.

	Testing~\eqref{eq:M} against $\frac{\M - \N}{\abs{\M - \N}}$ and
	integrating the result over $\mathcal{C}$, we obtain
	\begin{equation}\label{eq:test-eq-compu1}
	\begin{split}
		-\eps &\int_{\partial \mathcal{C}} \frac{\M - \N}{\abs{\M - \N}} \cdot
		\partial_\nnu \M \,{\d}\sigma + \eps \int_{\mathcal{C}} \nabla \M \cdot \nabla \left( \frac{\M - \N}{\abs{\M - \N}}  \right)\,{\d}x \\
		&+\frac{1}{\eps}\int_{\mathcal{C}} \nabla_\M V(\,\cdot\,,\M) \cdot
		\left( \frac{\M - \N}{\abs{\M - \N}}\right) \,{\d}x = 0,
	\end{split}
	\end{equation}
	where $\nnu$ denotes the exterior normal to $\partial \mathcal{C}$.
	Note that, by calculus and recalling~\eqref{eq:nablaV=nablaL},
	\begin{equation}\label{eq:calculus}
	\begin{split}
		\nabla_\M V(\,\cdot\,,\M) \cdot \left(\frac{\M - \N}{\abs{\M - \N}}\right) &= \nabla_\M \left( \ell(\Q,\,\M) - \ell(\Q,\,\N) \right) \cdot \left(\frac{\M - \N}{\abs{\M - \N}}\right) \\
		&\gtrsim \abs{\M -\N} \geq 0.
	\end{split}
	\end{equation}
	Clearly,
	\begin{gather*}
		\partial_\nnu \M \cdot \frac{\M - \N}{\abs{\M - \N}}  =
		\partial_\nnu \abs{\M - \N} + \partial_\nnu \N \cdot \frac{\M-\N}{\abs{\M - \N}}, \\
		\begin{split}
		\nabla \M \cdot \nabla \left( \frac{\M - \N}{\abs{\M - \N}} \right) &=
		\nabla(\M - \N) \cdot \nabla\left( \frac{\M - \N}{\abs{\M - \N}} \right) +
		\nabla\N \cdot \nabla\left( \frac{\M - \N}{\abs{\M - \N}} \right) \\
		&= \abs{\M - \N} \abs{\nabla\left( \frac{\M - \N}{\abs{\M - \N}}\right)}^2  + \nabla \N \cdot \nabla \left( \frac{\M - \N}{\abs{\M - \N}} \right).
		\end{split}
	\end{gather*}
	Assumption~\eqref{eq:assumption-Bethuel} implies
	\[
		\partial \mathcal{C} = \Gamma(\kappa) \cup
		\Gamma(\widetilde{\delta}_\eps),
	\]
	where we have set
	$\Gamma(\alpha) := \left\{ x \in B_\varrho(x_0) : \, \abs{\M(x)- \N(x) } = \alpha  \right\}$
	for $\alpha \in \R$.
	Therefore, by the identity \eqref{eq:test-eq-compu1} we obtain
	(recall that the exterior normal $\nnu$ to $\Gamma$ points in the direction that
	increases $\abs{\M - \N}$)
	\begin{equation}\label{eq:test-eq-compu2}
	\begin{split}
		\eps &\int_{\Gamma(\kappa)} \partial_\nnu \abs{\M -\N} \,{\d}\sigma
	   -\eps \int_{\Gamma\left(\widetilde{\delta}_\eps\right)}
	   \partial_\nnu \abs{\M -\N} \,{\d}\sigma\\
	   &+\frac{\eps}{\kappa} \int_{\Gamma(\kappa)} \partial_{\nnu}\N \cdot(\M - \N) \,{\d}\sigma
	   - \frac{\eps}{\widetilde{\delta}_\eps} \int_{\Gamma\left(\widetilde{\delta}_\eps\right)} \partial_{\nnu}\N \cdot(\M - \N) \,{\d}\sigma \\
	   &+\eps \int_{\mathcal{C}} \abs{\M - \N} \abs{ \nabla\left( \frac{\M - \N}{\abs{\M - \N}} \right) }^2 \,{\d}x + \eps \int_{\mathcal{C}} \nabla \N \cdot \nabla\left( \frac{\M - \N}{\abs{\M - \N}} \right) \,{\d}x \\
	   & \stackrel{\eqref{eq:test-eq-compu1}, \eqref{eq:calculus}}{\leq} 0.
	\end{split}
	\end{equation}
	On the other hand, we have
	\begin{equation}\label{eq:test-eq-compu3}
		\eps \int_{\mathcal{C}} \nabla \N \cdot \nabla\left( \frac{\M - \N}{\abs{\M - \N}} \right)\,{\d}x \leq \frac{\eps}{2}\int_{\mathcal{C}} \frac{\abs{\nabla \N}^2}{\abs{\M-\N}}\,{\d}x + \frac{\eps}{2} \int_{\mathcal{C}} \abs{\M - \N} \abs{\nabla\left( \frac{\M - \N}{\abs{\M - \N}}\right)}^2\,{\d}x
	\end{equation}
	and therefore, from~\eqref{eq:test-eq-compu2}, \eqref{eq:test-eq-compu3},
	and~\eqref{eq:coarea+average},
	\begin{equation}\label{eq:test-eq-compu4}
	\begin{split}
		\eps &\int_{\Gamma(\kappa)} \partial_\nnu \abs{\M -\N} \,{\d}\sigma
	   +\frac{\eps}{\kappa} \int_{\Gamma\left(\kappa\right)}
	   \partial_\nnu \N \cdot (\M - \N) \,{\d}\sigma \\
	  &\stackrel{\eqref{eq:test-eq-compu2}, \eqref{eq:test-eq-compu3}}\leq \eps \int_{\Gamma\left(\widetilde{\delta}_\eps\right)} \partial_\nnu \abs{\M -\N} \,{\d}\sigma
	  + \frac{\eps}{\kappa} \int_{\Gamma\left(\widetilde{\delta}_\eps\right)}
	  \partial_\nnu \N \cdot (\M - \N) \,{\d}\sigma +
	  \frac{\eps}{2}\int_{\mathcal{C}} \frac{\abs{\nabla \N}^2}{\abs{\M - \N}}\,{\d}x \\
	  &\quad \stackrel{\eqref{eq:coarea+average}}{\lesssim} \int_{B_{\varrho_\eps}} \frac{1}{\eps} V(x,\,\M(x)) \,{\d}x
	  + \frac{\eps}{\kappa} \int_{B_{\varrho_\eps}} \abs{\nabla \N}^2 \,{\d}x
	  + \eps \int_{B_{\varrho_\eps}} \abs{\nabla \N}^2 \,{\d}x,
	\end{split}
	\end{equation}
	where the implicit constant on the right-hand side depends only on
	the parameter $\beta$ in the potential $V$ and $\Omega$.

	Multiplying~\eqref{eq:test-eq-compu4} by $\kappa$ and
	combining the result with~\eqref{eq:ineq-Gi}, we obtain
	\[
		E_\eps(\M; \Upsilon) \lesssim \kappa
		\int_{B_{\varrho_\eps}} \left( \eps \abs{\nabla \Q}^2 + \frac{1}{\eps} V(x,\,\M(x)) \right) \,{\d}x
		+ \eps\left( E_\eps(\M; \partial B_{\varrho_\eps}) + \eps \int_{\partial B_{\varrho_\eps}} \abs{\nabla \Q}^2\,{\d}\sigma \right),
	\]
	where the implicit constant on the right-hand side depends only on 
	$\beta$ and on $\Omega$. This is precisely the desired
	inequality~\eqref{eq:est-close-to-Sigma}, and so we are done.
\end{proof}

With Proposition~\ref{prop:analogue-of-Bethuel-4.2} at hand, the proof of
Proposition~\ref{prop:analogue-of-Bethuel-4.2-R} follows immediately.

\begin{proof}[Proof of Proposition~\ref{prop:analogue-of-Bethuel-4.2-R}]
	Just scale back Proposition~\ref{prop:analogue-of-Bethuel-4.2} to
	$B(x_0,\,R)$ using the scaling properties in Remark~\ref{rk:scaling}.
\end{proof}

Proposition~\ref{prop:analogue-of-Bethuel-4.2-R} is concerned
with an energy estimate \emph{close} to the zero set of the potential
$V(\,\cdot\,,\,\M)$, in any ball $B = B(x_0,\,R)$ on which $\abs{\Q} \geq 1/2$.
As in \cite[Proposition~4.11]{Bethuel-AC-Acta},
using Proposition~\ref{prop:analogue-of-Bethuel-4.1},
we can obtain an estimate of the energy on the \emph{whole} ball
$B'' := B(x_0,\,R/2)$ essentially in terms of the
integral of the potential on the ball $B' := B(x_0,\,3R/4)$
plus a lower order term.

\begin{prop}\label{prop:analogue-of-Bethuel-4.3}
	Let $(\Q_\eps,\,\M_\eps)$ be a critical point
	of $\F_\eps$. There exists a constant ${\rm K}_\beta$, depending
	only on $\beta$, so that, for any ball $B := B(x_0,\,R) \subset \Omega $
	such that $\abs{\Q_\eps} \geq 1/2$ in $B$, there holds
	\begin{equation}\label{eq:est-M-whole-B''}
	\begin{split}
		E_\eps(\M_\eps;\, B'') \leq &{\rm K}_\beta \left\{ \int_{B'} \left( \frac{1}{\eps} V(x,\,\M_\eps) + \eps \abs{\nabla \Q_\eps}^2 \right) \,{\d}x \right. \\
		&\left.+ \frac{\eps}{R} E_\eps(\M_\eps;\, B \setminus B'') + \frac{\eps^2}{R} \int_{B \setminus B''} \abs{\nabla \Q_\eps}^2\,{\d}x\right\},
	\end{split}
	\end{equation}
	where $B' := B(x_0,\,3R/4)$ and $B'' := B(x_0,\,R/2)$.
\end{prop}

\begin{proof}
	By the scaling properties in Remark~\ref{rk:scaling}, we may
	assume that $B = B_1$ and then obtain~\eqref{eq:est-M-whole-B''} by
	scaling back.

	The argument involves two steps.
	The first step, corresponding to \cite[Lemma~4.12]{Bethuel-AC-Acta},
	consists in obtaining the estimate
	\begin{equation}\label{eq:ineq-Lemma-4.7-Bethuel}
	\begin{split}
		E_\eps( \M_\eps;\, B_{\varrho_\eps} ) \leq
		&{\rm C}_{\beta,\,\Omega} \left\{ \int_{B_{\varrho_\eps}} \left( \frac{1}{\eps} V(x,\,\M_\eps) + \eps \abs{\nabla \Q_\eps}^2 \right) \,{\d}x \right. \\
		&\left.+ \frac{\eps}{4} E_\eps(\M_\eps;\,\partial B_{\varrho_\eps}) + \frac{\eps^2}{4} \int_{\partial B_{\varrho_\eps}} \abs{\nabla \Q_\eps}^2 \,{\d}\sigma \right\},
	\end{split}
	\end{equation}
	where $\varrho \in \left[\frac{1}{2}, \frac{3}{4}\right]$ and
	${\rm C}_{\beta,\,\Omega}$ is a constant depending only on $\beta$
	and on $\Omega$.
	Similarly to as in \cite[Lemma~4.12]{Bethuel-AC-Acta}, this follows from the
	decomposition
	\[
		B_{\varrho_\eps} = \Theta\left(\M_\eps,\,\varrho_\eps\right)
		\cup \Upsilon_\eps\left( \varrho_\eps,\, \delta_\beta\right)
	\]
	by applying Remark~\ref{rk:control-Theta}
	and Proposition~\ref{prop:analogue-of-Bethuel-4.1}.
	Then, \eqref{eq:est-M-whole-B''} follows by combining the above inequality
	with a usual averaging argument,
	ensuring the existence of a distinguished radius
	$\sigma_\eps \in \left[ \varrho_\eps, 3/4 \right]$ such that
	\[
		E_\eps\left(\M_\eps; \, \partial B_{\sigma_\eps} \right)
		\leq 8 E_\eps\left( \M_\eps;\, B_{3/4} \setminus B_{1/2} \right),
	\]
	and then scaling back, according to~\eqref{eq:scaling-inv-E}.
\end{proof}


\subsection{Energy decreasing and clearing-out property for $\M_\eps$}
\label{sec:clout-M}
The goal of this section is to extend to our context
the energy decreasing property and the clearing-out
property proved, for the pure Allen-Cahn system,
in \cite[Proposition~1.12 and Proposition~1.13]{Bethuel-AC-Acta},
respectively.
The energy decreasing estimate in \cite[Proposition~1.12]{Bethuel-AC-Acta},
here replaced by Proposition~\ref{prop:energy-decreasing} below,
is the starting point of a delicate iterative
scheme which leads to an extremely
rapid decreasing of the energy in small balls whenever the
(rescaled) energy on the boundary is below a certain fixed threshold.
In turn, such a quantitative decay is
the key ingredient in the proof of the clearing-out property for critical points of the
pure Allen-Cahn system --- see \cite[proof of Theorem~1.11]{Bethuel-AC-Acta}.

In our context,
the coupling term in the potential
is responsible for the presence of perturbation terms and
this entails several modifications to Bethuel's arguments.
We handle them by exploiting the 
quantitative decay and convergence results for the $\Q$-component 
already proved in the companion paper~\cite{CDS1}. Relying on them, 
we can still follow the same path as in
\cite[Section~6]{Bethuel-AC-Acta} to obtain, in the end,
Theorem~\ref{thm:clearing-out-M}, which provides the counterpart
of \cite[Proposition~1.13]{Bethuel-AC-Acta}.

\begin{notation}
	In this section, we consider a \emph{sequence}
	$\{(\Q_\eps,\,\M_\eps)\}$ of critical
	points of $\F_\eps$
	satisfying the boundary conditions~\eqref{bc}--\eqref{hp:bc}
	or~\eqref{bcbis}--\eqref{hp:bcbis} as well as the
	assumption~\eqref{hp:potential_bound}. 
	\textbf{We assume throughout that
	such a subsequence has been fixed}, although we will keep
	extracting subsequences from it whenever necessary.
	To keep the notation simple, we
	will not emphasise the dependence of the various constants that will
	appear in the following course on this fixed subsequence.

	As usual, we shall often set
	$\rho_\eps := \abs{\Q_\eps}$.	
	In the following, $\delta_\beta$ denotes the positive number, depending
	only on $\beta$, given by Lemma~\ref{lemma:V-quadratic}.
	Finally, we recall from Proposition~\ref{prop:Eeps-bounded}
	that we have the uniform energy bound
	\[
		E_\eps(\M_\eps;\,\Omega) \leq \mathcal{E}_0,
	\]
	where $\mathcal{E}_0$ depends only on $\beta$ and the boundary data.
\end{notation}

\paragraph{Convergence of $\Q_\eps$ and decay of the Ginzburg-Landau energy of $\Q_\eps$ away from $\spt\mu_\star$}

Let $K \subset \Omega \setminus \spt\mu_\star$ be any compact set. 
Then, for any given $\alpha \in (1,\,2)$, there exists a positive number 
$C_\alpha(K)$, depending only $\alpha$ and $K$, such that, for any ball 
$B_R = B(x_0,\,R) \subseteq K$, we have (see~\cite[Proposition~2.6 and Remark~2.6]{CDS1}) the quantitative decay 
\begin{equation}\label{eq:GL-decay}
	\int_{B_R} \left( \frac{1}{2} \abs{\nabla \Q_\eps}^2 + \frac{1}{4\eps^2}\left( \abs{\Q_\eps}^2-1 \right)^2 \right)\,{\d}x \leq C_\alpha(K) R^\alpha.
\end{equation}

\begin{remark}\label{rk:eta-beta-K}
	Given a compact set $K \subseteq \Omega \setminus \spt \mu_\star$,
	by~\eqref{eq:unif-conv-mod-Qeps} we can find a number
	$\eps_{1/2} > 0$ (depending only on $K$) 
	so that, for any ball $B(x_0,\,R) \subset K$ and any
	$\eps$ with $0 < \eps \leq \eps_{1/2}R$, there holds
	$\abs{\Q_\eps} \geq 1/2$ on $B(x_0,\,R)$.
	The number
	$\eps_{1/2}$ depends only on the chosen sequence of
	critical points $\{(\Q_\eps,\,\M_\eps)\}$, on the parameter $\beta$,
	and on $K$. On the other hand, by~\eqref{eq:Q-M-s*}, we have
	$\abs{\Q_\eps} \leq s_*(\beta,\,\eps) = 1 + \eps \kappa_\star + \o_{\eps \to 0}(\eps)$,
	whence $\abs{\Q_\eps}\leq 2$ for any $0 < \eps \leq \eps_\beta R$, where
	$\eps_\beta$ depends only on $\beta$.
	Therefore, choosing
	\begin{equation}\label{eq:eps-K-beta}
		\eps_{K,\beta} := \min\{ \eps_{1/2}, \eps_\beta \},
	\end{equation}
	we have that $\eps_{K,\beta}$ depends only on the sequence of
	critical points, $K$,
	and $\beta$, and that
	\begin{equation}\label{eq:bound-Q-K}
		\frac{1}{2} \leq \abs{\Q_\eps(x)} \leq 2,
		\qquad \forall x \in B(x_0,\,R),
	\end{equation}
	for all balls $B(x_0,\,R) \subset K$ and
	for all $\eps$ with $0 < \eps \leq \eps_{K,\beta} R$.
\end{remark}

\paragraph{Clearing-out theorem for $\M_\eps$}
The main result of this section is the following theorem, providing
an analogue of \cite[Proposition~1.13]{Bethuel-AC-Acta} in our context.

\begin{theorem}\label{thm:clearing-out-M}
	For any compact set $K \subseteq \Omega \setminus \spt \mu_\star$
	and any $\alpha \in (1,\,2)$, there exists positive constants
	$\eta_{\star, K}^{(\alpha)}$, $\eps_{\star, K}^{(\alpha)}$,
	$R_{\star, K}^{(\alpha)}$, $C_{\rm well}$, $C_{\rm nrg}$, depending
	only on $K$, $\alpha$, $\beta$, and the energy bound $\mathcal{E}_0$,
	such that the following holds:
	if $B(x_0,\,R) \subset K$ has radius $R \leq R_{\star, K}^{(\alpha)}$,
	if $\eps \in \left(0,\,\eps_{\star,K}^{(\alpha)} R\right]$, and if
	\begin{equation}\label{eq:hp-co-M}
		E_\eps(\M_\eps;\,B(x_0,\,R)) \leq 4 \eta_{\star, K}^{(\alpha)} R,
	\end{equation}
	then
	\begin{equation}\label{eq:close-to-wells}
		\dist(\M_\eps,\,\Sigma(\Q_\eps)) \leq
		C_{\rm well}\left( \frac{E_\eps(\M_\eps;\,B(x_0,\,R))}{R} \right)^{1/6} \leq \delta_\beta
	\end{equation}
	on $B(x_0,\,R/2)$. Moreover,
	\begin{equation}\label{eq:energy-decay-M}
		E_\eps( \M_\eps;\, B(x_0,\,R/2)) \leq
		C_{\rm nrg} \eps \left( \frac{E_\eps(\M_\eps;\,B(x_0,\,R))}{R} + C_\alpha R^\alpha \right).
	\end{equation}
\end{theorem}

We start by adapting the `very weak clearing-out property' contained in
\cite[Proposition~6.1]{Bethuel-AC-Acta} to our needs. The proof
of Proposition~\ref{prop:weak-co} below only relies on the clearing-out
property for the $\Q_\eps$-components, the \emph{a priori} bounds
entailed by the maximum principle, and the quadratic structure of
the potential $V$ near the 
set
$\Sigma(\Q_\eps)$ defined in~\eqref{eq:Sigma-Q}.
\begin{prop}\label{prop:weak-co}
	For any compact set $K \subseteq \Omega \setminus \spt \mu_\star$
	there exists a positive constant $\eps(K,\,\beta)$, depending only
	on $K$ and $\beta$, and positive constants $\eta_\beta$
	and $C_{\rm weak} = C_{\rm weak}(\beta)$, depending only on
	$\beta$ and $\Omega$, such that
	if $B(x_0,\,R) \subset K$
	is any ball, if $0 < \eps \leq \eps(K,\,\beta) R$, and if
	\begin{equation}\label{eq:hp-weak-co}
		E_\eps(\M_\eps;\,B(x_0,\,R)) \leq \eta_\beta \eps,
	\end{equation}
	then there holds
	\begin{equation}\label{eq:weak-co}
		\dist(\M_\eps(x),\,\Sigma(\Q_\eps(x))) \leq
		C_{\rm weak} \left( \frac{E_\eps(\M_\eps;\,B(x_0,R))}{\eps}\right)^{1/6}
		\leq \delta_\beta
	\end{equation}
	for any $x \in B(x_0,\,{7R/8})$.
\end{prop}

\begin{proof}
	For convenience, we denote $B_r := B(x_0,\,r)$ the ball
	of radius $r$ and centre $x_0$.

	First, we choose $\eps(K,\,\beta) = \eps_{K,\beta}$, where $\eps_{K,\,\beta}$
	is the constant (depending only on $K$ and $\beta$) given by~\eqref{eq:eps-K-beta}.
	This ensures that~\eqref{eq:bound-Q-K} holds, so that the set
	$\Sigma(\Q_\eps)$ is well-defined, for any $\eps$ with
	$0 < \eps \leq \eps(\beta,K) R$; in turn, this implies that
	Lemma~\ref{lemma:V-quadratic}
	and Lemma~\ref{lemma:V-quadratic-bis} hold. For ease of notation, let us
	temporarily denote $\epsilon := \eps(K,\,\beta)$.

	Assume that~\eqref{eq:hp-weak-co} holds, for some constant
	$\eta_\beta > 0$ to be determined later.
	From this assumption and recalling~\eqref{eq:nablaV=nablaL}
	as well as the uniform
	bounds~\eqref{max-QM},~\eqref{max-gradients}, we see that,
	for any $\eps > 0$,
	\begin{align}
		& \abs{\nabla_\M V(x,\,\M_\eps(x))} \leq \frac{C_\beta}{\eps}
		\qquad \mbox{for any } x \in B_R,\label{eq:grad-bound-V}\\
		& \int_{B_R} V(x,\,\M_\eps(x))\,{\d}x
		\leq \eta_\beta \eps^2,
	\end{align}
	where $C_\beta$ is a constant depending only on $\beta$ and $\Omega$.
	(In fact,~\eqref{eq:grad-bound-V} holds globally on $\Omega$.)
	Set
	\begin{equation}\label{eq:gamma}
		\gamma := \left( \frac{16 C_\beta^2}{\pi} \frac{E_\eps(\M_\eps;\,B_R)}{\eps} \right)^{1/3},
	\end{equation}
	and assume
	\[
		\eta_\beta = \frac{\pi \gamma_\beta^3}{16 C_\beta^2},
	\]
	where $\gamma_\beta := \frac{\sqrt{2}\beta}{2}\delta_\beta^2$
	is the same number as in Lemma~\ref{lemma:V-quadratic}.
	Then, by definition of $\gamma$ and the assumption
	$E_\eps(\M_\eps;\,B(x_0,\,R))< \eta_\beta \eps$,
	it follows that $\gamma \leq \gamma_\beta$. Next, we claim
	that
	\begin{equation}\label{eq:claim-V}
		V(x,\,\M_\eps(x)) \leq \gamma \leq \gamma_\beta, \qquad
		\mbox{for all } x \in B(x_0,\,7R/8).
	\end{equation}
	Indeed, assume, for the sake of a contradiction, that there
	exists some $x_1 \in B(x_0,\,7R/8)$ such that
	$V(x_1,\,\M_\eps(x_1)) > \gamma$. Then, by the gradient
	bound~\eqref{eq:grad-bound-V}, we have
	\begin{equation}\label{eq:contradiction-V}
		V(x,\,\M_\eps(x)) \geq \frac{\gamma}{2}, \qquad
		\mbox{for all } \in B\left(x_1, \frac{\gamma \epsilon}{2 C_\beta}R \right).
	\end{equation}
	On the other hand, since $x_1 \in B(x_0,\,7R/8)$,
	we can shrink $\epsilon$
	dependingly only on $\beta$ in such a way that there holds
	\[
		B\left( x_1,\,\frac{\gamma \epsilon}{2C_\beta} R \right) \subset
		B\left( x_1,\,\frac{\gamma_\beta \epsilon}{2C_\beta} R \right) \subset
		B(x_0,\,R).
	\]
	Integrating~\eqref{eq:contradiction-V} over
	$B\left( x_1,\,\frac{\gamma \epsilon}{2C_\beta} R \right)$ and
	recalling the definition~\eqref{eq:gamma} of $\gamma$ readily
	leads to a contradiction, establishing the claim~\eqref{eq:claim-V}.

	With~\eqref{eq:claim-V} at hand, we may employ
	Lemma~\ref{lemma:V-quadratic-bis} to yield both
	\[
	\dist(\M_\eps(x),\,\Sigma(\Q_\eps(x))) \leq \delta_\beta
	\]
	and (recalling that $\lambda_- = \sqrt{2}\beta$ ---
	see~\eqref{eq:lambda-pm})
	\[
		\dist(\M_\eps(x),\,\Sigma(\Q_\eps(x))) \leq
		\sqrt{4 \left(\sqrt{2}\beta\right)^{-1} V(x,\,\M_\eps(x))},
	\]
	for every $x \in B_{7R/8}$ and any $\eps$ with $0 < \eps \leq \epsilon R$.
	It follows that, for every $x \in B_{7R/8}$
	and any $\eps$ with $0 < \eps \leq \epsilon R$,
	\[
		\dist(\M_\eps(x),\,\Sigma(\Q_\eps(x)) \leq
		\inf\left\{ \delta_\beta,\,\sqrt{4\left(\sqrt{2}\beta \gamma\right)^{-1}} \right\}
		\leq \inf\left\{ \delta_\beta, C_{\rm weak}\left( \frac{E_\eps(\M_\eps,\,B_R)}{\eps}\right)^{1/6} \right\},
	\]
	where the constant $C_{\rm weak}$ is defined as
	\[
		C_{\rm weak} := \left( \frac{512 C_\beta^2}{\sqrt{2}\beta^3 \pi}  \right)^{1/6}
	\]
	and depends only on $\beta$.
	We choose
	\[
		\eta_\beta
		= \min\left\{ \frac{\pi \gamma_\beta^3}{16 C_\beta^2}, \left( \frac{\delta_\beta}{2 C_{\rm weak}} \right)^6\right\},
	\]
	and we conclude the proof by
	observing that $\eta_\beta$ depends only on $\beta$
	and $\Omega$.
\end{proof}

\begin{prop}\label{prop:energy-decreasing}
	There exists a constant $C_0$, depending only on $\beta$
	and the energy bound $\mathcal{E}_0$ given by Proposition~\ref{prop:Eeps-bounded},
	and, for any compact set~$K\subseteq\Omega\setminus \spt\mu_\star$
	and any~$\alpha\in (1, \, 2)$, a number
	$\eps_0 = \eps_0(K, \, \alpha)$ 
	such that, for any ball~$B = B(x_0,\,R) \subset K$ and any~$\eps$
	with~$0 < \eps \leq \eps_0 R$, there holds
	\begin{equation}\label{eq:energy-decreasing}
	\begin{split}
		E_\eps\left(\M_\eps;\,B(x_0,\,R/2)\right)
		\leq C_0\left(\frac{E_\eps(\M_\eps; \, B)^{3/2}}{\sqrt{R}}
		+ \frac{\eps \, E_\eps(\M_\eps; \, B)}{R}
		+ C_\alpha(K) R^\alpha \right) \! ,
	\end{split}
	\end{equation}
	where $C_\alpha(K)$ is the constant, depending only on $\alpha$ and $K$,
	appearing in~\eqref{eq:GL-decay}.
\end{prop}

The proof of Proposition~\ref{prop:energy-decreasing} follows
a pattern similar to that of \cite[Proposition~1.12]{Bethuel-AC-Acta} and relies
on the energy estimates of the previous subsection.
The overlap with the arguments leading to \cite[Proposition~1.12]{Bethuel-AC-Acta}
is considerable, therefore we sketch the
main points and address the reader to \cite{Bethuel-AC-Acta}
when the details are exactly the same.

First of all, we need an analogue of \cite[Lemma~2.6]{Bethuel-AC-Acta}.

\begin{lemma}\label{lemma:Bethuel-2.4}
	Let $(\Q,\,\M) \in C^\infty\left(\Omega,\,\Sz\right) \times C^\infty(\Omega,\,\R^2)$
	and let $B = B(x_0,\,R) \subset \Omega$ be a ball such that
	$1/2 \leq \abs{\Q} \leq 2$ in $B$.
	Let $\eps \in (0,1)$ and $r \in (\eps R,\,R]$ be given. There exists a
	constant ${\rm C}_{\rm unf} > 0$, depending only on the parameter
	$\beta$, such that, for any given pair $(\Q,\,\M)$ as above, there exists an element
	$\N \in \Sigma(\Q)$ such that
	\[
		\abs{\M - \N} \leq {\rm C}_{\rm unf}
		\sqrt{E_\eps(\M;\,\partial B_r)} \qquad \mbox{on } \partial B_r(x_0).
	\]
\end{lemma}

\begin{proof}
	Thanks to Lemma~\ref{lemma:V-quadratic}, Lemma~\ref{lemma:V-quadratic-bis}, and
	the assumption $1/2 \leq \abs{\Q} \leq 2$ in $B$ (which implies, on the one hand,
	that $\Sigma(\Q)$ is well-defined and, on the other hand, that it is contained
	in the ball of radius $2( 1 + 2 \sqrt{2}\beta )^{1/2}$),
	the reasoning in \cite[Lemma~2.6]{Bethuel-AC-Acta} (which is based only on the
	previous ingredients, smoothness, and some clever distinction of cases)
	carries over, word for word.
\end{proof}

\begin{remark}
	In particular, by~\eqref{eq:bound-Q-K}, given a sequence $\{(\Q_\eps,\,\M_\eps)\}$
	of critical points of $\F_\eps$ and a compact set
	$K \subset \Omega \setminus \spt\mu_\star$,
	up to extraction of a subsequence,
	Lemma~\ref{lemma:Bethuel-2.4} holds for any $B = B(x_0,\,R) \subseteq K$
	and any pair $(\Q_\eps,\,\M_\eps)$,
	for any $\eps$ with $0 < \eps \leq \eps_{K,\beta} R$,
	where $\eps_{K,\beta}$ is the number given by~\eqref{eq:eps-K-beta}.
\end{remark}

In the following, we shall denote
\begin{equation}\label{eq:M-pm-eps}
	\left(\M_{\pm}\right)_\eps =
	\pm\left( 1 + \sqrt{2}\beta \rho_\eps \right)^{1/2}\n_\eps
\end{equation}
the maps realising the minimum of the potential $\ell(\Q_\eps,\,\M_\eps)$
for a given pair $(\Q_\eps,\,\M_\eps)$.

We are now ready for proving Proposition~\ref{prop:energy-decreasing}.

\begin{proof}[{Proof of Proposition~\ref{prop:energy-decreasing}}]
Fix a compact set $K \subseteq \Omega \setminus \spt\mu_\star$.
(Of course, it suffices to consider the case in which $K$ has
non-empty interior.)
Let $\eps_0 = \eps_{1/2}$, where $\eps_{1/2} > 0$ is the number, depending
on $K$, in Remark~\ref{rk:eta-beta-K}.
Let $B = B(x_0,\,R) \subset K$ and $\eps \in (0,\,\eps_0 R]$.
Then, it follows from
Remark~\ref{rk:eta-beta-K} that $\abs{\Q_\eps} \geq 1/2$ on $B$.
In particular, this implies that $\Sigma(\Q_\eps(x))$ is well-defined
for any $x \in B$ and $\Sigma(\Q_\eps(x)) = \{(\M_+(x))_\eps,\,(\M_-(x))_\eps\}$.

By a usual averaging argument via Fubini's theorem, we can find a radius
$\varrho_\eps \in \left[\frac{3}{4}R, R \right]$ such that
\begin{align}\label{eq:averaging-prop-energy-dec}
	E_\eps\left(\M_\eps;  \,\partial B_{\varrho_\eps}\right)
	&\leq 4 \frac{E_\eps(\M_\eps; B)}{R}, \\
	\int_{\partial B_{\varrho_\eps}} \abs{\nabla \Q_\eps}^2
	&\leq \frac{4}{R} \int_{B} \abs{\nabla \Q_\eps}^2 .
\end{align}
Now, we choose
\[
	\kappa_\eps := {\rm C} \sqrt{\frac{E_\eps(\M_\eps;\,B)}{R}},
\]
and, by Lemma~\ref{lemma:Bethuel-2.4} and definition of $\kappa_\eps$,
it follows that
\begin{equation}\label{eq:close-to-well-on-bdry}
	\abs{\M_\eps(s) - \N_\eps(s)} \leq \frac{\kappa_\eps}{2},
	\qquad \mbox{ for all } s \in \partial B_{\varrho_\eps},
\end{equation}
for either $\N_\eps = \left(\M_+\right)_{\eps}$ or
$\N_\eps = \left(\M_-\right)_{\eps}$, where $(\M_\pm)_\eps$ are given by~\eqref{eq:M-pm-eps}.

\setcounter{step}{0}
\begin{step}[Improved estimates close to $\Sigma(\Q_\eps)$] 
	This step corresponds to \cite[Proposition~5.1]{Bethuel-AC-Acta}.
	Its purpose is to obtain the inequality
	\begin{equation}\label{eq:improved-est-close-to-well}
	\begin{split}
		E_\eps\left(\M_\eps;\,\Upsilon_\eps(\varrho_\eps,\,\kappa_\eps) \right)
		\leq &{\rm K}_\Upsilon\left[ \frac{1}{\sqrt{R}}\left(  E_\eps(\M_\eps;\,B) + \int_B \eps \abs{\nabla \Q_\eps}^2\,{\d}x  \right)^{3/2} \right. \\
		&\left. + \frac{\eps}{R} \left( E_\eps(\M_\eps;\,B) + \int_B \eps \abs{\nabla \Q_\eps}^2\,{\d}x \right) \right],
	\end{split}
	\end{equation}
	where the constant ${\rm K_\Upsilon}$ depends only on the energy bound
	$\mathcal{E}_0$ in~\eqref{eq:Eeps-bounded} and on $\beta$.

	Since~\eqref{eq:improved-est-close-to-well} is obtained from
	Proposition~\ref{prop:analogue-of-Bethuel-4.2},~\eqref{eq:close-to-well-on-bdry},
	and the choice of $\kappa_\eps$
	exactly by the same argument as in \cite[Proposition~5.1]{Bethuel-AC-Acta},
	we skip the details.
\end{step}

\begin{step}[Improved potential estimates]
	In analogy with \cite[Proposition~5.2]{Bethuel-AC-Acta}, we now
	obtain an `improved potential estimate'.
	We claim that, for $\eps$ and $B$ as in the statement, for any given
	$\alpha \in (1,\,2)$, there exists a constant
	${\rm C}_1$, depending only on $\beta$ and the energy bound $\mathcal{E}_0$
	in Proposition~\ref{prop:Eeps-bounded}, such that
	\begin{equation}\label{eq:improved-pot-est-AC}
		\frac{1}{\eps} \int_{B_{5R/8}} V(x,\,\M_\eps(x)) \,{\d}x \leq
		{\rm C}_1 \left[ \frac{1}{\sqrt{R}} E_\eps(\M_\eps;\,B)^{3/2} +\frac{\eps}{R} E_\eps(\M_\eps;\,B) + C_\alpha(K) R^\alpha \right].
	\end{equation}
	To this purpose, by~\eqref{eq:close-to-well-on-bdry},
	we can apply classical averaging arguments
	(see, e.g., \cite[Lemma~2.10]{Bethuel-AC-Acta}) to deduce that, for some
	radius $\tau_\eps \in \left[ \frac{5R}{8},\, \varrho_\eps \right]$, there holds
	\[
		\tau_\eps  E_\eps\left(\M_\eps;\,\partial B_{\tau_\eps}\right)
		\leq 16 E_\eps\left(\M_\eps; \Upsilon_\eps(\varrho_\eps,\,\kappa_\eps)\right)
	\]
	Thus, by~\eqref{eq:improved-est-close-to-well},
	\begin{equation}\label{eq:improv-pot-compu1}
	\begin{split}
		\tau_\eps  E_\eps\left(\M_\eps;\,\partial B_{\tau_\eps}\right)
		\leq &16 {\rm K}_\Upsilon \left[ \frac{1}{\sqrt{R}}\left( E_\eps(\M_\eps;\,B) + \int_B \eps \abs{\nabla \Q_\eps}^2 \,{\d}x \right)^{3/2} \right. \\
		&\left.+ \frac{\eps}{R} \left( \int_B E_\eps(\M_\eps;\,B) + \int_B \eps  \abs{\nabla \Q_\eps}^2 \,{\d}x\right)\right],
	\end{split}
	\end{equation}
	On the other hand, thanks to Lemma~\ref{lemma:pohozaev}, 
	Lemma~\ref{lemma:pointwise-est-V}, and~\eqref{eq:GL-decay}, we have
	\begin{equation}\label{eq:improv-pot-compu2}
		\tau_\eps E_\eps\left(\M_\eps;\,\partial B_{\tau_\eps}\right)
		\geq \frac{1}{\eps} \int_{B_{\tau_\eps}} V(x,\,\M_\eps) \,{\d}x
		- C_\alpha(K) R^\alpha
	\end{equation}
	for every $\alpha \in (1,\,2)$. Now, we combine~\eqref{eq:improv-pot-compu2}
	and~\eqref{eq:improved-est-close-to-well}, 	and we shrink $\eps_0$, depending
	only on $K$ and $\alpha$, in such a way that
	\[
		\eps_0 (\diam K) \left( 2^{3/2} \eps^{1/2} C_\alpha(K)^{1/2} (\diam K)^{\alpha/2} + \eps_0 \right) \leq 1.
	\]
	Then,~\eqref{eq:improved-pot-est-AC}
	follows by choosing ${\rm C}_1 = \max\{ 32 {\rm K}_\Upsilon,\,2\}$ and observing
	that ${\rm C}_1$ depends only on the energy bound $\mathcal{E}_0$ in~\eqref{eq:Eeps-bounded}
	and $\beta$.
\end{step}

\begin{step}[Conclusion]
	Once that~\eqref{eq:improved-pot-est-AC} has been obtained,
	the desired inequality~\eqref{eq:energy-decreasing} follows by
	matching~\eqref{eq:improved-pot-est-AC}
	with (the scaled version of)~\eqref{eq:ineq-Lemma-4.7-Bethuel}
	and with~\eqref{eq:averaging-prop-energy-dec}, possibly shrinking
	again $\eps_0$ (still in a way depending only on $K$ and $\alpha$ only) and
	choosing appropriately the constant $C_0$, depending only on
	the energy bound $\mathcal{E}_0$ in~\eqref{eq:Eeps-bounded} and $\beta$.
	The argument goes on exactly the same way
	as in the proof of \cite[Proposition~1.12]{Bethuel-AC-Acta}
	(see \cite[Section~5.3]{Bethuel-AC-Acta}),
	to which the reader is addressed for full details.
	\qedhere
\end{step}
\end{proof}


We now show that, under a suitable smallness assumption
for the energy, \eqref{hp:clearunif} below, the
estimate~\eqref{eq:energy-decreasing} can be \emph{iterated}
to yield that $\M_\eps(x)$ stays \emph{close to $\Sigma(\Q_\eps(x))$,
uniformly with respect to $x$ and $\eps$}.
More precisely, we have the following statement, providing the
counterpart in our context of \cite[Proposition~6.5 and Corollary~6.4]{Bethuel-AC-Acta}.

\begin{prop} \label{prop:clearunif}
 For any compact set~$K\subseteq\Omega\setminus \spt\mu_\star$ 
 and any~$\delta > 0$,
 there exist positive constants~$\eps_1 = \eps_1(K, \, \delta)$,
 $R_1 = R_1(K, \, \delta)$, $\eta_1 = \eta_1(K, \, \delta)$,
 and~$C_1 = C_1(K, \, \delta)$, depending only on $K$, $\alpha$, $\beta$, and $\mathcal{E}_0$, such that the following holds:
 if a ball~$B = B(x_0,\,R)\subseteq K$ has radius~$R \leq R_1$,
 if~$\eps$ satisfies~$0 < \eps \leq \eps_1 R$ and if
 \begin{equation} \label{hp:clearunif}
  E_\eps(\M_\eps; \, B) \leq \eta_1 R,
 \end{equation}
 then
 \begin{equation}\label{eq:clearunif}
  \dist\left(\M_\eps(x), \, \Sigma(\Q_\eps(x))\right)
  \leq \delta
 \end{equation}
 for any~$x\in B(x_0,\, 3R/4)$.
\end{prop}
\begin{proof}
 Let~$x\in B(x_0,\,3R/4)$ be an arbitrary point.
 For the sake of brevity, we will use the
 notation~$m_\eps := \dist(\M_\eps(x), \, \Sigma(\Q_\eps(x)))$,
 $\bar{R} := R/4$, and
 \[
  E(r) := E_\eps(\M_\eps; \, B_r(x)),
  \qquad \Phi(r) := \frac{E(r)}{r}
 \]
 for any~$r \in (0, \, \bar{R}]$.
 Moreover, we fix an arbitrary~$\alpha\in (1, \, 2)$
 and we consider~$\eps_0=\eps_0(K, \, \alpha)$
 as given by
 Proposition~\ref{prop:energy-decreasing}, and, with
 a slight abuse of notation,
 $C_0 = C_0(K, \, \alpha,\,\beta)$
 as the maximum between
 the constant $C_0$ in the right-hand side of~\eqref{eq:energy-decreasing}
 and $C_0 C_\alpha(K)$.
 We define
 \begin{equation} \label{Lambda_alpha}
  \Lambda_\alpha := 2^\alpha C_0
   \sum_{j=0}^{+\infty} 2^{(\alpha - 2) j}
 \end{equation}
 and note that the series at the right-hand side
 converges, because~$\alpha < 2$.

 Now, we choose suitable parameters~$\eps_1$,
 $R_1$, and~$\eta_1$. First, we take
 \begin{equation} \label{eps0}
  \eps_1 := \min\left\{\eps_0, \, \frac{1}{8C_0}\right\}\!,
 \end{equation}
 so that $\eps_1 = \eps_1(K,\,\alpha,\,\beta)$.

 Due to the~$L^\infty$-gradient bound on~$\M_\eps$
 given by~\eqref{max-QM}
 and the quadratic growth of~$V$
 (cf.~\eqref{eq:V} and~\eqref{ell_eps}),
 there exists a constant~$\lambda_1$,
 depending on~$\eps_1$ (and hence on~$K$, $\alpha$, and $\beta$)
 but not on~$\eps$, such that
 \begin{equation} \label{clearunif-smallball}
  E\!\left(\frac{\eps}{\eps_1}\right)
  \geq \lambda_1 \eps \, m_\eps^2
 \end{equation}
 Next, we take~$\eta_1$ and~$R_1$ small enough that
 \begin{equation} \label{R_1}
  \left(4 \eta_1 + \Lambda_\alpha R_1^{\alpha - 1}\right)^{1/2}
   \leq \frac{1}{8C_0}, \qquad
  \left(16 \eta_1 + \Lambda_\alpha R_1^{\alpha - 1}\right)^{1/2}
   \leq \delta (\eps_1 \lambda_1)^{1/2}
 \end{equation}
 In particular, our choice of~$\eps_1$, $\eta_1$, and~$R_1$ implies that
 \begin{equation}\label{eta1}
  C_0\left(4 \eta_1 + \Lambda_\alpha R_1^{\alpha - 1}\right)^{1/2}
  + C_0 \eps_1 \leq \frac{1}{4}.
 \end{equation}
 For further reference, we also note that
 \begin{equation} \label{clearunif0}
  \begin{split}
   \Phi(\bar{R}) \leq \frac{4 E_\eps(\M_\eps; \, B)}{R}
   \leq 4\eta_1
  \end{split}
 \end{equation}
 because of the assumption~\eqref{hp:clearunif}.

 For any positive integer~$n$ such
 that~$2^n \eps \leq \eps_1\bar{R}$, we claim that
 \begin{equation} \label{clearunif-induction}
  E\!\left(\frac{\bar{R}}{2^n}\right)
  \leq \frac{E(\bar{R})}{4^n} + 2^\alpha C_0
  \sum_{j=0}^{n-1} 2^{(\alpha - 2)j}
   \left(\frac{\bar{R}}{2^n}\right)^\alpha \! .
 \end{equation}
 To prove this claim, we iterate the estimate given
 by Proposition~\ref{prop:energy-decreasing},
 along the lines of classical iteration arguments
 (see, e.g., \cite[Lemma~B.3]{Beck} or \cite[Lemma~5.13]{GiaquintaMartinazzi}).
 We will give a complete proof of~\eqref{clearunif-induction}
 below; for the time being, we assume~\eqref{clearunif-induction}
 holds and complete the proof of the proposition.
 We evaluate both sides of~\eqref{clearunif-induction}
 when~$n$ is such that~$\eps 2^n = \eps_1\bar{R}$.
 Keeping~\eqref{Lambda_alpha} and~\eqref{clearunif0}
 into account, we obtain
 \begin{equation*}
  \begin{split}
   E\!\left(\frac{\eps}{\eps_1}\right)
   \leq \frac{\eps^2 E(\bar{R})}{\eps_1^2 \bar{R}^2}
    + \Lambda_\alpha
    \left(\frac{\eps}{\eps_1}\right)^\alpha
   \leq \frac{4\eps^2 \eta_1}{\eps_1^2 \bar{R}}
    + \Lambda_\alpha
    \left(\frac{\eps}{\eps_1}\right)^\alpha
  \end{split}
 \end{equation*}
 and, since~$\eps \leq \eps_1 R = 4\eps_1\bar{R}$
 and~$R\leq R_1$ by assumption,
 \begin{equation} \label{clearunif4}
   E\!\left(\frac{\eps}{\eps_1}\right)
   \leq \left(16\eta_1 + \Lambda_\alpha R_1^{\alpha-1}\right)
    \frac{\eps}{\eps_1}
 \end{equation}
 Combining this estimate with~\eqref{clearunif-smallball},
 we obtain
 \[
  m_\eps^2
  \leq \frac{1}{\lambda_1 \eps} E\!\left(\frac{\eps}{\eps_1}\right)
  \leq \frac{1}{\lambda_1\eps_1} \left(16\eta_1
    + \Lambda_\alpha R_1^{\alpha-1}\right) \! .
 \]
 Keeping~\eqref{R_1} into account, we finally
 conclude~$m_\eps^2\leq \delta^2$, which is what
 we wanted to prove.
\end{proof}
\begin{proof}[Proof of~\eqref{clearunif-induction}]
 The proof proceeds by induction on~$n$. For~$n=1$,
 Proposition~\ref{prop:energy-decreasing} implies
 \[
  \begin{split}
   E\left(\frac{\bar{R}}{2}\right)
   \leq C_0\left(\Phi(\bar{R})^{1/2}
   + \frac{\eps}{\bar{R}} \right) E(\bar{R}) + C_0\bar{R}^\alpha
  \end{split}
 \]
 On the other hand, recalling~\eqref{eta1},
 \eqref{clearunif0}, and~\eqref{R_1}, we have
 \[
  \begin{split}
   E\left(\frac{\bar{R}}{2}\right)
   \leq C_0\left(2\eta_1^{1/2}
    + \eps_1 \right) E(\bar{R}) + C_0\bar{R}^\alpha
   \leq \frac{E(\bar{R})}{4} + C_0\bar{R}^\alpha.
  \end{split}
 \]
 Therefore, the claim~\eqref{clearunif-induction}
 holds true when~$n = 1$. Now, suppose~$n$
 is a positive integer, with~$2^n \eps \leq \eps_1\bar{R}$,
 such that~\eqref{clearunif-induction} is satisfied.
 Then, we can apply Proposition~\ref{prop:energy-decreasing}
 again:
 \begin{equation} \label{clearunif1}
  \begin{split}
   E\!\left(\frac{\bar{R}}{2^{n+1}}\right)
   \leq C_0 \left(\Phi\!\left(\frac{\bar{R}}{2^n}\right)^{1/2}
    + \frac{2^n \eps}{\bar{R}}\right)
    E\!\left(\frac{\bar{R}}{2^n}\right)
    + C_0\left(\frac{\bar{R}}{2^n}\right)^\alpha
  \end{split}
 \end{equation}
 By dividing both sides of~\eqref{clearunif-induction}
 by~$\bar{R}/2^n$, and then applying~\eqref{Lambda_alpha}
 and~\eqref{clearunif0} to further estimate the right-hand side,
 we deduce
 \begin{equation} \label{clearunif2}
  \begin{split}
   \Phi\!\left(\frac{\bar{R}}{2^n}\right)
   \leq \frac{\Phi(\bar{R})}{2^n} + 2^\alpha C_0
    \sum_{j=0}^{n-1} 2^{(\alpha - 2)j}
    \left(\frac{\bar{R}}{2^n}\right)^{\alpha - 1}
   \leq 4\eta_1 + \Lambda_\alpha R_1^{\alpha - 1}
  \end{split}
 \end{equation}
 We inject~\eqref{clearunif2} into the right-hand
 side of~\eqref{clearunif1}, use the assumption
 that~$\eps 2^n \leq \eps_1\bar{R}$ and our choice~\eqref{eta1}
 of~$\eta_1$ and~$R_1$:
 \begin{equation} \label{clearunif3}
  \begin{split}
   E\!\left(\frac{\bar{R}}{2^{n+1}}\right)
   &\leq C_0 \left(\left(4\eta_1
    + \Lambda_\alpha R_1^{\alpha - 1}\right)^{1/2}
    + \eps_1\right)
    E\!\left(\frac{\bar{R}}{2^n}\right)
    + C_0\left(\frac{\bar{R}}{2^n}\right)^\alpha \\
   &\leq \frac{1}{4} E\!\left(\frac{\bar{R}}{2^n}\right)
    + C_0\left(\frac{\bar{R}}{2^n}\right)^\alpha
  \end{split}
 \end{equation}
 Finally, from the induction assumption we deduce
 \begin{equation*} 
  \begin{split}
   E\!\left(\frac{\bar{R}}{2^{n+1}}\right)
   &\leq \frac{E(\bar{R})}{4^{n+ 1}} + 2^{\alpha - 2} C_0
    \sum_{j=0}^{n-1} 2^{(\alpha - 2)j}
    \left(\frac{\bar{R}}{2^n}\right)^\alpha
    + C_0\left(\frac{\bar{R}}{2^n}\right)^\alpha \\
   &\leq \frac{E(\bar{R})}{4^{n+ 1}} + C_0
    \sum_{j=0}^{n} 2^{(\alpha - 2)j}
    \left(\frac{\bar{R}}{2^n}\right)^\alpha
  \end{split}
 \end{equation*}
 This completes the proof of~\eqref{clearunif-induction},
 and hence of Proposition~\ref{prop:clearunif}.
\end{proof}
Next, following the line of the argument in \cite[Section~6]{Bethuel-AC-Acta},
we use Proposition~\ref{prop:clearunif} to show that, if the energy is
appropriately small in a ball, then it decays as fast as $\eps$ in
a slightly smaller ball.
Proposition~\ref{prop:analogue-of-Bethuel-6.3} below is the analogue, in our context,
of \cite[Proposition~6.8]{Bethuel-AC-Acta}.
%
\begin{prop}\label{prop:analogue-of-Bethuel-6.3}
	For any compact set $K \subseteq \Omega \setminus \spt \mu_\star$
	there exists positive a constant $C_{\rm dec}$,
	depending only on $\beta$,
	such that, for any $\alpha \in (1,\,2)$,
	if $B(x_0,\,R) \subset K$, if $\eps \in (0,\,\eps_{K,\,\beta} R]$,
	and if
		\begin{equation}\label{eq:hp-eps-decay-E}
		\dist\left( \M_\eps,\,\Sigma(\Q_\eps) \right) \leq \delta_\beta
		\qquad \mbox{in } B(x_0,\,{3R/4}),
	\end{equation}
	then
	\begin{equation}\label{eq:eps-decay-E}
		E_\eps(\M_\eps;\,B(x_0,\,5R/8)) \leq C_{\rm dec}
		\eps\left(\frac{E_\eps( \M_\eps;\,B(x_0,\,R))}{R} + C_\alpha(K) R^\alpha\right).
	\end{equation}
\end{prop}

\begin{proof}
	Let $K \subset \Omega \setminus \spt\mu_\star$ be a compact set
	and assume that $\eps \leq \eps_{K,\,\beta}$, where $\eps_{K,\,\beta}$
	is the constant defined in~\eqref{eq:eps-K-beta},
	so that $1/2 \leq \abs{\Q_\eps} \leq 2$ in $K$
	for any $\eps \leq \eps_{K,\,\beta}$. From now on,
	we argue exactly as in \cite[Proposition~6.8]{Bethuel-AC-Acta}.

	By a classical averaging argument based on Fubini's theorem, we can find
	$r_\eps \in \left[ \frac{5R}{8},\,\frac{3R}{4} \right]$ such that
	\begin{align}\label{eq:eps-decay-E-fubini}
		& E_\eps\left(\M_\eps;\,\partial B_{r_\eps} \right)
		\leq 16 \frac{E_\eps(\M_\eps;\,B_R)}{R} \\
		& \int_{\partial B_{r_\eps}} \sqrt{V(\,\cdot,\,\M_\eps)}
		\abs{\nabla \M_\eps} \,{\d}s \leq 16 \frac{E_\eps(\M_\eps;\,B_R)}{R},
	\end{align}
	where, for the sake of brevity, we have set $B_R := B(x_0,\,R)$ and
	$B_{r_\eps} := B(x_0,\,r_\eps)$.
	By assumption~\eqref{eq:hp-eps-decay-E} and recalling that
	$\abs{(\M_+)_\eps - (\M_-)_\eps} > 2$ independently of $x \in K$
	by~\eqref{eq:dist-wells}, we have
	$\abs{\M_\eps - \N_\eps} \leq \delta_\beta$
	for \emph{either} $\N_\eps = (\M_+)_\eps$ \emph{or} $\N_\eps = (\M_-)_\eps$.
	Testing~\eqref{eq:M} against $\M_\eps - \N_\eps$
	and integrating over $B_{r_\eps}$, we find
	\begin{equation}\label{eq:eps-decay-E-compu1}
	\begin{split}
		\int_{B_{r_\eps}} & \left( \eps \nabla(\M_\eps - \N_\eps)\cdot \nabla \M_\eps + \frac{1}{\eps}\nabla_\M V(x,\,\M_\eps(x)) \cdot (\M_\eps - \N_\eps)\right)\,{\d}x \\
		&= \int_{\partial B_{r_\eps}} \eps(\M_\eps - \N_\eps)\partial_{\nnu} \M_\eps\,{\d}s.
	\end{split}
	\end{equation}
	By~\eqref{eq:hp-eps-decay-E} and Lemma~\ref{lemma:V-quadratic},
	the pointwise estimate
	\begin{equation}\label{eq:eps-decay-E-compu2}
		\nabla_\M V(\,\cdot\,,\,\M_\eps) \cdot(\M_\eps - \N_\eps)
		\gtrsim V(\,\cdot\,,\,\M_\eps)
	\end{equation}
	holds on $B_{3R/4}$, where the implicit constant depends only on $\beta$.
	On the other hand, by Young's inequality and~\eqref{minimisers_l}, for
	any choice of $\alpha \in (1,\,2)$, we have
	\begin{equation}\label{eq:eps-decay-E-compu3}
	\begin{split}
		\eps \int_{B_{r_\eps}} \nabla \N_\eps \cdot \nabla \M_\eps\,{\d}x
		&\eps \leq \eps \int_{B_{r_\eps}} \left( C_\beta \abs{\nabla \Q_\eps}^2 + \frac{1}{2}\abs{\nabla \M_\eps}^2\right)\,{\d}x \\
		& \leq \eps C_\alpha C_\beta R^\alpha
		+ \eps \int_{B_{r_\eps}} \frac{1}{2}\abs{\nabla \M_\eps}^2\,{\d}x
	\end{split}
	\end{equation}
	where $C_\beta > 0$ is a suitable constant depending only on $\beta$ and
	$C_\alpha = C_\alpha(K)$
	is the constant in~\eqref{eq:GL-decay}. Finally,
	by~\eqref{eq:hp-eps-decay-E},~\eqref{eq:eps-decay-E-fubini},
	Lemma~\ref{lemma:V-quadratic}, and an application of Young's inequality,
	we get
	\begin{equation}\label{eq:eps-decay-E-compu4}
	\begin{split}
		\eps \int_{\partial B_{r_\eps}} (\M_\eps - \N_\eps)\partial_{\nnu} \M_\eps \,{\d}s &\leq \eps \int_{\partial B_{r_\eps}} \abs{\M_\eps - \N_\eps}\abs{\nabla \M_\eps} \,{\d}s \\
		&\lesssim \eps \int_{\partial B_{r_\eps}} \sqrt{V(\,\cdot\,,\M_\eps)} \abs{\nabla \M_\eps} \,{\d}s \\
		& \stackrel{\eqref{eq:eps-decay-E-fubini}}{\lesssim} \eps \left( \frac{E_\eps(\M_\eps;\,B_R)}{R}\right)
	\end{split}
	\end{equation}
	Combining~\eqref{eq:eps-decay-E-compu1}
	with~\eqref{eq:eps-decay-E-compu2},~\eqref{eq:eps-decay-E-compu3},
	and~\eqref{eq:eps-decay-E-compu4}
	yields~\eqref{eq:eps-decay-E} for an appropriate choice of $C_{\rm dec}$,
	depending only 
	on $\beta$.
	This concludes the proof.
\end{proof}

With Proposition~\ref{prop:clearunif},
Proposition~\ref{prop:analogue-of-Bethuel-6.3},
and Proposition~\ref{prop:weak-co}
at hand, it is now easy to deduce the next result,
the last piece of information still missing to carry out the
argument in \cite[Section~6]{Bethuel-AC-Acta} and
conclude the proof of Theorem~\ref{thm:clearing-out-M}.
\begin{prop}\label{prop:analogue-of-Bethuel-6.4}
	For any compact set $K \subseteq \Omega \setminus \spt \mu_\star$
	and any $\alpha \in (1,\,2)$, there exists positive constants
	$\eps_2$, $\eta_2$, $R_2$, and $C_{\rm well}$, depending only on
	$K$, $\alpha$, and $\beta$, such that, if $B(x_0,\,R) \subset K$
	has radius $R \leq R_2$, if $\eps \in (0,\,\eps_2 R]$, and if
	\begin{equation}\label{eq:hp-6.4}
		E_\eps(\M_\eps;\,B(x_0,\,R)) \leq \eta_2 R,
	\end{equation}
	then
	\begin{equation}
		\dist(\M_\eps,\,\Sigma(\Q_\eps))
		\leq C_{\rm well}\left( \frac{E_\eps(\M_\eps;\,B(x_0,\,R))}{R}\right)^{1/6}
		\leq \delta_\beta
		\qquad \mbox{on } B(x_0,\,R/2).
	\end{equation}
\end{prop}

\begin{proof}
	The proof follows very closely the path
	of \cite[Proposition~6.9]{Bethuel-AC-Acta}.

	First, for fixed $\alpha \in (1,\,2)$ and $\delta = \delta_\beta$,
	we first choose $\eps_2 = \eps_1$, $\eta_2 = \eta_1$, and
	$R_2 = R_1$ as given by Proposition~\ref{prop:clearunif} and we let
	$R \leq R_2$, so that we the conclusion of
	Proposition~\ref{prop:clearunif} holds.
	Thus, for any $\eps \leq \eps_2 R$, we have
	\[
		\dist\left(\M_\eps(x),\, \Sigma(\Q_\eps(x))\right) \leq \delta_\beta \qquad
		\mbox{for any } x \in B(x_0,\,3R/4).
	\]
	Then, we may apply Proposition~\ref{prop:analogue-of-Bethuel-6.3}
	to obtain that~\eqref{eq:eps-decay-E} holds. In turn, up to reducing
	$\eta_2$ and $R_2$ so that
	(again, this can be done dependingly only on $K$, $\alpha$,
	and $\beta$)
	\[
		C_{\rm dec} \eta_2 \leq \frac{1}{2} \min\{ \eta_1, \eta_\beta\},
		\qquad
		C_{\rm dec} C_\alpha(K) R_2^\alpha \leq \frac{1}{2} \min\{ \eta_1, \eta_\beta\},
	\]
	it follows that
	\[
		C_{\rm dec}\left( \eta_2 + C_\alpha(K) R_2^\alpha \right)
		\leq \min\{ \eta_1, \eta_\beta\}.
	\]
	Thus, we are allowed to
	employ Proposition~\ref{prop:weak-co}, and the conclusion follows.
\end{proof}

We are now ready for the proof of Theorem~\ref{thm:clearing-out-M}.
\begin{proof}[Proof of Theorem~\ref{thm:clearing-out-M}]
	Exactly as for \cite[Theorem~1.11]{Bethuel-AC-Acta} (of which the
	aforementioned \cite[Proposition~1.13]{Bethuel-AC-Acta} is just a
	scaled version), the proof is achieved immediately
	choosing
	\[
		\eps_{\star,K}^{(\alpha)} = \eps_2, \qquad R_{\star,K}^{(\alpha)} = R_2,
		\qquad \eta_{\star,K}^{(\alpha)} = \frac{1}{4}\eta_2,
	\]
	and applying Proposition~\ref{prop:analogue-of-Bethuel-6.4} to
	get~\eqref{eq:close-to-wells}, followed by an application of
	Proposition~\ref{prop:analogue-of-Bethuel-6.3} to
	get~\eqref{eq:energy-decay-M}.
\end{proof}


\section{The limiting measures $\nu_\star$ and $\zeta_\star$}
\label{sect:nustar}

In this section, we study the asymptotic behaviour of
the energy densities $\nu_\eps$ and the potential energy
densities $\zeta_\eps$ (introduced in~\eqref{eq:def-nu-eps}
and in~\eqref{eq:def-zeta-eps-intro}, respectively) as
well as
the properties of the respective limiting measures
$\nu_\star$ and $\zeta_\star$. We shall see that, exactly as in the
pure vectorial Allen-Cahn theory as developed in~\cite{Bethuel-AC-Acta},
the prominent role in the investigation of the structure
of the set $\mathfrak{S}_\star := \spt \nu_\star \setminus \spt\mu_\star$
is played by $\zeta_\star$ rather than by $\nu_\star$.
This is because $\spt\zeta_\star = \mathfrak{S}_\star$, and
the ultimate reason for this lies in the fact
that $\zeta_\star$ solves an equation, Equation~\eqref{eq:omega*-zeta*},
while an analogous property is not known for $\nu_\star$.
Deferring a deeper discussion of this crucial point to later,
we observe that such an equation is key to prove that the function
$r \mapsto \frac{1}{r}\zeta_\star(B(x_0,\,r))$ is monotone
non-decreasing for every $x_0 \in \Omega$
(while, once again, an analogous
property is not known for $\nu_\star$, in sharp contrast with the
scalar case). In turn, the monotonicity of $\frac{1}{r}\zeta_\star$ immediately
implies that $\zeta_\star$ has a well-defined density
$\mathfrak{v}_\star(x_0)$ at any point $x_0 \in \Omega$ with respect
to the measure $\lambda_\star := \H^1 \mres \mathfrak{S}_\star$,
hence in particular it is
absolutely continuous with respect to $\lambda_\star$. Furthermore, in
Proposition~\ref{prop:nu*<<lambda*} we prove that the density
$\mathfrak{e}_\star$ of $\nu_\star$ exists at $\H^1$-a.e. point of
$\mathfrak{S}_\star$ and that the inequality
$\mathfrak{e}_\star \lesssim \mathfrak{v}_\star$ holds at $\H^1$-a.e. point
of $\mathfrak{S}_\star$ for a constant depending only on $\beta$ and
$\Omega$. Moreover, from Proposition~\ref{prop:analogue-of-Bethuel-Prop-6}
and the inequality $\mathfrak{e}_\star \lesssim \mathfrak{v}_\star$, we obtain
the inequalities of measures
\[
	\zeta_\star \leq \nu_\star \quad \mbox{in } \Omega, \qquad
	\nu_\star \lesssim \zeta_\star \quad
	\mbox{in } \Omega \setminus \spt \mu_\star.
\]
In particular, the set $\spt \zeta_\star$
coincides exactly with $\mathfrak{S}_\star$.

The monotonicity of $\frac{1}{r}\zeta_\star$ will be proved
in Proposition~\ref{prop:analogue-of-Bethuel-Prop-6},
at the end of a rather long path that goes trough the same steps
as in \cite{Bethuel-AC-Acta}.
As we will see
along this Section, this is possible because we can get rid of the
perturbation terms containing $\Q_\eps$ at the $\eps$-level and
prove that they vanish in an appropriately strong sense as $\eps \to 0$,
so that the analysis at the `$\star$-level' boils down to the one
for the pure Allen-Cahn system in~\cite{Bethuel-AC-Acta}.

The key tools to this purpose are Theorem~\ref{thm:clearing-out-M} and the strong convergence
properties of $\Q_\eps$ towards $\Q_\star$ away from $\spt\mu_\star$, 
obtained in \cite{CDS1} and recalled in Theorem~\ref{mainthm:asymp},
which, combined, yield a clearing-out property
for $\nu_\star$, see item~\ref{item:cl-o} in Theorem~\ref{thm:properties-nu*}.
Such a property is crucial for our arguments.
Moreover, arguing like in \cite[Theorem~1.16]{Bethuel-AC-Acta},
we obtain in Theorem~\ref{thm:analogue-of-Bethuel-8}
another kind of clearing-out type property for $\nu_\star$, stating essentially that there are
no `islands' in $\mathfrak{S}_\star$. The argument
relies on Proposition~\ref{prop:analogue-of-Bethuel-4.3} and the Pohozaev
identity (Lemma~\ref{lemma:stressenergy}). Again, this is
possible because the contribution of the $\Q_\eps$-part is well controlled at the $\eps$-level thanks to the improved bounds from 
\cite[Proposition~2.6]{CDS1}.

Together,
the two types of clearing-out results above entail,
by arguing along the lines of \cite[Section~7]{Bethuel-AC-Acta},
rectifiability and local connectedness properties of $\mathfrak{S}_\star$.

Concerning the proof of the monotonicity of $\frac{1}{r}\zeta_\star$,
we argue again along the lines of~\cite{Bethuel-AC-Acta}.
A pivotal role will be played by
the \emph{limiting Hopf differential} $\omega_\star$, introduced in
Section~\ref{sec:hopf}. The limiting Hopf differential $\omega_\star$ is related to
$\zeta_\star$ by Proposition~\ref{prop:hopf}, specifically by~\eqref{eq:omega*-zeta*}.
Testing~\eqref{eq:omega*-zeta*} against suitable vector fields and
using the rectifiability properties of
$\mathfrak{S}_\star$ (see item~\ref{item:S_*} in Theorem~\ref{thm:properties-nu*} and the results in Section~\ref{sec:good-points}),
we can show that there exist certain
crucial relations between $\zeta_\star$ and the components of $\omega_\star$
satisfied at $\H^1$-almost every point of $\mathfrak{S}_\star$,
see Theorem~\ref{thm:analogue-of-Bethuel-Thm5}. With those relations
at hand, and using again~\eqref{eq:omega*-zeta*}, we can prove that
$\zeta_\star$ is the weight measure of a one-dimensional \emph{varifold}
$\mathbb{V}_\star$ in $\Omega$, for which can compute the first variation.
We show that $\mathbb{V}_\star$ is stationary as a varifold in
$\Omega \setminus \spt \mu_\star$
(Theorem~\ref{thm:stationary-varifold}), and then we take advantage of this fact to
compute its first variation as a varifold in $\Omega$,
obtaining the balance law~{\eqref{eq:first-variation}--\eqref{eq:zj}}
in Theorem~\ref{thm:first-variation}. Such a balance law relates $\spt\mu_\star$ and
$\spt\nu_\star$ and provides a criticality property for them.
Relying again on Theorem~\ref{thm:properties-nu*} and on a classical structure
results for stationary 1-varifolds with locally bounded positive density due to W.~Allard
and F.J.~Almgren~\cite{AllardAlmgren}, we finally conclude
in Theorem~\ref{thm:structure-S*}
that $\mathfrak{S}_\star$ is locally the union of relatively open segments,
with locally constant densities.
In fact, from \cite[Theorem~1.3]{Bethuel-AC-Acta},
at sufficiently small scale around any point in
$\mathfrak{S}_\star \setminus \mathfrak{E}_\star$, where the exceptional set
$\mathfrak{E}_\star \subset \mathfrak{S}_\star$ is a $\H^1$-null set,
such a union consists of a \emph{single} segment with constant density.

\subsection{Lower bounds on $\nu_\star$, the set $\mathfrak{S}_\star$, and its properties}
\label{sec:lower-bounds-nu*}

In Equation~\eqref{eq:def-nu-eps},
we defined the energy densities
\[
	\nu_\eps := \frac{\eps}{2} \abs{\nabla \M_\eps}^2 +
	\frac{1}{\eps^2}f_\eps(\Q_\eps,\,\M_\eps),
\]
dually seen as measures in $\Omega$. It follows
from Theorem~\ref{lemma:energy-est} and~\eqref{hp:potential_bound}
that, up to extraction
of a subsequence, $\nu_\eps$ converges weakly* in the
sense of measures in $\Omega$ to a limiting
Radon measure $\nu_\star$.
In this section,
we prove some
properties of $\nu_\star$. More precisely,
following the reasoning in \cite[Section~7]{Bethuel-AC-Acta},
we prove clearing-out properties for $\nu_\star$ far from
$\spt\mu_\star$ that allow us to show that the support
of $\nu_\star$ coincides locally, i.e., in
any compact set $K \subset \Omega\setminus \spt\mu_\star$,
with the concentration set for the $1$-density of $\nu_\star$
and that
it is locally a finite union of path-connected components.
(See Theorem~\ref{thm:properties-nu*} below for the precise statements.)
As a consequence,
we obtain locally uniform convergence in
$\Omega \setminus (\spt\mu_\star \cup \spt\nu_\star)$
of the $\M_\eps$-component towards the
limiting map $\M_\star$ given by 
item~\ref{item:mainthm-asymp-conv-QM} of Theorem~\ref{mainthm:asymp}
(Cf. Theorem~\ref{thm:unif-conv-Meps}.)

\begin{notation}
	In this section, and in the forthcoming ones, we continue using
	the notation introduced in Section~\ref{sec:clout-M}.
\end{notation}
\vskip10pt

\paragraph{Auxiliary lemmas and preliminary remarks}
First of all, we show that the full potential
$\frac{1}{\eps^2}f_\eps(\Q_\eps,\,\M_\eps)$
can be locally replaced with the `Allen-Cahn' potential energy
$\frac{1}{\eps} V(\,\cdot\,,\,\M_\eps)$ in the limit as $\eps \to 0$.
This result, contained in Lemma~\ref{lemma:limit-f-minus-V} below,
will be key to many results of this and the next sections. 
The proof of Lemma~\ref{lemma:limit-f-minus-V} (and in turn that 
of Lemma~\ref{lemma:V-f-limit-Lp}) exploits crucially~\eqref{eq:strong-p-conv-rhoeps} 
\begin{lemma}\label{lemma:limit-f-minus-V}
	Let $\left\{\left( \Q_\eps,\,\M_\eps \right)\right\}$ be a sequence of critical
	points of $\F_\eps$, subject to boundary conditions either as
	in~\eqref{bc}--\eqref{hp:bc} or as in \eqref{bcbis}--\eqref{hp:bcbis}, and
	assume that~\eqref{hp:potential_bound} holds.
	Let $B := B(x_0,\,R) \csubset \Omega \setminus \spt \mu_\star$ be any ball.
	Then,
	\begin{equation}\label{eq:limit-f-minus-V}
	\lim_{\eps \to 0}  \int_B \left\{\frac{1}{\eps^2} f_\eps(\Q_\eps(x),\,\M_\eps(x)) - \frac{1}{\eps}V(x,\,\M_\eps(x))\right\}\,{\d}x
		= 0.
	\end{equation}
\end{lemma}

\begin{proof}
	Integrating both sides of~\eqref{eq:pointwise-est-V-bis} over $B$,
	we have
	\[
	\begin{split}
	\int_B &\left\{\frac{1}{\eps^2} f_\eps(\Q_\eps(x),\,\M_\eps(x)) - \frac{1}{\eps}V(x,\,\M_\eps(x))\right\}\,{\d}x \\
	&= \frac{1}{4\eps^2}\int_B \left( 1-\rho_\eps^2 \right)^2\,{\d}x
	+ \int_B \frac{\beta}{2}\left( \frac{1 -\rho_\eps}{\eps} \right) \left(\sqrt{2}+\beta+\beta\rho_\eps\right) \,{\d}x
	+ \chi_\eps \abs{B}.
	\end{split}
	\]
	By~\eqref{eq:chi-to-k*}, the last term on the right-hand
	side tends to $\kappa_\star^2 \abs{B}$ as $\eps \to 0$.
	On the other hand, by~\eqref{eq:strong-p-conv-rhoeps} (with $p = 2$) 
	and~\eqref{eq:unif-conv-mod-Qeps}, as $\eps \to 0$
	we get both
	\[
		\frac{1}{4\eps^2}\int_B \left( 1-\rho_\eps^2 \right)^2\,{\d}x
		\to \kappa_\star^2 \abs{B}
	\]
	and (recalling also the definition~\eqref{eq:k*} of $\kappa_\star$)
	\[
		\int_B \frac{\beta}{2}\left( \frac{1-\rho}{\eps} \right) \left(\sqrt{2}+\beta+\beta\rho_\eps\right) \,{\d}x
		\to -\kappa_\star \frac{\beta}{2}\left(\sqrt{2} + 2\beta\right)
		= -2\kappa_\star^2 \abs{B}.
	\]
	Summing the three contributions above, we obtain
	exactly~\eqref{eq:limit-f-minus-V}.
\end{proof}
The following strengthened version of Lemma~\ref{lemma:limit-f-minus-V}
will be useful later on.
\begin{lemma}\label{lemma:V-f-limit-Lp}
	Let $\left\{\left( \Q_\eps,\,\M_\eps \right)\right\}$ be a sequence of critical
	points of $\F_\eps$, subject to boundary conditions either as
	in~\eqref{bc}--\eqref{hp:bc} or as in \eqref{bcbis}--\eqref{hp:bcbis}, and
	assume that~\eqref{hp:potential_bound} holds.
	Let $B := B(x_0,\,R) \csubset \Omega \setminus \spt \mu_\star$ be any ball.
	Then, for every $p$ with $1 \leq p < +\infty$,
	\begin{equation}\label{eq:limit-f-minus-V-Lp}
	\lim_{\eps \to 0}  \int_B \abs{\frac{1}{\eps^2} f_\eps(\Q_\eps(x),\,\M_\eps(x)) - \frac{1}{\eps}V(x,\,\M_\eps(x))}^p\,{\d}x
		= 0.
	\end{equation}
\end{lemma}

\begin{proof}[{Proof of Lemma~\ref{lemma:V-f-limit-Lp}}]
	The conclusion follows easily by
	combining~\eqref{eq:pointwise-est-V-bis}, the definition~\eqref{eq:k*} of
	$\kappa_\star$, and the uniform convergence
	$\rho_\eps \to 1$ in $B$
	from~\eqref{eq:unif-conv-mod-Qeps} with~\eqref{eq:strong-p-conv-rhoeps}.
\end{proof}

%
%
%

\paragraph{Properties of $\nu_\star$}
We now come to the limiting energy measure $\nu_\star$.
For any $\eps > 0$, we let
\begin{equation}\label{eq:tilde-nu-eps}
	\widetilde{\nu}_\eps := \frac{\eps}{2}\abs{\nabla \M_\eps}^2
	+ \frac{1}{\eps}V(\,\cdot\,,\M_\eps).
\end{equation}
Next, we define
\begin{equation}\label{eq:lower-density-nu*}
	\mathfrak{e}_{\star}(x_0) := \liminf_{r \to 0} \frac{\nu_\star(B(x_0,\,r))}{2r},
	\qquad \forall x_0 \in \Omega.
\end{equation}
Similarly, we set
\begin{equation}\label{eq:upper-density-nu*}
	\mathfrak{e}^{\star}(x_0) := \limsup_{r \to 0} \frac{\nu_\star(B(x_0,\,r))}{2r},
	\qquad \forall x_0 \in \Omega.
\end{equation}
Next, taken a compact set $K \subset \Omega \setminus \spt\mu_\star$,
we define the set
\begin{equation}\label{eq:S-gotico}
	\mathfrak{S}_{\star, K} := \left\{ x \in K \, \colon\, \mathfrak{e}_{\star}(x) > 0 \right\}.
\end{equation}

Theorem~\ref{thm:properties-nu*} and Theorem~\ref{thm:unif-conv-Meps} below,
providing analogues of \cite[Theorem~1.14 and Theorem~1.16]{Bethuel-AC-Acta}
and of \cite[Theorem~1.2]{Bethuel-AC-Acta}, respectively, constitute the main results
of this section.
\begin{theorem}\label{thm:properties-nu*}
	Let $\nu_\eps$ be the energy-measures defined by~\eqref{eq:def-nu-eps}.
	Then, there exists a non-negative Radon measure $\nu_\star$ on $\Omega$
	and a (not-relabelled)
	subsequence so that $\nu_\eps \rightharpoonup^* \nu_\star$ weakly*
	as Radon measures in $\Omega$ as $\eps \to 0$.
	In addition, the measure $\nu_\star$ has the following properties.
	\begin{enumerate}[(i)]
		\item\label{item:tilde-nu-eps}
		Let $\widetilde{\nu}_\eps$ be the energy densities defined
		in~\eqref{eq:tilde-nu-eps}. Then,
		for any compact set $K \subset \Omega \setminus \spt \mu_\star$,
		we have $\widetilde{\nu}_\eps \mres K \rightharpoonup \nu_\star \mres K$
		weakly* as measure in $K$ as $\eps \to 0$.
		\item\label{item:cl-o}
		For any compact set $K \subset \Omega \setminus \spt \mu_\star$ and
		any $\alpha \in (1,\,2)$,
		there exists positive constant $\eta_{\star,K}^{(\alpha)}$ and
		$R_{\star, K}^{(\alpha)}$,
		depending only on $K$, $\alpha$, and $\beta$,
		such that, if $B(x_0,\,R) \subset K$ has radius $R \leq R_{\star,K}^{(\alpha)}$
		and if
		\begin{equation}\label{eq:hp-clearing-out-nu*}
			\nu_\star\left(\overline{B(x_0,\,R)}\right) \leq \eta_{\star,K}^{(\alpha)} R,
		\end{equation}
		then
		\begin{equation}\label{eq:clearing-out-nu*}
			\nu_\star\left(\overline{B\left(x_0,\,\frac{R}{2}\right)}\right) = 0.
		\end{equation}
		\item\label{item:consequence-cl-o}
			Define the sets
				\[
					\mathfrak{S}_{\star,K}^{(\alpha)} := \left\{ x \in K \, \colon\, \mathfrak{e}_{\star}(x) \geq \eta_{\star, K}^{(\alpha)} \right\}, \qquad
					\mathfrak{S}_{\star,K}^{({\rm sup})} := \left\{ x \in K \, \colon\, \mathfrak{e}_{\star}(x) \geq \eta_{\star, K} \right\},
				\]
				where
				\begin{equation}\label{eq:eta*K}
					\eta_{\star,K}
					:= \sup_{\alpha \in (1,\,2)} \eta_{\star,K}^{(\alpha)}.
				\end{equation}
				Then,
				\[
					\mathfrak{S}_{\star,K} = \mathfrak{S}_{\star,K}^{(\alpha)} = \mathfrak{S}_{\star,K}^{({\rm sup})}.
				\]
		\item\label{item:S_*}
		The support of $\nu_\star \mres K$ coincides with $\mathfrak{S}_{\star, K}$. In
		addition, $\mathfrak{S}_{\star, K}$ is a closed,
		$\H^1$-rectifiable set,
		with $\H^1(\mathfrak{S}_{\star, K}) < +\infty$. Moreover, for any
		$x_0 \in \Omega \setminus \spt\mu_\star$ and for any $r > 0$ such that
		$B(x_0,\,2r) \csubset \Omega\setminus\spt\mu_\star$ there exists
		$\rho_0 \in (r,\,2r)$ such that $\mathfrak{S}_{\star, K} \cap B(x_0,\,\rho_0)$
		is a finite union of path-connected components.
	\end{enumerate}
\end{theorem}

\begin{remark}
	The very same argument as in the proof
	of \cite[Proposition~7.3]{Bethuel-AC-Acta} shows that
	\begin{equation}\label{eq:bound-S-gotico}
	\H^1(\mathfrak{S}_{\star, K}) \leq
	{\rm C}_{{\rm H},K} \mathcal{E}_0,
	\end{equation}
	where $\mathcal{E}_0$ is the global bound on $E_\eps(\M_\eps,\,\Omega)$
	given by~\eqref{eq:Eeps-bounded} and ${\rm C}_{{\rm H},K}$
	can actually be chosen as ${\rm C}_{{\rm H},K} = 50 / \eta_{\star,K}$.
\end{remark}

\begin{remark}\label{rk:nu*-spt-mu*}
	At the current stage, we do not know whether
	\[
		\nu_\star(\spt\mu_\star) = 0
	\]
	or not.
\end{remark}

\begin{proof}[{Proof of Theorem~\ref{thm:properties-nu*}, first part}]
Here, we accomplish the first three steps in the proof
of Theorem~\ref{thm:properties-nu*}, i.e., the existence of
$\nu_\star$ and the proof of items~\ref{item:tilde-nu-eps},~\ref{item:cl-o},
and~\ref{item:consequence-cl-o}.
Proving item~\ref{item:S_*} requires some further clearing-out
properties for $\nu_\star$ that, being of interest in their own, will
be stated and proved separately. The proof of item~\ref{item:S_*} will
be therefore provided thereafter.
\setcounter{step}{0}
\begin{step}[Existence of $\nu_\star$ and proof of item~\ref{item:tilde-nu-eps}]
By definition, we have
\begin{equation}\label{eq:nu-eps-tilde-nu-eps}
	\nu_\eps = \widetilde{\nu}_\eps +
	\left( \frac{1}{\eps^2}f_\eps(\Q_\eps,\,\M_\eps) - \frac{1}{\eps} V(x,\,\M_\eps)\right)
\end{equation}
for any $\eps > 0$.
Thus, on the one hand,
the mass of $\nu_\eps$ on $\Omega$ is given by
\begin{equation}\label{eq:mass-nu-eps}
	\nu_\eps(\Omega) = E_\eps(\M_\eps,\,\Omega) +
	\int_\Omega \left( \frac{1}{\eps^2}f_\eps(\Q_\eps,\,\M_\eps) - \frac{1}{\eps} V(x,\,\M_\eps)\right)\,{\d}x;
\end{equation}
and, thanks to the global energy
bound~\eqref{eq:Eeps-bounded}, assumption~\eqref{hp:potential_bound},
and~\eqref{eq:V-finite},
the right-hand side of~\eqref{eq:mass-nu-eps}
is bounded independently of $\eps$,
so that, by standard compactness properties of
Radon measures, there exist a (not-relabelled) subsequence
and a non-negative Radon measure $\nu_\star$ in $\Omega$ such that
\begin{equation}\label{eq:nu-eps-to-nu*}
	\nu_\eps \rightharpoonup^* \nu_\star
\end{equation}
weakly* as measures in $\Omega$, as $\eps \to 0$. On the
other hand, if $K$ is any compact set contained in
$\Omega \setminus \spt\mu_\star$, then~\eqref{eq:nu-eps-tilde-nu-eps}
and Lemma~\ref{lemma:limit-f-minus-V} together with a
standard covering argument entail that
$\widetilde{\nu}_\eps \mres K$ converge to $\nu_\star \mres K$
weakly* as measures in $K$ as $\eps \to 0$.
\end{step}

\begin{step}[Proof of item~\ref{item:cl-o} and item~\ref{item:consequence-cl-o}]
	The proof of is almost identical to that of \cite[Theorem~1.14]{Bethuel-AC-Acta}
	(see \cite[Section~7.1]{Bethuel-AC-Acta}).
	The only difference is given by the presence of the perturbation
	term in~\eqref{eq:energy-decay-M}, which however vanishes in the limit
	$\eps \to 0$. Anyway, for the reader's convenience, we provide some details.

	Let $K \subset \Omega \setminus \spt\mu_\star$ be any compact set,
	which we may assume to have non-empty interior, and fix any $\alpha \in (1,\,2)$.
	Let $\eta_{\star,K}^{(\alpha)}$, $R_{\star,K}^{(\alpha)}$, and
	$\eps_{\star,K}^{(\alpha)}$ be the same numbers,
	depending only on $K$, $\alpha$, $\beta$,
	and the energy bound $\mathcal{E}_0$, as given by Theorem~\ref{thm:clearing-out-M}.
	(For ease of notation, we drop
	the indices $K$ and $\alpha$ in the following computations.)

	By the definition~\eqref{eq:tilde-nu-eps} of $\widetilde{\nu}_\eps$,
	for any ball
	$B(x_0,\,R) \csubset \Omega\setminus \spt\mu_\star$, we have
	\[
		E_\eps(\M_\eps,\,B(x_0,\,R)) = \widetilde{\nu}_\eps(B(x_0,\,R))
		= \widetilde{\nu}_\eps\left( \overline{B(x_0,\,R)} \right),
	\]
	and therefore, by item~\ref{item:tilde-nu-eps},
	for any $r \in (0,\,R)$ we have
	\[
		\limsup_{\eps \to 0} \widetilde{\nu}_\eps(B(x_0,\,r))
		= \limsup_{\eps \to 0} \nu_\eps(B(x_0,\,r))
		\leq \nu_\star\left( \overline{B(x_0,\,r)} \right)
		\leq \nu_\star(B(x_0,\,R)).
	\]
	Thus, if $B(x_0,\,R) \subset K$ is any ball with radius $R \leq R_\star$
	and we assume that~\eqref{eq:hp-clearing-out-nu*} holds, then for any
	$r \in (0,\,R)$, there exists $\eps_r > 0$
	so that, for all $\eps \leq \eps_r R$,
	\[
		\widetilde{\nu}_\eps(B(x_0,\,r)) \leq \frac{5}{4}\eta_\star R.
	\]
	Therefore, choosing $r = 8R/9$ we get
	\[
		\widetilde{\nu}_\eps(B(x_0,\,r)) \leq 2\eta_\star r.
	\]
	Using~\eqref{eq:energy-decay-M},
	we obtain
	\[
	\begin{split}
		\widetilde{\nu}_\eps\left( B\left(x_0,\,\frac{5R}{9}\right) \right)
		= \widetilde{\nu}_\eps\left( B\left(x_0,\, \frac{5r}{8}\right) \right)
		& \leq {\rm C}_{\rm nrg} \eps
		\left(\frac{E_\eps(\M_\eps;\,B(x_0,\,r))}{r} + C_\alpha(K) R^\alpha \right) \\
		& \leq \eps \left(\frac{9}{8} \eta_\star + C_\alpha(K) R^\alpha\right).
	\end{split}
	\]
	for any $\eps \leq \min\{\eps_\star,\,\eps_r\} R$.
	Letting
	$\eps \to 0$, the conclusion follows.

	Once item~\ref{item:cl-o} has been obtained, item~\ref{item:consequence-cl-o}
	follows immediately. Indeed, on the one hand, the inclusions
	\[
		 \mathfrak{S}_{\star,K}^{({\rm sup})}
		 \subseteq \mathfrak{S}_{\star,K}^{(\alpha)}
		 \subseteq \mathfrak{S}_{\star,K}
	\]
	are obvious. On the other hand, it follows from item~\ref{item:cl-o} that
	for any $x_0 \in K$, \emph{either} $\mathfrak{e}_\star(x_0) = 0$ or the
	inequality
	\[
		\mathfrak{e}_\star(x_0) \geq \eta_{\star,K}^{(\alpha)}
	\]
	holds for every $\alpha \in (1,\,2)$, hence it must hold also passing to the
	supremum over $\alpha \in (1,\,2)$ on both sides. Since $\mathfrak{e}_\star(x_0)$
	does not depend on $\alpha$, we must have
	\[
		\mathfrak{e}_\star(x_0) \geq \sup_{\alpha \in (1,\,2)} \eta_{\star,K}^{(\alpha)} = \eta_{\star,K}.
	\]
	Thus $\mathfrak{S}_{\star,K} \subseteq \mathfrak{S}_{\star,K}^{({\rm sup})}$,
	and the conclusion follows.
	\qedhere
\end{step}
\end{proof}

A first consequence of Theorem~\ref{thm:properties-nu*} is 
the following lemma, which refines \cite[Lemma~3.17]{CDS1} 
and will be useful later on.
\begin{lemma}\label{lemma:V-h-limit}
	Let $\left\{\left( \Q_\eps,\,\M_\eps \right)\right\}$ be a sequence of critical
	points of $\F_\eps$, subject to boundary conditions either as
	in~\eqref{bc}--\eqref{hp:bc} or as in \eqref{bcbis}--\eqref{hp:bcbis}, and
	assume that~\eqref{hp:potential_bound} holds.
	Then,
for any simply connected open set $G \csubset \Omega \setminus \spt \mu_\star$
	with smooth boundary, we have 
	\begin{equation}\label{eq:fail}
		c_\beta \H^1(\S_{\M_\star} \cap G)
		\leq \nu_\star \mres G,
	\end{equation}
	where $c_\beta = \frac{2\sqrt{2}}{3}\left( \sqrt{2}\beta + 1 \right)^{3/2}$.
Consequently,
	\begin{equation}\label{eq:fail-bis}
		c_\beta \H^1 \mres \S_{\M_\star}
		\leq \nu_\star
	\end{equation}
	as measures in $\Omega$.
\end{lemma}

\begin{proof}
	The inequality~\eqref{eq:fail} is a straightforward 
	consequence of \cite[(3.59)]{CDS1} 
	and of item~\ref{item:tilde-nu-eps} of Theorem~\ref{thm:properties-nu*}. 
	Given~\eqref{eq:fail},~\eqref{eq:fail-bis} follows easily by the strong convergence
	$\M_\eps \to \M_\star$ in $L^1(\Omega)$ as $\eps \to 0$ given by
	Theorem~\ref{mainthm:asymp}.
\end{proof}

As another elementary consequence of the clearing-out
property for $\nu_\star$ in~\ref{item:cl-o} of
Theorem~\ref{thm:properties-nu*}, we
have the following lemma, corresponding
to \cite[Lemma~7.9 and Proposition~7.2]{Bethuel-AC-Acta}.

\begin{lemma}\label{lemma:analogue-of-Bethuel-7.1}
	Let $K \subset \Omega \setminus \spt\mu_\star$ be
	any compact set. 
\begin{enumerate}[(i)]
	\item For any $x_0 \in {\rm int}\, K \setminus \mathfrak{S}_{\star,K}$,
	there exists a positive radius $r_{x_0}$ such that
	$B(x_0,\,r_0) \subset K$ and
	$\nu_\star\left( \overline{B(x_0,\,r_{x_0})} \right) = 0$. In
	particular, $\mathfrak{e}_\star(x) = \mathfrak{e}^\star(x) = 0$
	for any $x \in B(x_0,\,r_{x_0})$.
	\item Let $x_0 \in \mathfrak{S}_{\star,K}$ be such that
	$B(x_0,\,r_{x_0}) \subset K$. Then,
	$\nu_{\star}\left( B(x_0,\,r) \right) \geq \eta_{\star,K} r$
	for any $r \in (0,\,r_0]$.
	\item The set $\mathfrak{S}_{\star,K}$ is a closed subset
	of $\Omega \setminus \spt\mu_\star$ (and hence of $\Omega$).
\end{enumerate}
\end{lemma}

\begin{proof}
Being the proof truly elementary, for the sake of brevity,
it is left to the reader (who can however
consult \cite[Section~7.2]{Bethuel-AC-Acta}
for full details).
\end{proof}

The proof of item~\ref{item:S_*} of Theorem~\ref{thm:properties-nu*} requires
to adapt some results from \cite{Bethuel-AC-Acta}. Our fist goal is 
providing an analogue of \cite[Proposition~4.15]{Bethuel-AC-Acta} accounting
for the presence of the perturbation due to the coupling term in $\F_\eps$.

Similarly to as in \cite[Section~1.5]{Bethuel-AC-Acta},
we introduce the following piece of notation:
let $\mathcal{U} \csubset \Omega \setminus \spt\mu_\star$
be a bounded, open set and, for $\delta > 0$, denote
\begin{align}
	& \mathcal{U}_\delta := \left\{ x \in \Omega \setminus \spt\mu_\star : \dist(x,\,\mathcal{U}) \leq \delta \right\}, \label{eq:Udelta} \\
	& \mathcal{V}_\delta := \left\{ x \in \Omega \setminus \left( \spt\mu_\star \cup \mathcal{U} \right) : \dist(x,\,\mathcal{U}) \leq \delta \right\} \label{eq:Vdelta}.
\end{align}
\begin{prop}\label{prop:analogue-of-Bethuel-4.6}
	Let $K \subset \Omega \setminus \spt\mu_\star$ be a compact set,
	let $\mathcal{U} \subset K$ be a bounded, open
	set, and let $\delta > 0$ be so small that
	$\mathcal{U}_\delta \subset K$.
	Then, for any $\alpha \in (1,\,2)$
	there exists a constant $C_{\rm ext}(\mathcal{U},\,\delta)$, possibly depending
	on $\alpha$, $\beta$, $K$, $\Omega$, and $\mathcal{U}$
	and $\delta$, but not on $\eps$, such that
	\begin{equation}\label{eq:est-4.6-B}
		E_\eps\left(\M_\eps;\,\mathcal{U}_{\delta / 4}\right)
		\leq C_{\rm ext}(\mathcal{U},\,\delta)\left(
		E_\eps(\M_\eps;\,\mathcal{V}_\delta) + \eps E_\eps(\M_\eps; \mathcal{U}_\delta) + \o_{\eps \to 0}(1) \right).
	\end{equation}
\end{prop}

\begin{proof}
	The proof of \cite[Proposition~4.15]{Bethuel-AC-Acta} involves only
	three ingredients: an application of the Pohozaev
	inequality with an appropriate test field to obtain
	control of the energy on $\mathcal{U}_{\delta/4}$
	with that on `external domain' $\mathcal{V}_\delta$
	(\cite[Proposition~3.11]{Bethuel-AC-Acta}); a standard
	covering argument; and \cite[Proposition~4.13]{Bethuel-AC-Acta},
	whose analogue in our case has already been obtain
	in Proposition~\ref{prop:analogue-of-Bethuel-4.3}.

	In our case, the proofs are largely similar to as in \cite{Bethuel-AC-Acta},
	but some adaptation is required, especially because
	we need to take care of the further dependence on $\Q_\eps$
	when dealing with Pohozaev inequalities (i.e., when testing
	the stress-energy tensor $T_{jk}^\eps$ in~\eqref{stressenergy}).

	Before going on, we explicitly remark that, in our case,
	the condition \cite[(4.33)]{Bethuel-AC-Acta}, i.e., a uniform local bound
	on $\M_\eps$,
	is globally satisfied in $\Omega$, thanks to~\eqref{max-QM},
	and therefore we do not need to verify the further condition
	\cite[(4.24)]{Bethuel-AC-Acta} (which, as a matter of fact, is used in
	\cite{Bethuel-AC-Acta} only to ensure a local uniform bound on $u_\eps$).
	As a consequence, Proposition~\ref{prop:analogue-of-Bethuel-4.3}
	above provides the analogue of the three statements
	\cite[Proposition~4.11, Proposition~4.13, and Proposition~4.14]{Bethuel-AC-Acta},
	without changes (if not in the dependence of the constants, as
	specified in Proposition~\ref{prop:analogue-of-Bethuel-4.3}).

	\setcounter{step}{0}
	\begin{step}[Controlling the energy on external domains]
		Let $K \subset \Omega \setminus \spt\mu_\star$ be a
		given compact set and, as described in the statement, choose
		$\delta > 0$ so small that $\mathcal{U}_\delta \subset K$.
		Fix $\chi_\delta : \R^2 \to \R$, a smooth cut-off function so that
		\[
			\chi_\delta(x) = 1 \quad \mbox{for } x \in \mathcal{U}_{\delta/2},
			\qquad \chi_\delta(x) = 0 \quad \mbox{for }
			x \in \R^2 \setminus \mathcal{U}_\delta,
			\qquad \abs{\nabla \chi_\delta} \leq \frac{4}{\delta}
			\quad \mbox{on } \R^2.
		\]
		Note that the support of $\chi_\delta$ is anyway
		contained in $\Omega$ and that
		$\mathcal{V}_\delta \supseteq \mathcal{U}_\delta \setminus \mathcal{U}_{\delta/2}$.
		Going back to Proposition~\ref{lemma:pohozaev} and
		testing the stress-energy tensor $T^\eps_{jk}$
		in~\eqref{stressenergy} against the
		field $\X(x) = x \chi_\delta(x)$,
		we obtain
		\begin{equation}\label{eq:ineq-3.27-B-preliminare}
		\begin{split}
			\int_{\mathcal{U}_\delta} \frac{1}{\eps^2} f_\eps(\Q_\eps,\,\M_\eps) (\div \X) \,{\d}x = &\int_{\mathcal{U}_\delta} \eps \left\{ \partial_j \M_\eps \cdot \partial_k \M_\eps \partial_j X_k - \frac{1}{2} (\div \X) \abs{\nabla \M_\eps}^2 \right\}\,{\d}x \\
			& + \int_{\mathcal{U}_\delta} \left\{ \partial_j \Q_\eps \cdot \partial_k \Q_\eps \partial_j X_k - \frac{1}{2} (\div \X) \abs{\nabla \Q_\eps}^2 \right\}\,{\d}x
		\end{split}
		\end{equation}
		We observe that, since $\mathcal{U}_\delta \csubset \Omega \setminus \spt\mu_\star$,
		by item~\ref{item:mainthm-asymp-strong-conv} of Theorem~\ref{mainthm:asymp},
		we have 
		strong convergence $\Q_\eps \to \Q_\star$ in
		$W^{1,2}(\mathcal{U}_\delta)$ as $\eps \to 0$.
		Thus, by~\eqref{eq:pohozaev-Q*}, we deduce that
		the second line in~\eqref{eq:ineq-3.27-B-preliminare} vanishes as
		$\eps \to 0$. Consequently, recalling~\eqref{eq:Udelta},~\eqref{eq:Vdelta},
		by an easy computation we obtain the inequality
		\begin{equation}\label{eq:ineq-3.27-B}
			\frac{1}{\eps^2}\int_{\mathcal{U}_{\delta/2}} f_\eps(\Q_\eps(x),\,\M_\eps(x)) \,{\d}x
			\leq C(\mathcal{U},\,\delta) E_\eps(\M_\eps;\,\mathcal{V}_\delta) + \o_{\eps\to 0}(1),
		\end{equation}
		where $C(\mathcal{U}, \delta)$ is a constant depending on $\mathcal{U}$ and $\delta$, but not on $\eps$.
		On the other hand, exploiting
		Lemma~\ref{lemma:limit-f-minus-V},~\eqref{eq:ineq-3.27-B} becomes
		\begin{equation}\label{eq:ineq-3.27-B-bis}
			\frac{1}{\eps}\int_{\mathcal{U}_{\delta/2}} V(x,\,\M_\eps(x))\,{\d}x
			\leq C(\mathcal{U},\,\delta) \left( E_\eps(\M_\eps;\,\mathcal{V}_\delta) + \o_{\eps \to 0}(1) \right),
		\end{equation}
		which replaces~\cite[(3.27)]{Bethuel-AC-Acta},
		and where the constant $C(\mathcal{U},\delta)$ depends on $\mathcal{U}$ and $\delta$, but not on $\eps$.
	\end{step}

	\begin{step}[Local bounds and covering argument]
		By a standard argument using Lebesgue's covering lemma, we may cover
		$\overline{\mathcal{U}_{\delta / 4}}$
		using only finitely many balls $B(x_i,\,\delta/8)$,
		with $x_i \in \overline{\mathcal{U}}_{\delta/4}$,
		$i \in I$, $\sharp I < +\infty$.
		Notice that we have also
		$\cup_{i \in I} B(x_i,\,\delta/4) \subset \mathcal{U}_{\delta/2}$.
		Using Proposition~\ref{prop:analogue-of-Bethuel-4.3}
		with $x_0 = x_i$ and $R = \delta/4$, for any $i \in I$, and
		recalling~\ref{eq:GL-decay}, we see that
		\[
			E_\eps(\M_\eps;\,B(x_i,\,\delta/8)) \lesssim
			\frac{1}{\eps}\int_{B(x_i,\, 3\delta/16)} V(x,\,\M_\eps(x)) \,{\d}x+
			\frac{\eps}{\delta} E_\eps(\M_\eps;\,B(x_i,\,\delta/4))
			+ \eps C_\alpha(K) \delta^\alpha,
		\]
		for any $\alpha \in (1,\,2)$, where the implicit constant
		on the right-hand side depends only on $\beta$ and on $\Omega$.
		Summing all these contributions,
		we get
		\begin{equation}\label{eq:ineq-4.51-B}
			E_\eps(\M_\eps;\,\mathcal{U}_{\delta/4}) \lesssim
			\sharp I \left( \frac{1}{\eps}\int_{\mathcal{U}_{\delta/2}} V(x,\,\M_\eps(x)) \,{\d}x+
			\frac{\eps}{\delta} E_\eps(\M_\eps;\,\mathcal{U}_{\delta/2})
			+ \eps C_\alpha(K) \delta^{\alpha} \right),
		\end{equation}
		where the implicit constant on the right-hand side depends
		only on $\alpha$, $\beta$, $K$, $\Omega$ (while $\sharp I$
		depends on $\delta$, of course), and not on $\eps$.
	\end{step}

	\begin{step}[Conclusion]
		The desired inequality~\eqref{eq:est-4.6-B} follows immediately
		by combining~\eqref{eq:ineq-3.27-B-bis} and~\eqref{eq:ineq-4.51-B}.
		\qedhere
	\end{step}
\end{proof}

With Proposition~\ref{prop:analogue-of-Bethuel-4.6} at hand, we
easily obtain the analogue of \cite[Theorem~1.16]{Bethuel-AC-Acta},
which is the key tool in the proof of part~\ref{item:S_*}
of Theorem~\ref{thm:properties-nu*}.

\begin{theorem}\label{thm:analogue-of-Bethuel-8}
	Let $\mathcal{U} \csubset \Omega \setminus \spt\mu_\star$
	be an open set.
	Assume that, for some $\delta > 0$, there holds
	$\nu_\star(\mathcal{V}_{\delta}) = 0$. Then,
	$\nu_\star(\overline{\mathcal{U}}) = 0$.
\end{theorem}

\begin{proof}
	Let $K \subset \Omega \setminus \spt\mu_\star$ be a
	compact set so that $\mathcal{U} \csubset K$.
	Proceeding similarly to as in \cite[Section~7.3]{Bethuel-AC-Acta}, since
	$\nu_\star(\mathcal{V}_\delta) = 0$, we have
	\[
		E_\eps(\M_\eps,\,\mathcal{V}_\delta) \to 0
		\qquad \mbox{as } \eps \to 0.
	\]
	Hence, by Proposition~\ref{prop:analogue-of-Bethuel-4.6}, we obtain
	\[
		\nu_\star\left(\mathcal{U}_{\delta/8}\right) \lesssim
		\limsup_{\eps \to 0} \left( E_\eps(\M_\eps,\,\mathcal{V}_\delta) +
		\eps E_\eps\left(\M_\eps;\,\mathcal{U}_{\delta/4}\right) \right)
		= 0,
	\]
	which implies the claimed conclusion.
\end{proof}

We are now ready for completing the proof of Theorem~\ref{thm:properties-nu*}.
\begin{proof}[{Proof of Theorem~\ref{thm:properties-nu*}, completion}]
In order to complete the proof of Theorem~\ref{thm:properties-nu*},
we need analogues of the results
in \cite[Section~7.4 and Section~7.5]{Bethuel-AC-Acta}.
However, with Lemma~\ref{lemma:analogue-of-Bethuel-7.1} and
Theorem~\ref{thm:analogue-of-Bethuel-8} at hand,
these results follow with no significant changes
by exactly the same arguments as in \cite{Bethuel-AC-Acta}.
Therefore, we just sketch the main points of the argument
below, addressing the reader to~\cite{Bethuel-AC-Acta} for
full details.

\setcounter{step}{2}
\begin{step}[$\mathfrak{S}_{\star,K}$ is closed, it has finite $\H^1$-measure,
and it coincides with $\spt (\nu_\star \mres K)$]
	The fact that $\mathfrak{S}_{\star,K}$ is closed in $\Omega \setminus \spt\mu_\star$
	(hence, in $\Omega$) is already contained
	Lemma~\ref{lemma:analogue-of-Bethuel-7.1}, from which it also follows
	immediately that $\spt(\nu_\star \mres K) = \mathfrak{S}_{\star,K}$. The fact that
	$\mathfrak{S}_{\star,K}$ has finite $\H^1$-measure,
	with the estimate
	\[
		\H^1\left( \mathfrak{S}_{\star,K} \right)
		\leq {\rm C}_{{\rm H},K} \mathcal{E}_0,
	\]
	where $\mathcal{E}_0$ is the constant on the right-hand side
	of~\eqref{eq:Eeps-bounded} and
	${\rm C}_{{\rm H},K} = 50\,/\,\eta_{\star,K}$
	follows again from Lemma~\ref{lemma:analogue-of-Bethuel-7.1}
	exactly as in \cite[Proposition~7.3]{Bethuel-AC-Acta}.
	Being the argument based on a covering argument and a
	trivial application of Lemma~\ref{lemma:analogue-of-Bethuel-7.1},
	we leave it to the reader to check that the details presented
	in \cite{Bethuel-AC-Acta} carry over completely unchanged.
\end{step}

\begin{step}[Connectedness properties of $\mathfrak{S}_{\star,K}$]
	Given $r > 0$ and any $x_0 \in \Omega \setminus \spt\mu_\star$
	such that $B(x_0,\,2r) \csubset \Omega \setminus \spt\mu_\star$,
	we consider the compact set
	\[
		\mathfrak{S}_{\star,K,\rho}(x_0) :=
		\mathfrak{S}_{\star,K} \cap \overline{B(x_0,\,r)}
		\qquad \mbox{for } \rho \in [0,\,2r).
	\]
	Then, arguing exactly as in \cite[Proposition~7.4]{Bethuel-AC-Acta},
	we can prove that the compact set
	\[
		\mathcal{C}_{\star,r}(x_0) :=
		\mathfrak{S}_{\star,K,r}(x_0) \cup
		\mathbb{S}^1(x_0,\,r)
	\]
	is a continuum, i.e., it is compact and connected.
	Once that Theorem~\ref{thm:analogue-of-Bethuel-8} is given,
	the details are identical to as in \cite{Bethuel-AC-Acta}.
	Since they are based only on an approximation argument in the Hausdorff
	metric for the set $\mathcal{C}_{\star,r}(x_0)$ by continua
	and general properties of continua under Hausdorff convergence,
	we leave them to the reader (who, for the general theory of
	continua, may consult, for instance, \cite{Falconer}).

	Path-connectedness of $\mathcal{C}_{\star,r}(x_0)$ now
	follows immediately, because any continuum with finite
	$\H^1$-measure is path-connected
	(see, e.g., \cite[Lemma~3.12]{Falconer}).

	Gathering the pieces of information above and arguing as
	in \cite[Section~7.4.1]{Bethuel-AC-Acta}, we obtain the
	analogue of \cite[Proposition~1.17]{Bethuel-AC-Acta}, i.e.,
	the fact that for any
	$x_0 \in \Omega \setminus \spt\mu_\star$ and for any $r > 0$ such that
	$B(x_0,\,2r) \csubset \Omega\setminus\spt\mu_\star$ there exists
	$\rho_0 \in (r,\,2r)$ such that $\mathfrak{S}_{\star, K} \cap B(x_0,\,\rho_0)$
	is a finite union of path-connected components.
	(Being an argument in general topology based solely on the fact
	that $\mathcal{C}_{\star,r}(x_0)$ is path-connected,
	no changes are needed with respect to~\cite{Bethuel-AC-Acta},
	as the reader may easily check.)
\end{step}

	\begin{step}[Rectifiability]
		Since rectifiability is a local property, it
		suffices to show that $\mathfrak{S}_{\star,K,r}$
		is rectifiable. To this purpose, it is clearly
		enough to prove that $\mathcal{C}_{\star,r}(x_0)$
		is rectifiable, for $x_0$ and $r > 0$ as in the
		previous step. But this follows immediately,
		because every continuum with finite $\H^1$-measure
		is rectifiable (cf. \cite[Lemma~3.13]{Falconer}).
		\qedhere
	\end{step}
\end{proof}

\paragraph{The set $\mathfrak{S}_{\star}$}
Taking advantage of the connectedness and rectifiability properties
of the sets $\mathfrak{S}_{\star,K}$,
we show that the set $\spt \nu_\star \setminus \spt \nu_\star$ is
locally path-connected, countably $\H^1$-rectifiable, with locally finite measure.
First of all, we introduced a couple of definitions that will be used throughout.
\begin{definition}\label{def:S_star}
	We define
	\begin{equation}\label{eq:S_star}
		\mathfrak{S}_{\star} := \spt \nu_\star \setminus \spt \mu_\star
	\end{equation}
	and
	\begin{equation}\label{eq:h_star}
		\lambda_\star := \H^1 \mres \mathfrak{S}_\star.
	\end{equation}
\end{definition}

\begin{prop}\label{prop:S_star}
	The set $\mathfrak{S}_{\star}$ defined in~\eqref{eq:S_star}
	satisfies
	\[
		\mathfrak{S}_\star = \{ x \in \Omega \setminus \spt \mu_\star\,:\, \mathfrak{e}_\star > 0 \}.
	\]
	Moreover, $\mathfrak{S}_\star$
	is countably $\H^1$-rectifiable, with locally finite
	$\H^1$-measure and locally path-con\-nect\-ed.
\end{prop}

\begin{proof}
	The fact that $\mathfrak{S}_\star$ equals the set of points with
	positive lower density for $\nu_\star$ in $\Omega \setminus \spt\mu_\star$
	is an obvious consequence of item~\ref{item:consequence-cl-o} of
	Theorem~\ref{thm:properties-nu*}. In addition, clearly,
	$\mathfrak{S}_\star \cap K = \mathfrak{S}_{\star, K}$
	for any $K \subset \Omega \setminus \spt\mu_\star$, hence
	the property $\H^1(\mathfrak{S}_\star \cap K) < +\infty$ and
	the fact that $\mathfrak{S}_\star \cap K$ is $\H^1$-rectifiable
	for any $K \subset \Omega \setminus \spt\mu_\star$
	are obvious consequences of item~\ref{item:S_*} of Theorem~\ref{thm:properties-nu*}.
	Again in view of item~\ref{item:S_*} of Theorem~\ref{thm:properties-nu*},
	it also follows that $\mathfrak{S}_\star$ is locally path-connected.
	Finally, $\mathfrak{S}_\star$ is countably $\H^1$-rectifiable
	because we could write it as a countable union of sets $\mathfrak{S}_{\star,K_n}$
	(which are rectifiable by item~\ref{item:S_*} of Theorem~\ref{thm:properties-nu*}),
	where $\{K_n\}$ is a sequence of compact sets
	$K_n \subset \Omega \setminus \spt\mu_\star$ such that
	$K_n \subset K_{n+1}$ and
	$\cup_{n \in \mathbb{N}} K_n = \Omega \setminus \spt\mu_\star$
	(such an exhaustion in nested compact set certainly exists, as
	$\Omega \setminus \spt\mu_\star$ is an open set of $\R^2$ --- recall that
	$\spt\mu_\star$ is a finite set of points. 
	Then, the conclusion follows because countable rectifiability is a local property,
	stable under countable unions
	(and, clearly, the precise choice of the sequence $\{K_n\}$ is irrelevant).
\end{proof}


\begin{remark}\label{cor:H1-res-S*}
	As an obvious consequence of Proposition~\ref{prop:S_star},
	the measure $\lambda_\star = \H^1 \mres \mathfrak{S}_\star$ is locally
	finite on $\Omega \setminus \spt \mu_\star$, hence it is a
	Radon measure on $\Omega \setminus \spt \mu_\star$.
\end{remark}

\paragraph{Locally uniform convergence away from $\spt\nu_\star \cup \spt\mu_\star$}
We are finally ready to improve on the $L^p$-compactness for
the $\M_\eps$-component (already given by Theorem~\ref{mainthm:asymp})
and obtain locally uniform convergence to the limiting map $\M_\star$
away from $\spt \nu_\star \cup \spt\mu_\star$.
To this purpose, we prove first an auxiliary lemma which says that, thanks to
the clearing-out property of $\nu_\star$ and to~\eqref{eq:fail-bis}, and up to
redefinition on a set of measure zero, the
regularity of $\M_\star$ (which was defined only almost everywhere)
can be improved. 
We recall from \cite[Theorem~3.14]{CDS1} that almost every point $x \in \Omega$, 
	\begin{equation}\label{eq:def-M*}
		\M_\star(x) = \tau(x) \left( \sqrt{2}\beta + 1 \right)^{1/2} \n_\star(x),
	\end{equation}
	where $\n_\star(x)$ is a unit eigvector of $\Q_\star(x)$ corresponding 
	to its positive eigenvalue.
\begin{lemma}\label{lemma:M*-continuous}
	Up to redefinition on a set of (Lebesgue) measure zero,
	the map $\M_\star : \Omega \to \R^2$
	defined in~\eqref{eq:def-M*} is continuous on the open set
	$\Omega \setminus (\spt\mu_\star \cup \spt\nu_\star)$ and
	given by
	\[
		\M_\star = \tau \left( \sqrt{2}\beta + 1 \right)^{1/2} \n_\star,
	\]
	where $\tau(x) \in \{-1,\,+1\}$ and the function $x \mapsto \tau(x)$
	is constant in each connected component of
	$\Omega \setminus (\spt\mu_\star \cup \spt\nu_\star)$.
\end{lemma}

\begin{proof}
	Let $x_0 \in \Omega \setminus (\spt\mu_\star \cup \spt\nu_\star)$. Then,
	by Lemma~\ref{lemma:analogue-of-Bethuel-7.1},
	we can find a radius $r_{x_0} > 0$ so that
	$B(x_0,\,r_{x_0}) \csubset \Omega \setminus \spt\mu_\star$ and
	$\nu_\star(B(x_0,\,r_{x_0})) = 0$. Hence, we can find a
	compact set $K \subset \Omega \setminus \spt\mu_\star$ with
	non-empty interior $G := {\rm int}\, K$, such that
	$x_0 \in G \setminus \mathfrak{S}_{\star, K}$. Thus,
	from~\eqref{eq:fail-bis} and Lemma~\ref{lemma:analogue-of-Bethuel-7.1},
	we have
	\[
		0 \leq \lim_{r \to 0} \frac{\H^1\left(\S_{\M_\star} \cap B(x_0,\,r)\right)}{r}
		\leq \lim_{r \to 0} \frac{\nu_\star(B(x_0,\,r))}{r} = 0.
	\]
	Since, by Theorem~\ref{mainthm:asymp},
	$\M_\star \in \SBV(\Omega,\,\R^2)$ and it is bounded,
	recalling
	\cite[Theorem~7.8]{AmbrosioFuscoPallara}, we see that
	$x_0 \not\in \S_{\M_\star}$. By the continuity of $\n_\star$ in $G$,
	it follows that, up to
	redefining $\M_\star$ on a set of zero Lebesgue measure, $\tau$ must constant
	in $G$, so that $\M_\star$ is continuous in $G$.
	Consequently, $\M_\star$ is continuous, and $\tau$ is constant, 
	in the connected component of
	$\Omega \setminus (\spt\mu_\star \cup \spt\nu_\star)$ containing $x_0$.
	The conclusion follows.
\end{proof}

\begin{remark}\label{rk:SM*-small}
	In particular, $\overline{\S_{\M_\star}} \subset \spt \mu_\star \cup \spt\nu_\star$.
	Thus, in view of~\eqref{eq:fail-bis} and of~\eqref{eq:S-gotico}, we have
	$\H^1\left(\overline{\S_{\M_\star}} \cap K\right) < +\infty$ for every compact set
	$K \subset \Omega$.
\end{remark}

With Theorem~\ref{thm:properties-nu*} and
Lemma~\ref{lemma:M*-continuous} at hand, we can
finally establish local uniform convergence of $\M_\eps$
towards $\M_\star$, in every connected component of
$\Omega \setminus (\spt\mu_\star \cup \spt \nu_\star)$.

\begin{theorem}\label{thm:unif-conv-Meps}
	Let $\{(\Q_\eps,\,\M_\eps)\}$ be a sequence of critical points of $\F_\eps$,
	subject to either boundary conditions as in~\eqref{bc}--\eqref{hp:bc}
	or as in~\eqref{bcbis}--\eqref{hp:bcbis}, and assume that~\eqref{hp:potential_bound}
	holds. Set $U_\star := \Omega \setminus (\spt\mu_\star \cup \spt \nu_\star)$.
	Then, up to extraction of a (not-relabelled) subsequence, for every connected
	component $U_j$ of $U_\star$
	we have that $\M_\eps \to \M_\star$ locally uniformly on $U_j$
	as $\eps \to 0$.
\end{theorem}

\begin{proof}
The proof is along the lines of that of
\cite[item~{(ii)} of Theorem~1.2]{Bethuel-AC-Acta}, presented
in \cite[Section~7.6 and Section~7.7]{Bethuel-AC-Acta},
but it requires some modifications in order to deal with the moving
wells.

\setcounter{step}{0}
\begin{step}
\emph{
	We prove that
	for any $x_0 \in U_\star$ and for the radius $r_{x_0} > 0$ given by
	Lemma~\ref{lemma:analogue-of-Bethuel-7.1}, there exists a
	(not relabelled) subsequence so that}
	\begin{equation}\label{eq:analogue-of-Bethuel-7.22}
		\norm{\M_\eps - \M_\star}_{L^\infty(B(x_0,\,3r_{x_0}/4))}
		\to 0 \qquad \mbox{as } \eps \to 0.
	\end{equation}
\noindent
	Let $K$ be any compact set such that
	$B(x_0,\,r_{x_0}) \csubset K \subset U_\star$. (It will turn out
	that the choice of $K$ is not important.) Then,
	$\nu_\star\left( \overline{B(x_0,\,r_{x_0})} \right) = 0$,
	so that
	\begin{equation}\label{eq:unif-conv-compu0}
		\limsup_{\eps \to 0} E_{\eps}(\M_\eps;\,B(x_0,\,r_{x_0}))
		= \limsup_{\eps \to 0} \nu_\eps(B(x_0,\,r_{x_0}))
		\leq \nu_\star \left( \overline{B(x_0,\,r_{x_0})} \right) = 0
	\end{equation}
	Hence, we can apply Theorem~\ref{thm:clearing-out-M},
	so that~\eqref{eq:close-to-wells} yields
	\begin{equation}\label{eq:unif-conv-compu1}
		\dist(\M_\eps(x),\,\Sigma(\Q_\eps(x))) \leq \delta_\beta \qquad
		\mbox{for any } x \in B(x_0,\,3 r_{x_0}/4),
	\end{equation}
	for any $\eps > 0$ small enough that, say, $\abs{\Q_\eps} \geq 1/2$
	in $K$.
	Recall that $\Sigma(\Q_\eps) = \{(\M_+)_\eps,\,(\M_-)_\eps\}$
	and that $(\M_\pm)_\eps := \pm (1+\sqrt{2}\beta \rho_\eps(x))^{1/2} \n_\eps(x)$, so that
	$\abs{(\M_+)_\eps - (\M_-)_\eps} > 2$ independently of
	$x \in K$, for any $\eps$ small enough that $\abs{\Q_\eps} \geq 1/2$
	in $K$ (see~\eqref{eq:M-pm-eps}).
	Since $\delta_\beta \leq 1$,
	this implies that, at $x_0$, ~\eqref{eq:unif-conv-compu1} is
	satisfied for \emph{either} $\N_\eps(x_0) = (\M_+)_\eps(x_0)$ \emph{or}
	$\N_\eps(x_0) = (\M_-)_\eps(x_0)$ (and not for both).
	By~\eqref{eq:unif-conv-compu1}, the uniform convergence $\rho_\eps \to 1$
	given by~\eqref{eq:unif-conv-mod-Qeps}, and the continuity of $\M_\eps$,
	it follows that
	\begin{equation}\label{eq:claim-unif-conv}
		\| \N_\eps(x) - \N_\eps(x_0) \| \leq \delta_\beta \qquad
		\mbox{for any } x \in B(x_0,\,3r_{x_0}/4).
	\end{equation}
	Thus, we see that, for any $\eps > 0$ small enough,
	\[
		\N_\eps(x) = \tau_\eps \left( 1 + \sqrt{2}\beta\rho_\eps(x) \right)^{1/2} \n_\eps(x),
		\qquad \mbox{for all } x \in B(x_0,\,3r_{x_0}/4)
	\]
	where $\tau_\eps = \tau_\eps(x_0) \in \{-1,\,+1\}$ depends on $\eps$ but not on
	$x \in B(x_0,\,3r_{x_0}/4)$.
	Hence, again by~\eqref{eq:unif-conv-mod-Qeps}, 
	and after possible extraction of a subsequence,
	it follows that, as $\eps \to 0$,
	\begin{equation}\label{eq:unif-conv-Neps}
		\N_\eps \to \tau\left( 1 + \sqrt{2}\beta \right)^{1/2} \n_\star
		\qquad \mbox{uniformly in } B(x_0,\,3r_{x_0}/4),
	\end{equation}
	where $\tau \in \{-1,\,1\}$.
	On the other hand, Theorem~\ref{thm:clearing-out-M} and~\eqref{eq:close-to-wells}
	yield not only~\eqref{eq:unif-conv-compu1} but also
	\begin{equation}\label{eq:unif-conv-compu2}
		\dist(\M_\eps,\,\Sigma(\Q_\eps)) \leq
		{\rm C}_{\rm well}\left( \frac{E_\eps(\M_\eps;\,B(x_0,\,r_{x_0}))}{r_{x_0}} \right)^{1/6}
		\qquad \mbox{in } B(x_0,\,3 r_{x_0}/4).
	\end{equation}
	for any $\eps > 0$ so small
	that~\eqref{eq:unif-conv-compu1} holds. Therefore,
	combining~\eqref{eq:unif-conv-compu2} and~\eqref{eq:unif-conv-compu0},
	we obtain
	\begin{equation}\label{eq:unif-conv-compu4}
		\dist(\M_\eps,\,\Sigma(\Q_\eps)) = \o_{\eps \to 0}(1)
		\qquad \mbox{in } B(x_0,\, 3r_{x_0}/4).
	\end{equation}
	Thus, for the previously selected subsequence,
	by~\eqref{eq:unif-conv-compu4}, \eqref{eq:claim-unif-conv},
	and~\eqref{eq:unif-conv-Neps},
	\[
		\M_\eps \to \tau\left( 1 + \sqrt{2}\beta \right)^{1/2} \n_\star
		\qquad \mbox{uniformly in } B(x_0,\,3r_{x_0}/4).
	\]
	Finally, from Lemma~\ref{lemma:M*-continuous} we may assume that
	$\M_\star$ is continuous on $B(x_0,\,r_{x_0})$ and so,
	by the pointwise a.e. convergence
	$\M_\eps \to \M_\star$ as $\eps \to 0$ given by
	Theorem~\ref{mainthm:asymp}, it follows that
	\begin{equation}\label{eq:claim-unif-conv-bis}
		\M_\star(x) = \tau\left( 1 + \sqrt{2}\beta \right)^{1/2} \n_\star(x)
		\qquad \mbox{for all } x \in B(x_0,\,3r_{x_0}/4),
	\end{equation}
	for the same $\tau$ as in~\eqref{eq:unif-conv-Neps}.
	Combining~\eqref{eq:claim-unif-conv-bis} with~\eqref{eq:unif-conv-compu4},
	we obtain~\eqref{eq:analogue-of-Bethuel-7.22}.
\end{step}

\begin{step}\label{step:unif-conv-covering}
\emph{Let $K \subset U_\star$ be compact and connected. Then, there
exists a (not relabelled) subsequence such that}
	\begin{equation}\label{eq:analogue-of-Bethuel-7.26}
		\norm{\M_\eps - \M_\star}_{L^\infty(K)} \to 0,
		\qquad \mbox{as } \eps \to 0.
	\end{equation}
With~\eqref{eq:analogue-of-Bethuel-7.22} at hand, the proof
of~\eqref{eq:analogue-of-Bethuel-7.26} follows by a standard
covering argument involving the connectedness of $K$ and the
continuity of $\M_\star$ on each connected component of
$\Omega \setminus (\spt\mu_\star \cup \spt\nu_\star)$ given
by Lemma~\ref{lemma:M*-continuous}.
\end{step}

\begin{step}[Conclusion]
	Along the lines of \cite[Section~7.7]{Bethuel-AC-Acta},
	we consider the sets
	\[
		U_n := \left\{ x \in \Omega \setminus \spt\mu_\star : \dist\left(x,\,\spt\nu_{\star}\right) > \frac{1}{n} \right\}, \qquad \mbox{for } n \in \mathbb{N}.
	\]
	Then, $\overline{U_n}$ is a compact subset of $U_\star$.
	Moreover, there hold
	\[
		\overline{U_n} \subset U_{n+1} \quad \mbox{for any } n \in \mathbb{N},
		\qquad \bigcup_{n \in \mathbb{N}} U_n = U_\star.
	\]
	Letting $U_n^j$, with $j \in J_n$, be the connected components of $U_n$ and
	arguing as in \cite[Section~7.7]{Bethuel-AC-Acta}, we can first
	prove that the set $J_n$ is at most countable and, from here and
	a diagonal argument involving Step~\ref{step:unif-conv-covering},
	that $\M_\eps \to \M_\star$ uniformly in $U_n^j$, for any $j \in J_n$.
	In turn, letting now $U_j$ be any connected component of $U_\star$
	and $K \subset U_j$ be any compact set,
	the conclusion follows as in~\cite{Bethuel-AC-Acta} using
	again Step~\ref{step:unif-conv-covering} and the fact that
	the sign function $\tau$ in $\M_\star$ takes only two values and so it must
	be constant in $U_j \cap K$.
	\qedhere
\end{step}
\end{proof}

\paragraph{Tangent cone property at regular points of $\mathfrak{S}_{\star}$}
By Proposition~\ref{prop:S_star}, the set $\mathfrak{S}_{\star}$
is $\H^1$-rectifiable, with locally finite $\H^1$-measure. Thus, by
standard results in Geometric Measure Theory (see, for instance,
\cite[Chapter~3, Definition~1.4, Remark~1.5(3), and Theorem~1.6]{Simon}), there exists a
set $\mathfrak{A}_{\star} \subset \mathfrak{S}_{\star}$,
with $\H^1(\mathfrak{A}_{\star}) = 0$, so that, if
$x_0 \in \mathfrak{S}_{\star} \setminus \mathfrak{A}_{\star}$,
then
\begin{equation}\label{eq:H1-density-1}
	\lim_{r \to 0}\frac{\H^1\left( \mathfrak{S}_{\star} \cap B(x_0,\,r) \right)}{2 r} = 1,
\end{equation}
i.e.,
\[
		\lim_{r \to 0}\frac{\lambda_\star(B(x_0,\,r))}{2r}  = 1.
\]
Moreover, there exists a unit vector $\e_{x_0}$ (depending on $x_0$) with
the following \emph{tangent line property}.
For any $\theta > 0$, there holds
\begin{equation}\label{eq:regular-point}
	\lim_{r \to 0}\frac{\H^1\left( \mathfrak{S}_{\star} \cap \left(B(x_0,\,r) \setminus \mathcal{C}_{\rm one}(x_0,\,\e_{x_0},\,\theta)\right)\right)}{2 r} = 0,
\end{equation}
where $\mathcal{C}_{\rm one}(x_0,\,\e_{x_0},\,\theta)$ is the cone
defined by
\begin{equation}\label{eq:cone}
	\mathcal{C}_{\rm one}(x_0,\,\e_{x_0},\,\theta)
	= \left\{ y \in \R^2 : \abs{\e_{x_0}^\perp \cdot (y-x_0)} \leq \tan\theta \abs{\e_{x_0}\cdot(y-x_0)} \right\}.
\end{equation}
According to standard terminology in Geometric Measure Theory,
a point $x_0 \in \mathfrak{S}_{\star} \setminus \mathfrak{A}_{\star}$
is called a \emph{regular point} of $\mathfrak{S}_{\star}$ and $\e_{x_0}$
is said to be \emph{tangent to $\mathfrak{S}_\star$ at $x_0$},
see, e.g., \cite[Section~2.1]{Falconer}.

\begin{remark}\label{rk:moving-frame}
In other words, the tangent line property~\eqref{eq:regular-point} allows
us to choose, at any $x_0 \in \mathfrak{S}_\star \setminus \mathfrak{A}_\star$,
an orthonormal frame $(\e_1,\,\e_2)$ so that $\e_1$ is the
unit vector tangent to $\mathfrak{A}_\star$ at $x_0$,
i.e., $\e_1 = \e_{x_0}$.
As in \cite{Bethuel-AC-Acta}, this yields an \emph{intrinsic} moving
frame associated with $\mathfrak{S}_\star$ whose importance will be
crucial in the analysis later on.
\end{remark}

With the above results at hand, in particular with the connectedness
properties of the sets $\mathfrak{S}_{\star,K}$, and arguing
as in the proof of \cite[Proposition~1.18]{Bethuel-AC-Acta},
we obtain a relevant strengthening of the
tangent line property at regular points of $\mathfrak{S}_{\star}$,
i.e., the \emph{tangent cone property} below.
\begin{prop}\label{prop:analogue-of-Bethuel-prop-4}
	Let $x_0 \in \mathfrak{S}_{\star}$ be any regular point of
	$\mathfrak{S}_{\star}$. Given any $\theta > 0$, there exists
	a radius $R_{\rm cone}(\theta,\,x_0)$ such that
	\begin{equation}\label{eq:Bethuel-prop-4}
		\mathfrak{S}_{\star} \cap B(x_0,\,r) \subset
		\mathcal{C}_{\rm one}(x_0,\,\e_{x_0},\,\theta),
		\qquad \mbox{for any } 0 < r \leq R_{\rm cone}(\theta,\,x_0)
	\end{equation}
\end{prop}

\begin{proof}
	Let $x_0 \in \mathfrak{S}_\star$ be any given regular point of
	$\mathfrak{S}_\star$. Since $x_0$ must be contained in
	some compact set $K \subset \Omega \setminus \spt\mu_\star$,
	we can find $r_{x_0} > 0$ such that
	$\dist(\{x_0\},\,\spt\mu_\star) = 2 r_{x_0} > 0$, and so
	$B(x_0,\,2 r_{x_0}) \subset \Omega \setminus \spt\mu_\star$.
	Let, for instance, $K_{x_0} := \overline{B(x_0,\,3r_{x_0}/2)}$.
	Then, $K_{x_0}$ is a compact set contained in $\Omega \setminus \spt\mu_\star$ and
	the proof presented in \cite[Section~7.8]{Bethuel-AC-Acta},
	being based only on general arguments in Geometric Measure Theory
	(in particular, on~\eqref{eq:regular-point}) and on the connectedness
	properties of $\mathfrak{S}_{\star,K_{x_0}}$, which are the same as
	in~\cite{Bethuel-AC-Acta}, carries over without modifications.
\end{proof}

A crucial consequence of Proposition~\ref{prop:analogue-of-Bethuel-prop-4}
is Proposition~\ref{prop:analogue-of-Bethuel-8.4} below. In order to state
Proposition~\ref{prop:analogue-of-Bethuel-8.4}, we need to introduce some notation, analogous to that in \cite[Section~8.1]{Bethuel-AC-Acta}.

Given $x_0 = \left( x_{0,1},\,x_{0,2,} \right) \in \Omega \setminus \spt\mu_\star$
and $\rho > 0$ such that $B(x_0,\,2\rho) \subset \Omega \setminus \spt\mu_\star$, we
denote $Q_\rho(x_0)$ the closed square defined by
\begin{equation}\label{eq:square}
	Q_\rho(x_0) = \mathcal{I}_\rho(x_{0,1}) \times \mathcal{I}_\rho(x_{0,2}),
\end{equation}
where
\begin{equation}\label{eq:interval}
	\mathcal{I}_\rho(s) = [s-\rho,\,s+\rho] \qquad \mbox{for } s \in (0,\,\rho).
\end{equation}
Next, we let
\[
	\mathcal{R}_\rho(x_0) := \mathcal{I}_\rho(x_{0,1}) \times \mathcal{I}_{\frac{3\rho}{4}}(x_{0,2}),
\]
so that $Q_\rho(x_0) \setminus \mathcal{R}_\rho(x_0)$ is the union
of two disjoint rectangles. The following condition plays a pivotal
role in most arguments in \cite[Section~8]{Bethuel-AC-Acta}:
\begin{equation}\label{eq:Bethuel-8.8}
	\nu_\star\left( \overline{Q_\rho(x_0) \setminus \mathcal{R}_\rho(x_0)} \right) = 0.
\end{equation}
If~\eqref{eq:Bethuel-8.8} holds, then $\nu_\star$ is concentrated,
locally near $x_0$, in a neighbourhood of the segment
$(x_0 - \rho \e_1,\,x_0 + \rho \e_1)$.
In \cite[Proposition~8.4]{Bethuel-AC-Acta}, it is proved
that~\eqref{eq:Bethuel-8.8} holds at any regular point
of $\mathfrak{S}_\star$. This property holds also in our context
(by exactly the same argument, in fact).
\begin{prop}\label{prop:analogue-of-Bethuel-8.4}
	Assume that $x_0 \in \mathfrak{S}_\star \setminus \mathfrak{A}_\star$.
	Then, there exists $\rho_0 > 0$ such that~\eqref{eq:Bethuel-8.8} is
	satisfied for any $0 < \rho \leq \rho_0$.
\end{prop}

\begin{proof}
	The argument in~\cite{Bethuel-AC-Acta}, which relies just
	on choosing an orthonormal basis at $x_0$ as in
	Remark~\ref{rk:moving-frame} and on the tangent cone property
	in Proposition~\ref{prop:analogue-of-Bethuel-prop-4}, carries
	over without modifications.
\end{proof}

\subsection{The limiting potential energy $\zeta_\star$}
\label{sec:zeta*}
This short section will serve us to introduce a fundamental
object in the following course, 
the \emph{limiting potential energy} $\zeta_\star$.

To begin with, following~\cite[(1.14)]{Bethuel-AC-Acta},
for any $\eps > 0$, we define the potential energy densities
\begin{equation}\label{eq:def-zeta-eps}
	\zeta_\eps := \frac{V(\,\cdot\,,\,\M_\eps)}{\eps},
\end{equation}
dually seen as measures in $\Omega$. In view of
the
bound~\eqref{eq:Eeps-bounded}, it follows immediately that
\begin{equation}\label{eq:zeta-eps-bounded}
	\zeta_\eps(\Omega) \leq E_\eps(\M_\eps;\,\Omega)
	\leq \mathcal{E}_0
\end{equation}
independently of $\eps > 0$,
so that we can find a (not-relabelled)
subsequence and a non-negative Radon measure $\zeta_\star$
on $\Omega$ such that, as $\eps \to 0$,
\begin{equation}\label{eq:zeta-eps-to-zeta*}
\zeta_\eps \rightharpoonup^* \zeta_\star
\end{equation}
weakly* as measures in $\Omega$.

In the scalar theory,
as direct consequences of the positivity of the discrepancy
function in the entire case~\cite{Modica} or, in the case
of bounded domains, of its asymptotic vanishing~\cite{HutchinsonTonegawa}
and of the resulting monotonicity formula for $\nu_\star$,
one has both that $\zeta_\star$
controls $\nu_\star$ and that $\zeta_\star$ is absolutely
continuous with respect to $\lambda_\star$. In fact, the
equipartition property $\nu_\star = 2 \zeta_\star$ holds.
In the vectorial case, no monotonicity formula is known for
$\nu_\star$. In addition, the discrepancy should be not
expected to have a sign (in the entire case, a counterexample to the positivity of
the discrepancy is presented in \cite{Smyrnelis}).
For these reasons, the vectorial theory developed in \cite{Bethuel-AC-Acta}
goes along a different path and it based on exploiting
the relationship between $\zeta_\star$ and the \emph{limiting Hopf differential}
$\omega_\star$,
to be introduced in the next section. This relationship can be proved to hold
also in our `perturbed' setting thanks to the refined energy estimates that
we have already obtained at $\eps$-level and the strong convergence of
$\Q_\eps$ to a harmonic map $\Q_\star$ away from $\spt\mu_\star$, and it is
key to further developments.
In particular, in Section~\ref{sec:mon-zeta*}
we will prove the monotonicity
of the function
$r \mapsto \frac{1}{r}\zeta_\star(B(x_0,\,r))$ for any $x_0 \in \Omega$,
which immediately entails the absolute continuity of $\zeta_\star$
with respect to $\lambda_\star$. In Section~\ref{sec:upper-bounds}
we will instead recover the equivalence between
$\zeta_\star$ and $\nu_\star$ on $\mathfrak{S}_\star$.
Finally, in Section~\ref{sec:S*}, we will use Equation~\eqref{eq:omega*-zeta*}
below to obtain refined structure properties for $\mathfrak{S}_\star$.

\subsection{The limiting Hopf differential}
\label{sec:hopf}

In this section, we show that testing the Pohozaev
inequality in Lemma~\ref{lemma:stressenergy} against smooth vector field
with support in $\Omega \setminus \spt\mu_\star$
and passing to the limit as $\eps \to 0$ produces
a limiting measure $\omega_\star$ that
will be called, adopting the terminology in 
\cite{Bethuel-AC-Acta},
the \emph{limiting Hopf differential}.
Proposition~\ref{prop:hopf}
below relates the limiting Hopf differential with the limiting
potential energy measure $\zeta_\star$ and it is the analogue,
in our context, of \cite[Lemma~1.19]{Bethuel-AC-Acta}. Crucially,
thanks to~\eqref{eq:hopf-Q*} below,
only the contribution $\omega_\star^{\rm M}$
to the Hopf differential coming from the $\M$-component appears
in Equation~\eqref{eq:omega*-zeta*}.
Combined with a natural decomposition of $\omega_\star^{\rm M}$,
\eqref{eq:hopf-m*} below, Proposition~\ref{prop:hopf}
will be key to prove, in the forthcoming sections,
the aforementioned further properties of $\zeta_\star$ and
its precise relationship with the limiting energy density $\nu_\star$.

\vskip5pt
\noindent
For a start, we notice that,
in view of the global bound~\eqref{eq:M_bound},
the functions $\eps \partial_j \M_\eps \cdot \partial_k \M_\eps$,
for $(j,\,k) \in \{1,\,2\}^2$, are bounded in $L^1(\Omega)$.
Therefore, 
as $\eps \to 0$ (after possible extraction of a subsequence),
we have
\begin{equation}\label{eq:m*}
	\eps \partial_j \M_\eps \cdot \partial_k \M_\eps \rightharpoonup^*
	\mathbbm{m}_{\star,j,k}, \qquad \mbox{for any } (j,\,k) \in \{1,\,2\}^2,
\end{equation}
weakly* in the sense of measures in $\Omega$, where
$\mathbbm{m}_{\star,j,k}$ is a bounded signed Radon measure
on $\Omega$ for any $(j,\,k) \in \{1,\,2\}^2$.

On the other hand, for any compact set
$K \subset \Omega \setminus \spt\mu_\star$,
the maps $\Q_\eps$ converge strongly in $W^{1,2}(K)$ to $\Q_\star$
(see item~\ref{item:mainthm-asymp-strong-conv} of Theorem~\ref{mainthm:asymp}) and hence
we have
\begin{equation}\label{eq:q*}
	\partial_j \Q_\eps \cdot \partial_k \Q_\eps
	\to \partial_j \Q_\star \cdot \partial_k \Q_\star
	\qquad \mbox{for any } (j,\,k) \in \{1,\,2\}^2,
\end{equation}
strongly in $L^1(K)$ as $\eps \to 0$.

Now, we define the following quantities:
\begin{align}
	& \omega_\eps^{\rm Q} := \abs{\partial_1 \Q_\eps}^2 - \abs{\partial_2 \Q_\eps}^2
	- 2 i \partial_1 \Q_\eps \cdot \partial_2 \Q_\eps \label{eq:hopf-Q}\\
	& \omega_\eps^{\rm M} := \abs{\partial_1 \M_\eps}^2 - \abs{\partial_2 \M_\eps}^2
	- 2 i \partial_1 \M_\eps \cdot \partial_2 \M_\eps \label{eq:hopf-M} \\
	& \omega_\eps := \omega_\eps^{\rm Q} + \eps \omega_\eps^{\rm M} \label{eq:hopf}.
\end{align}
According to the terminology in \cite{Bethuel-AC-Acta}, we will call
$\omega_\eps^{\rm Q}$, $\omega_\eps^{\rm M}$, and $\omega_\eps$
\emph{Hopf differentials (at the $\eps$-level)}. (Although such terminology
is usually referred to the above objects multiplied by the complex differential
${\d}z$; see, e.g., \cite[Chapter~4]{Helein}.)

Since the sequence $\{\eps \omega_\eps^{\rm M}\}$ is bounded in $L^1(\Omega)$,
we can find a (not relabelled)
subsequence and a locally bounded Radon measure
$\eps \omega_\star^{\rm M}$ such that, as $\eps \to 0$,
\begin{equation}\label{eq:omega-star-M}
	\eps \omega_\eps^{\rm M} \rightharpoonup^* \omega_\star^{\rm M}
\end{equation}
weakly* in the sense of measures in $\Omega$.
On the other hand,~\eqref{eq:m*} implies
\begin{equation}\label{eq:hopf-m*}
	\omega_\star^{\rm M} = (\mathbbm{m}_{\star,1,1}-\mathbbm{m}_{\star,2,2}) -
	2i \mathbbm{m}_{\star,1,2}
\end{equation}
as measures in $\Omega$.

Next, denoting
\[
	\omega_\star^{\rm Q} := \abs{\partial_1 \Q_\star}^2 - \abs{\partial_2 \Q_\star}^2
	- 2 i \partial_1 \Q_\star \cdot \partial_2 \Q_\star
\]
the Hopf differential of $\Q_\star$, we see that~\eqref{eq:q*} implies
\begin{equation}\label{eq:conv-hopf-Q*}
	\omega_\eps^{\rm Q} \to \omega_\star^{\rm Q}
\end{equation}
as $\eps \to 0$,
both weakly* in the sense of measures and pointwise a.e.
in $\Omega \setminus \spt\mu_\star$.
On the other hand, by a classical result in harmonic maps
(see, e.g., \cite[Chapter~3, p.~28]{Helein}), the
Hopf differential of any harmonic map from an open set in
$\R^2 \simeq \C$ is a holomorphic function, and therefore we have
\begin{equation}\label{eq:hopf-Q*}
	\frac{\partial \omega_\star^{\rm Q}}{\partial {\bar{z}}} \equiv 0
\end{equation}
pointwise in $\Omega \setminus \spt\mu_\star$.

From the above, it follows that $\{\omega_\eps\}$ is a
bounded sequence in $L^1(K)$ for any $K \subset \Omega\setminus \spt\mu_\star$,
so that there exist a (not relabelled) subsequence and a locally bounded Radon measure
$\omega_\star$ in $\Omega \setminus \spt\mu_\star$ such that, as $\eps \to 0$,
\begin{equation}\label{eq:hopf*}
	\omega_\eps \rightharpoonup^* \omega_\star
\end{equation}
weakly* as measures in $\Omega \setminus \spt\mu_\star$.

We are now ready to prove the main result of this section, relating
$\omega_\star$, $\omega^{\rm M}_\star$, and $\zeta_\star$.
\begin{prop}\label{prop:hopf}
	Let $\omega_\star$, $\omega^{\rm M}_\star$, and $\zeta_\star$ be the
	bounded Radon measures given by~\eqref{eq:hopf*},~\eqref{eq:omega-star-M}
	and~\eqref{eq:zeta-eps-to-zeta*}, respectively.
	Then, for any $X \in C^1_c(\Omega \setminus \spt\mu_\star,\,\C)$,
	we have
	\begin{equation}\label{eq:hopf=hopf-M}
		\ip{\omega_\star}{\frac{\partial X}{\partial \bar{z}}} =
		\ip{\omega_\star^{\rm M}}{\frac{\partial X}{\partial \bar{z}}}
	\end{equation}
	and
	\begin{equation}\label{eq:omega*-zeta*}
		\ip{\omega_\star^{\rm M}}{\frac{\partial X}{\partial \bar{z}}}
		= \ip{2 \zeta_\star}{\frac{\partial X}{\partial z}}.
	\end{equation}
\end{prop}

\begin{proof}
	The equality~\eqref{eq:hopf=hopf-M}
	is an immediate consequence of~\eqref{eq:hopf-Q*},
	so we focus on~\eqref{eq:omega*-zeta*}.

	\setcounter{step}{0}
	\begin{step}
	\emph{We prove that
		\begin{equation}\label{eq:omega*-zeta*-Re}
		\Re\left( \ip{\omega_\star^{\rm M}}{\frac{\partial X}{\partial \bar{z}}}\right)
		= \ip{2 \zeta_\star}{\Re\left( \frac{\partial X}{\partial z} \right)}.
	\end{equation}}

	Reasoning as in \cite[(3.20)]{Bethuel-AC-Acta}, 
	by the Pohozaev identity~\eqref{stren4} we obtain
	\begin{equation}\label{eq:hopf-compu1}
	\begin{split}
	\int_\Omega \Re\left( \omega_\eps \frac{\partial X}{\partial {\bar{z}}} \right) \,{\d}x
	&= \frac{1}{\eps^2}\int_\Omega f_\eps(\Q_\eps,\,\M_\eps)\div X \,{\d}x\\
	&= \frac{2}{\eps^2}\int_\Omega f_\eps(\Q_\eps,\,\M_\eps) \Re\left( \frac{\partial X}{\partial z}\right) \,{\d}x
	\end{split}
	\end{equation}
	for any $X \in C^1_c(\Omega; \,\C)$ and for any $\eps > 0$.
	Take a compact set
	$K \subset \Omega \setminus \spt\mu_\star$ and assume that
	$\spt X \subseteq K$. 
	Recalling Lemma~\ref{lemma:limit-f-minus-V} and that, as $\eps \to 0$,
	\[
		\frac{1}{\eps}V(\,\cdot\,,\,\M_\eps) \rightharpoonup^* \zeta_\star
	\]
	weakly* as measures in $K$, we obtain
	\begin{equation}\label{eq:hopf-compu4}
		\frac{1}{\eps^2}f_\eps(\Q_\eps,\,\M_\eps) \rightharpoonup^* \zeta_\star
	\end{equation}
	as $\eps \to 0$ and weakly* as measures in $K$. Thus,
	using~\eqref{eq:hopf*},~\eqref{eq:hopf=hopf-M}, and~\eqref{eq:hopf-compu4}
	to pass to the limit in~\eqref{eq:hopf-compu1},
	we get~\eqref{eq:omega*-zeta*-Re}.
	\end{step}

	\begin{step}[Proof of~\eqref{eq:omega*-zeta*}]
		Using $iX$ as a test function in~\eqref{eq:omega*-zeta*-Re},
		we obtain
		\begin{equation}\label{eq:omega*-zeta*-Im}
		\Im\left( \ip{\omega_\star^{\rm M}}{\frac{\partial X}{\partial \bar{z}}}\right)
		= \ip{2 \zeta_\star}{\Im\left( \frac{\partial X}{\partial z} \right)}.
		\end{equation}
		Taking the sum of~\eqref{eq:omega*-zeta*-Re} with~\eqref{eq:omega*-zeta*-Im},
		the conclusion follows.
		\qedhere
	\end{step}
\end{proof}

As an immediate consequence of~\eqref{eq:omega*-zeta*},
 we have that $\mathbbm{m}_{\star,1,2}$ satisfies
\emph{modified Cauchy-Riemann relations}~\eqref{eq:CR} below
 (cf. \cite[(1.72)]{Bethuel-AC-Acta}).

\begin{corollary}\label{cor:CR}
	The equations
	\begin{equation}\label{eq:CR}
	\begin{cases}
		&\frac{\partial}{\partial x_2} (2\mathbbm{m}_{\star,1,2}) = \frac{\partial}{\partial x_1}( 2 \zeta_\star - \mathbbm{m}_{\star,1,1} + \mathbbm{m}_{\star,2,2}) \\
		&\frac{\partial}{\partial x_1} (2\mathbbm{m}_{\star,1,2}) = \frac{\partial}{\partial x_2}( 2 \zeta_\star + \mathbbm{m}_{\star,1,1} - \mathbbm{m}_{\star,2,2})
	\end{cases}
	\end{equation}
	hold in the sense of distributions in $\Omega \setminus \spt\mu_\star$.
\end{corollary}

\begin{proof}
	Equations~\eqref{eq:CR} follow immediately from~\eqref{eq:omega*-zeta*}
	and the definition of Wirtinger's operators
	\[
		\frac{\partial}{\partial z} =
		\frac{1}{2}\left( \frac{\partial}{\partial x_1} - i \frac{\partial}{\partial x_2} \right),
		\qquad
		\frac{\partial}{\partial \bar{z}} =
		\frac{1}{2}\left( \frac{\partial}{\partial x_1} + i \frac{\partial}{\partial x_2} \right).
	\]
\end{proof}

\begin{remark}[Changing coordinates in the limiting Hopf differential]
The expressions~\eqref{eq:hopf-Q}, \eqref{eq:hopf-M}, and~\eqref{eq:hopf}
for $\omega_\eps^{\rm Q}$, $\omega_\eps^{\rm M}$, and $\omega_\eps$, respectively,
depend on the chosen orthonormal basis. However, it is well-known that, under conformal maps, they transform simply by multiplication by the conformal
factor (in fact, the same quantities multiplied by the complex differential ${\d}z$ are invariant under conformal
maps, see e.g. \cite[Chapter~4]{Helein}).
Consequently, for any compact set $K \subset \Omega \setminus \mu_\star$, the
boundedness of the sequences $\{\omega_\eps^{\rm Q}\}$ and
$\{\eps \omega_\eps^{\rm M}\}$, and hence of $\{\omega_\eps\}$, is a property invariant
under conformal maps. In turn, the same is true for the existence of their limits;
moreover, the explicit expression for the limiting maps change only by
multiplication for the conformal factor. In particular, this is true for
coordinate changes (from any orthonormal moving frame to any other).
Detailed computations, accounting for all cases that will be used
thereafter, are provided in \cite[Section~8.9.3]{Bethuel-AC-Acta}.
\end{remark}

\subsection{Properties near good points of $\mathfrak{S}_\star$}
\label{sec:good-points}
With Proposition~\ref{prop:hopf} at hand, not only we can replicate the
arguments in \cite[Section~8]{Bethuel-AC-Acta}, but actually
we could also \emph{apply them directly}. 
However, the proof of the monotonicity 
of the (rescaled) limiting potential energy in Proposition~\ref{prop:analogue-of-Bethuel-Prop-6} in the next section requires some adaptation with respect to~\cite[Section~9]{Bethuel-AC-Acta}. The reason is that, in order to prove that $\zeta_\star$ does not charge 
$\spt\mu_\star$ (i.e., item~\ref{item:Z/r-monotone} of Proposition~\ref{prop:analogue-of-Bethuel-Prop-6}), which is a key ingredient in the proof of the balance law~\eqref{eq:first-variation-intro} in Theorem~\ref{thm:first-variation}, 
we really have to go through the argument of \cite[Proposition~1.23]{Bethuel-AC-Acta}, because we need it to be developed on 
\emph{annuli} rather than on full balls (as it is instead done in \cite{Bethuel-AC-Acta}). However, while stressing that using annuli rather balls is absolutely necessary to our purposes, we also emphasise that only notational modifications are needed with respect to~\cite[Section~9]{Bethuel-AC-Acta}, since a careful reading shows that the argument presented there is \emph{already valid on annuli}, despite being written in terms of full balls.   
Nonetheless, this circumstance forces us to introduce some notation and to retrace in this section and in the next ones the main steps of
Bethuel's arguments. Such arguments are made-up by a remarkable
(and quite long) combination of elementary pieces that rely, essentially,
on testing the equation~\eqref{eq:omega*-zeta*} against suitable test vector fields and
on identifying a class of \emph{good points}
(see Definition~\ref{def:good-points} below) having full $\lambda_\star$-measure
at which remarkable relationships between the various measures
at stake exist. The main result of this section is
Proposition~\ref{prop:analogue-of-Bethuel-Prop5}, which extends
\cite[Proposition~1.20]{Bethuel-AC-Acta} to our situation.

\paragraph{Decomposition of the relevant measures with respect to
$\lambda_\star$}
Similarly to as in \cite{Bethuel-AC-Acta}, the first
step in the analysis consists in decomposing $\nu_\star$, $\zeta_\star$,
and each of the measures $\mathbbm{m}_{\star,j,k}$ into their
absolutely continuous and singular part
with respect to the measure $\lambda_\star = \H^1 \mres \mathfrak{S}_\star$:
\begin{align}\label{eq:ac-sing}
	&\nu_\star = \nu_\star^{\rm ac} + \nu_\star^{\rm sing},
	&& \nu_\star^{\rm ac} \perp \nu_\star^{\rm sing}, \\
	&\zeta_\star = \zeta_\star^{\rm ac} + \zeta_\star^{\rm sing},
	&& \zeta_\star^{\rm ac} \perp \zeta_\star^{\rm sing}, \\
	&\mathbbm{m}_{\star,j,k} =
	\mathbbm{m}_{\star,j,k}^{\rm ac} + \mathbbm{m}_{\star,j,k}^{\rm sing},
	&&
	\mathbbm{m}_{\star,j,k}^{\rm ac} \perp \mathbbm{m}_{\star,j,k}^{\rm sing},
\end{align}
for any $(j,\,k) \in \{1,\,2\}^2$. We can write
\begin{align}
	& \nu_\star^{\rm ac} = \mathfrak{e}_\star \lambda_\star, \label{eq:density-nu} \\\
	& \zeta_\star^{\rm ac} = \mathfrak{v}_\star \lambda_\star,\label{eq:density-zeta}\\
	&\mathbbm{m}_{\star,j,k}^{\rm ac} = \mathfrak{m}_{\star,j,k} \lambda_\star, \qquad \mbox{for } (j,\,k) \in \{1,\,2\}^2\label{eq:density-m},
\end{align}
so that $\mathfrak{e}_\star$, $\mathfrak{v}_\star$, and $\mathfrak{m}_\star$
denote the densities of $\nu_\star^{\rm ac}$, $\zeta_\star^{\rm ac}$, and
$\mathbbm{m}_{\star,j,k}^{\rm ac}$ with respect to $\lambda_\star$.

\begin{remark}
	For later purposes, we observe that the values of the
	densities $\mathfrak{m}_{\star,j,k}$ strongly depend on the
	choice of the (Cartesian) orthonormal frame in their
	definition~\eqref{eq:m*}. In the sequel, it will be convenient
	to have more intrinsic objects at disposal and therefore,
	following the discussion in \cite[Section~1.3]{Bethuel-AC-Acta},
	we define the `intrinsic versions' of the functions $\mathfrak{m}_{\star,j,k}$,
	i.e., their expressions in the moving frame associated
	with $\mathfrak{S}_\star$ introduced in Remark~\ref{rk:moving-frame}
	above. To this purpose, given any
	$x_0 \in \mathfrak{S}_\star \setminus \mathfrak{A}_\star$ and choosing
	the orthonormal frame $(\e_1,\,\e_2)$ so that $\e_1 = \e_{x_0}$, we define
	\begin{equation}\label{eq:m*-perp-par}
		\mathfrak{m}_{\star,\perp,\perp}(x_0) = \mathfrak{m}_{\star,2,2}(x_0), \quad
		 \mathfrak{m}_{\star,||,||}(x_0) = \mathfrak{m}_{\star,1,1}(x_0),
\quad
		 \mathfrak{m}_{\star,\perp,||}(x_0) = \mathfrak{m}_{\star,1,2}(x_0).
	\end{equation}
	Then, letting $x_0$ vary in $\mathfrak{S}_\star \setminus \mathfrak{A}_\star$
	gives rise to the `intrinsic' measures
	\begin{equation}\label{eq:m*-perp-par-bis}
		\mathbbm{m}_{\star,\perp,\perp} = \mathfrak{m}_{\star,\perp,\perp} \lambda_\star, \quad
		\mathbbm{m}_{\star,||,||} = \mathfrak{m}_{\star,||,||} \lambda_\star, \quad
		\mathbbm{m}_{\star,\perp,||} = \mathfrak{m}_{\star,\perp,||} \lambda_\star,
	\end{equation}
	whose importance will be clear later on, in light of
	Theorem~\ref{thm:analogue-of-Bethuel-Thm5}
	and Theorem~\ref{thm:stationary-varifold} below.
\end{remark}

The crux of the following sections will
be proving that the singular parts of the measures above are actually zero,
while the absolutely continuous parts enjoy suitable bounds. To this purpose,
the first step, carried out below, is ruling out certain bad, $\lambda_\star$-null
sets from $\mathfrak{S}_\star$. Then, the decisive step, performed in
Section~\ref{sec:mon-zeta*}, will be showing the monotonicity of the rescaled measure
$\zeta_\star$, \emph{at any $x_0 \in \Omega$}.

\paragraph{Exceptional points and good points}
In Section~\ref{sec:lower-bounds-nu*}, we defined the $\lambda_\star$-null
set $\mathfrak{A}_\star$ of points at
which $\mathfrak{S}_\star = \spt\nu_\star \setminus \spt\mu_\star$
does not have density $1$. Now, following \cite{Bethuel-AC-Acta},
we define two other classes of exceptional points. Precisely,
in view of~\eqref{eq:ac-sing}, we may find a subset $\mathfrak{B}_\star$
of $\mathfrak{S}_\star$ so that
\begin{equation}\label{eq:B*-1}
	\lambda_\star(\mathfrak{B}_\star) = 0
\end{equation}
and
\begin{equation}\label{eq:B*-2}
	\nu_\star^{\rm sing}(\mathfrak{S}_\star \setminus \mathfrak{B}_\star) = 0, \qquad
	\zeta_\star^{\rm sing}(\mathfrak{S}_\star \setminus \mathfrak{B}_\star) = 0, \qquad
	\mathbbm{m}_{\star,j,k}^{\rm sing}(\mathfrak{S}_\star \setminus \mathfrak{B}_\star)
	= 0,
\end{equation}
for any $(j,\,k) \in \{1,\,2\}^2$.

In addition to $\mathfrak{A}_\star$ and $\mathfrak{B}_\star$,
we define the set $\mathfrak{C}_\star \subset \mathfrak{S}_\star$
as the complementary set of Lebesgue's points for the densities of the measures
$\nu_\star^{\rm ac}$, $\zeta_\star^{\rm ac}$, and $\mathbbm{m}_{\star,j,k}^{\rm ac}$
for $(j,\,k) \in \{1,\,2\}^2$ with respect to the measure $\lambda_\star$.
By Besicovitch's differentiation theorem and the 
decomposition~\eqref{eq:ac-sing}, there holds
\begin{equation}\label{eq:C*}
	\lambda_\star(\mathfrak{C}_\star) = 0.
\end{equation}
Finally, like in \cite[(1.76)]{Bethuel-AC-Acta}, we define the set
of \emph{exceptional points}
\begin{equation}\label{eq:E*}
	\mathfrak{E}_\star := \mathfrak{A}_\star \cup \mathfrak{B}_\star \cup \mathfrak{C}_\star .
\end{equation}
Since each of the sets on the right-hand side of~\eqref{eq:E*} is null for
the measure $\lambda_\star$, we have
\begin{equation}\label{eq:E*-bis}
	\lambda_\star(\mathfrak{E}_\star) = 0.
\end{equation}

\begin{definition}\label{def:good-points}
	The points $x_0 \in \mathfrak{S}_\star \setminus \mathfrak{E}_\star$
	will be called the \emph{good points} of $\mathfrak{S}_\star$.
\end{definition}

\begin{remark}\label{rk:density-nu*-good-point}
	In particular, for any good point $x_0$, the density
	$\D_{\lambda_\star}(\nu_\star)(x_0)$
	of $\nu_\star$ with respect to $\lambda_\star$ at $x_0$ exists and satisfies
	\begin{align}
		& \D_{\lambda_\star}(\nu_\star)(x_0) = \D_{\lambda_\star}(\nu_\star^{\rm ac})(x_0) = \lim_{r \to 0} \frac{\nu_\star(B(x_0,\,r))}{2r} < +\infty, \label{eq:density-nu*-good-point-ac}\\
		& \D_{\lambda_\star}(\nu_\star^{\rm sing})(x_0) = 0\label{eq:density-nu*-good-point-sing}.
	\end{align}
\end{remark}

\paragraph{Eliminating derivatives along the transversal direction to the moving frame}
Let $x_0 \in \mathfrak{S}_\star \setminus \mathfrak{A}_\star$ and choose
an orthonormal basis at $x_0$ as described in Remark~\ref{rk:moving-frame}.
Let $Q_\rho(x_0)$ be a (closed) square as in~\eqref{eq:square}. Along the lines of
\cite[Section~1.6]{Bethuel-AC-Acta}, we define the \emph{localised measures}
\begin{equation}\label{eq:localised-measures}
	\widetilde{\mathbbm{m}}_{\star,j,k} = \one_{Q_\rho(x_0)} \mathbbm{m}_{\star, j,k}
	\quad \mbox{for } (j,\,k) \in \{1,\,2\}^2,
	\qquad
	\widetilde{\zeta}_\star = \one_{Q_\rho(x_0)} \zeta_\star,
\end{equation}
where $\one_{Q_\rho(x_0)}$ denotes the characteristic function of the
square $Q_\rho(x_0)$.
Denoting $\mathbb{P}$ the orthogonal projection onto the tangent line
$\gamma_{x_0} := \{ x_0 + s \e_1 \,:\, s \in \R \}$, we let
\[
	\widetilde{\mathbbm{m}}_{\star,j,k}^1 := \mathbb{P}_\sharp(\widetilde{\mathbbm{m}}_{\star,j,k}),\qquad
	\widetilde{\zeta}_\star^1 := \mathbb{P}_\sharp(\widetilde{\zeta}_\star)
\]
be the push-forward of the localised measures~\eqref{eq:localised-measures}
with respect to $\mathbb{P}$. As in~\cite[(1.80)]{Bethuel-AC-Acta}, we
introduce the measures on the interval $\mathcal{I}_\rho(x_{0,1})$
(defined as in~\eqref{eq:interval})
\begin{equation}\label{eq:J-L-N}
	\begin{cases}
		\mathcal{J}_{x_0,\rho} = \mathcal{J}_{\rho}
		= \widetilde{\mathbbm{m}}_{\star,1,2}^1, \\
		\mathcal{L}_{x_0,\rho} = \mathcal{L}_{\rho}
		= 2 \widetilde{\zeta}_\star^1 - \widetilde{\mathbbm{m}}_{\star,1,1}^1 + \widetilde{\mathbbm{m}}_{\star,2,2}^1, \\
		\mathcal{N}_{x_0,\rho} = \mathcal{N}_{\rho}
		= 2 \widetilde{\zeta}_\star^1 + \widetilde{\mathbbm{m}}_{\star,1,1}^1 - \widetilde{\mathbbm{m}}_{\star,2,2}^1.\\
	\end{cases}
\end{equation}
The properties of the one-dimensional measures~\eqref{eq:J-L-N} are given by
\cite[Proposition~8.2, Proposition~8.4, Proposition~8.6, Proposition~8.7, Proposition~8.11, and Proposition~8.12]{Bethuel-AC-Acta}. Although the proofs of
\cite[Proposition~8.2, Proposition~8.4, and Proposition~8.6]{Bethuel-AC-Acta}
involve several computations and occupy a few pages, they
are based only on the validity of the condition~\eqref{eq:Bethuel-8.8} and on testing the Equation~\eqref{eq:omega*-zeta*} against suitable fields, of the form
\[
	X_f(x_1,\,x_2) = f_1(x_1)f_2(x_2)\e_i
\]
for either $i = 1$ or $i = 2$ and appropriate choices of the
functions $f_j \in C_c^\infty(\mathcal{I}_\rho(x_{0,j}))$,
for $j \in \{1,\,2\}$. (Specifically, we address the reader to
\cite[(8.18), (8.24), (8.28), (8.31)]{Bethuel-AC-Acta}.)
Therefore, these propositions carry over to our case
without modifications and thus, for the sake
of brevity, we avoid repeating the arguments; the reader is referred to
\cite[Sections from 8.1 to 8.7]{Bethuel-AC-Acta}
for fully detailed proofs.

About \cite[Proposition~8.7]{Bethuel-AC-Acta}, the first part of
its statement claims the validity of~\eqref{eq:Bethuel-8.8} at
points of $\mathfrak{S}_\star \setminus \mathfrak{A}_\star$
and corresponds precisely to Proposition~\ref{prop:analogue-of-Bethuel-8.4}.
Combined with \cite[Proposition~8.2]{Bethuel-AC-Acta}, it
yields immediately that the measures $\mathcal{J}_{\rho}$,
$\mathcal{L}_{\rho}$ are proportional to the Lebesgue measure
on $\mathcal{I}_\rho(x_{0,1})$ and satisfy the differential relations
\cite[(8.15), (8.16)]{Bethuel-AC-Acta} (which are part of
the statement of \cite[Proposition~8.2]{Bethuel-AC-Acta} and are, once again,
direct consequences of~\eqref{eq:omega*-zeta*}).

Assume now the stricter condition
$x_0 \in \mathfrak{S}_\star \setminus \mathfrak{E}_\star$, so
that $x_0$ is a Lebesgue point for the absolutely continuous part (with
respect to $\lambda_\star$) of the measures $\nu_\star$,
$\zeta_\star$, and $\mathbbm{m}_{\star,j,k}$, for $(j,\,k) \in \{1,\,2\}^2$.
This means that the densities $\mathfrak{e}_\star$, $\mathfrak{v}_\star$, $\mathfrak{m}_{\star,j,k}$ in~\eqref{eq:density-nu},~\eqref{eq:density-zeta},~\eqref{eq:density-m} satisfy
\[
\begin{split}
	&\lim_{r \to 0} \frac{1}{r}\int_{\mathfrak{S}_\star \cap B(x_0,r)} \abs{\mathfrak{e}_\star(\tau) - \mathfrak{e}_\star(x_0)} \,{\d}\tau = 0, \\
	&\lim_{r \to 0} \frac{1}{r}\int_{\mathfrak{S}_\star \cap B(x_0,r)} \abs{\mathfrak{v}_\star(\tau) - \mathfrak{v}_\star(x_0)} \,{\d}\tau = 0, \\
	&\lim_{r \to 0} \frac{1}{r}\int_{\mathfrak{S}_\star \cap B(x_0,r)} \abs{\mathfrak{m}_{\star,j,k}(\tau) - \mathfrak{m}_{\star,j,k}(x_0)} \,{\d}\tau = 0 \quad
	\mbox{for } (j,\,k) \in \{1,\,2\}^2,
\end{split}
\]
whence the existence of some $K(x_0) > 0$ so that
\begin{equation}\label{eq:nu-ac-est}
	\nu_\star^{\rm ac}(B(x_0,\,r)) \leq K(x_0) r, \qquad
	\mbox{for any } 0 < r \leq \rho_0.
\end{equation}
Arguing word-for-word as in \cite[Proposition~8.11]{Bethuel-AC-Acta},
one obtains that every point $x_0 \in \mathfrak{S}_\star \setminus \mathfrak{E}_\star$
is a Lebesgue point for the densities $J_{\rho_0}$, $L_{\rho_0}$, and $N_{\rho_0}$
associated with $\mathcal{J}_{\rho}$, $\mathcal{L}_{\rho}$, and $\mathcal{N}_{\rho}$,
respectively, and that the identities
\begin{align*}
	& J_{\rho_0}(x_{0,1}) = \mathfrak{m}_{\star,1,2}(x_0), \\
	& L_{\rho_0}(x_{0,1}) = 2 \mathfrak{v}_\star(x_0) - \mathfrak{m}_{\star,1,1}(x_0) + \mathfrak{m}_{\star,2,2}(x_0), \\
	& N_{\rho_0}(x_{0,1}) = 2 \mathfrak{v}_\star(x_0) + \mathfrak{m}_{\star,1,1}(x_0) - \mathfrak{m}_{\star,2,2}(x_0),
\end{align*}
hold. Moreover, by \cite[Proposition~8.12]{Bethuel-AC-Acta} (which is essentially
a direct consequence of~\eqref{eq:nu-ac-est} and of the tangent cone property), we have
\begin{align*}
	& J_{x_0, \rho_0}(s) = 0  \qquad \mbox{for } s \in (x_{0,1}-\rho_0,\,x_{0,1}+\rho_0), \\
	& N_{x_0, \rho_0}(x_{0,1}) = 0.
\end{align*}
Gathering these identities leads to the following analogue of
\cite[Proposition~1.20]{Bethuel-AC-Acta}.
\begin{prop}\label{prop:analogue-of-Bethuel-Prop5}
	Let $x_0 \in \mathfrak{S}_\star \setminus \mathfrak{E}_\star$ be arbitrary
	and let $\e_{x_0}$ be the tangent vector to $\mathfrak{S}_\star$ at $x_0$.
	Assume that the orthonormal frame $(\e_1,\,\e_2)$ is chosen so that
	$\e_1 = \e_{x_0}$. Then, we have the identities
	\begin{equation}\label{eq:bethuel-prop5}
	\begin{cases}
		2 \mathfrak{v}_\star(x_0) = \mathfrak{m}_{\star,2,2}(x_0) - \mathfrak{m}_{\star,1,1}(x_0), \\
		\mathfrak{m}_{\star,1,2}(x_0) = 0.
	\end{cases}
	\end{equation}
\end{prop}

\begin{remark}\label{rk:bethuel-prop5}
	Using the notation in~\eqref{eq:m*-perp-par},~\eqref{eq:m*-perp-par-bis},
	we could rephrase~\eqref{eq:bethuel-prop5} in the `intrinsic' form
	\[
	\begin{cases}
		2 \mathfrak{v}_\star = \mathfrak{m}_{\star,\perp,\perp} - \mathfrak{m}_{\star,||,||}, \\
		\mathfrak{m}_{\star,\perp,||} = 0,
	\end{cases}
	\]
	which is an identity holding on the whole set $\mathfrak{S}_\star \setminus \mathfrak{E}_\star$ of good points in $\mathfrak{S}_\star$.
\end{remark}
From the discussion at the end of Section~\ref{sec:hopf}, we immediately draw
a counterpart of \cite[Lemma~1.21]{Bethuel-AC-Acta}.
\begin{lemma}\label{lemma:analogue-of-Bethuel-Lemma2}
	For any $x \in \mathfrak{S}_\star \setminus \mathfrak{E}_\star$,
	we have the identity
	\begin{equation}\label{eq:Bethuel-Lemma2}
		\omega_\star^{\rm ac}(x) = - 2 e^{-2i\gamma_\star(x)} \zeta_\star^{\rm ac}(x),
	\end{equation}
	where $\gamma_\star(x) \in \left[-\frac{\pi}{2},\,\frac{\pi}{2}\right]$ denotes
	the angle between $\e_1$ and the tangent vector $\e_{x}$ to $\mathfrak{S}_\star$
	at $x$.
\end{lemma}

\subsection{Monotonicity properties of the limiting potential $\zeta_\star$}
\label{sec:mon-zeta*}

In this section, we show that, for any 
$x_0 \in \Omega$
and any
annulus centred at $x_0$ and properly contained in $\Omega \setminus \spt \mu_\star$,
the function $r \mapsto \frac{\zeta_\star(B(x_0,\,r)}{r}$ is monotone non-decreasing
for $0 < r \leq R$, where $R$ denotes the exterior radius of the
annulus. As a consequence, we obtain the existence of the density
of $\zeta_\star$ at $x_0$ and, hence, the absolute continuity of $\zeta_\star$
with respect to $\lambda_\star$.

\vskip5pt

\noindent
The first step towards the desired monotonicity property is
the following result, which is analogous to~\cite[Lemma~1.24]{Bethuel-AC-Acta}.
\begin{lemma}\label{lemma:analogue-of-Bethuel-Lemma3}
	Let $x_0 \in \Omega$ be any point. Then, for every $R > 0$ and every radii
	$r_1$, $r_2$ with $0 < r_1 < r_2 \leq R$ and such that
	$B(x_0,\,r_2) \setminus B(x_0,\,r_1) \csubset \Omega \setminus \spt\mu_\star$,
	we have the identity
	\begin{equation}\label{eq:lemma3}
		\frac{\zeta_\star(B(x_0,\,r_2))}{r_2} - \frac{\zeta_\star(B(x_0,\,r_1))}{r_1}
		= \int_{B(x_0,\,r_2) \setminus B(x_0,\,r_1)} \frac{1}{4r} \,{\d}\mathscr N_{x_0,\star},
	\end{equation}
	where
	\begin{equation}\label{eq:N*}
		\mathscr{N}_{\star,x_0} = (\mbox{weak*-})
		\lim_{\eps \to 0}\left(\frac{2}{\eps} V(\,\cdot,\,\M_\eps) - \eps \abs{\partial_\ttau \M_\eps}^2 + \eps \abs{\partial_\nnu \M_\eps}^2\right).
	\end{equation}
\end{lemma}

\begin{proof}
	Let us set, for brevity, $B_r := B(x_0,\,r)$, for any $r \in (0,\,R)$.

	By~\eqref{eq:pohozaev-ball}, the pointwise identity~\eqref{eq:pointwise-est-V-bis},
	and in view of Lemma~\ref{lemma:limit-f-minus-V}, we can write
	\begin{equation}\label{eq:N*-compu1}
	\begin{split}
		r^2\frac{\d}{{\d}r}\left( \frac{1}{\eps r} V(x,\,\M_\eps(x)) \right) &= \frac{r}{2\eps}\int_{\partial B_r} V(x,\,\M_\eps(x)) \\
		&+ \frac{r}{4} \int_{\partial B_r}\left( \abs{\partial_\nnu \Q_\eps} +
		\eps \abs{\partial_\nnu \M_\eps}^2 - \abs{\partial_\ttau \Q_\eps}^2 - \eps \abs{\partial_\ttau \M_\eps}^2 \right) \\
		&+ \o_{\eps \to 0}(1)
	\end{split}
	\end{equation}
	In virtue of the strong convergence $\Q_\eps \to \Q_\star$ in
	$W^{1,2}_{\rm loc}(\Omega \setminus \spt\mu_\star)$ and recalling that,
	being harmonic, $\Q_\star$ satisfies
	\[
		\int_{\partial B_r} \abs{\partial_\nnu \Q_\star}^2
		= \int_{\partial B_r} \abs{ \partial_\ttau \Q_\star}^2
	\]
	for any $r \in (0,\,R)$, passing to the limit as $\eps \to 0$
	in~\eqref{eq:N*-compu1}, it follows that, for a.e. $r \in (0,\,R)$,
	\[
		r^2 \frac{\d}{{\d}r}\left( \zeta_\star(B_r) \right)
		= \frac{r}{4} \mathscr{N}_{\star,x_0}(\partial B_r),
	\]
	where $\mathscr{N}_{\star,x_0}$ is as in~\eqref{eq:N*}.
	The conclusion follows straightforwardly.
\end{proof}

With Lemma~\ref{lemma:analogue-of-Bethuel-Lemma3} at hand,
the monotonicity of $r \mapsto \frac{\zeta_\star(B(x_0,\,r))}{r}$ is proved once we show
that $\mathscr{N}_{\star,x_0}$ is a non-negative measure, for any
$x_0 \in \Omega$.
Although we could also conclude immediately by invoking
\emph{directly} the results in \cite[Section~9]{Bethuel-AC-Acta},
to facilitate the reader we go through the main points of the argument.

First of all, it is convenient to switch to polar coordinates $(r,\,\theta)$ centred
at $x_0$. Upon defining
\[
\begin{cases}
	\mathbbm{m}_{\star,r,r} := \cos^2 \theta \mathbbm{m}_{\star,1,1}
	+ \sin^2\theta \mathbbm{m}_{\star,2,2}
	+ 2 \sin\theta \cos\theta \mathbbm{m}_{\star,1,2}, \\
	\frac{1}{r^2}\mathbbm{m}_{\star,\theta,\theta}
	:= \sin^2 \theta \mathbbm{m}_{\star,1,1}
	+ \cos^2\theta \mathbbm{m}_{\star,2,2}
	- 2 \sin\theta \cos\theta \mathbbm{m}_{\star,1,2},
\end{cases}
\]
we have, as $\eps \to 0$,
\[
	\eps \abs{\partial_r \M_\eps}^2 \rightharpoonup^* 	\mathbbm{m}_{\star,r,r}, \qquad
	\eps \abs{\partial_\theta \M_\eps}^2 \rightharpoonup^* 	\mathbbm{m}_{\star,\theta,\theta},
\]
as measures in $\Omega$. Using this notation, we can rewrite~\eqref{eq:N*}
in the form
\begin{equation}\label{eq:N*-bis}
	\mathscr{N}_{\star,x_0} = 2 \zeta_\star - \frac{1}{r^2} \mathbbm{m}_{\star,\theta,\theta} + \mathbbm{m}_{\star,r,r}.
\end{equation}
Decomposing also the densities $\mathfrak{m}_{\star,j,k}$,
for $(j,\,k) \in \{1,\,2\}^2$,
in the polar frame $(r,\,\theta)$ and using
\eqref{eq:N*-bis} and Lemma~\ref{lemma:analogue-of-Bethuel-Lemma2} as
in the proof of \cite[Lemma~9.1]{Bethuel-AC-Acta}, we end up with
\begin{equation}\label{eq:Nu*-ac}
	\NN_{\star,x_0}^{\rm ac} = 4 \sin^2(\gamma_\star - \theta) \zeta_\star^{\rm ac},
\end{equation}
where $\NN_{\star,x_0}^{\rm ac}$ denotes the absolutely continuous part of $\NN_{\star,x_0}$
with respect to $\lambda_\star$.

Next, let $\rho > \delta > 0$ be given. Following \cite[Section~9.3]{Bethuel-AC-Acta},
we introduce the auxiliary functions
\begin{align*}
	& Z(r) := \zeta_\star(B(x_0,\,r)), \\
	& F(r) := \frac{Z(r)}{r}, \\
	& G_\delta(r) := \int_{B(x_0,\,r) \setminus B(x_0,\,\delta)} \frac{1}{\abs{x-x_0}}\,{\d}\NN_{\star,x_0},
\end{align*}
where $r \in [\delta,\,\rho]$.
Letting
\[
\Pi : B(x_0,\rho) \setminus \{x_0\} \to (0,\rho) \quad\colon\quad
x \mapsto \Pi(x) := \sqrt{ \left(x-x_{0,1}\right)^2 + \left(x-x_{0,2}\right)^2 },
\]
we also introduce the auxiliary measures
\[
	\check{\zeta}_\star := \Pi_\sharp(\zeta_\star), \qquad
	\check{\NN}_{\star,x_0} := \Pi_\sharp(\NN_{\star,x_0}),
\]
so that, for any Borel sets $A \subset [\delta,\,\rho]$, there we have
\[
	\check{\zeta}_\star(A) = \check{\zeta}_\star(\Pi^{-1}(A)), \qquad
	\check{\NN}_{\star,x_0}(A) = \check{\NN}_{\star,x_0}(\Pi^{-1}(A)).
\]
Exactly as in \cite[Lemma~9.3]{Bethuel-AC-Acta}, as a direct
consequence of the definitions and Fubini's theorem, it follows
that $Z$ and $G_\delta$ have bounded variation and that
\begin{equation}\label{eq:Bethuel-Lemma-9.2}
	\frac{\d}{{\d}r} Z = \check{\zeta}_\star \geq 0, \qquad
	\frac{\d}{{\d}r} G_\delta = \frac{1}{r}\check{\NN}_{\star,x_0}, \qquad
	\mbox{in } \mathscr{D}'((\delta,\,\rho)).
\end{equation}
Using these facts, one proves that $F$ has bounded variation
and that
\begin{equation}\label{eq:Bethuel-Lemma-9.3}
	\frac{\d}{{\d}r} F = \frac{1}{r}\check{\zeta}_\star - \frac{1}{r^2} Z
	= \frac{1}{4r}\check{\NN}_{\star,x_0}, \qquad \mbox{in } \mathscr{D}'((\delta,\,\rho)).
\end{equation}
cf. \cite[Lemma~9.4]{Bethuel-AC-Acta}. Then, arguing as
in \cite[Section~9.4]{Bethuel-AC-Acta},
we take advantage of~\eqref{eq:Nu*-ac} and of the two different forms for
$\frac{\d}{{\d}r}F$ to show that the latter is a non-negative measure.
More precisely, once defined the annulus
\[
	A_{\delta,\rho} := B(x_0,\,\rho) \setminus B(x_0,\,\delta),
\]
we consider the set
\[
	\mathbb{A}_{\delta,\rho} = \Pi^{-1}\left( \mathfrak{E} \cap A_{\delta,\rho} \right).
\]
Then, $\H^1(\mathbb{A}_{\delta,\rho}) = 0$ and, since $Z$ is a bounded function,
\begin{equation}\label{eq:Z/r-B-rho}
	\frac{Z}{r} \mres {\mathbb{A}_{\delta, \rho}} = 0.
\end{equation}
Furthermore, since
$\NN_{\star,x_0} = \NN_{\star,x_0}^{\rm ac}$ on $B(x_0,\,\rho) \setminus \mathfrak{E}_\star$,
we have
\begin{equation}\label{eq:Bethuel-Lemma-9.4}
	\check{\NN}_{\star,x_0} \mres ((0,\,\rho) \setminus \mathbb{A}_{\delta,\rho}) \geq 0,
\end{equation}
by~\eqref{eq:Nu*-ac}. On the other hand, by~\eqref{eq:Bethuel-Lemma-9.3},
\begin{equation}\label{eq:Bethuel-Lemma-9.3-bis}
	\check{\NN}_{\star,x_0} = 4\left( \check{\zeta}_\star - \frac{Z}{r} \right)
	\qquad \mbox{in } \mathscr{D}'((\delta,\,\rho)).
\end{equation}
Since both sides of~\eqref{eq:Bethuel-Lemma-9.3-bis} are (not only distributions
but also) bounded measures,
the identity~\eqref{eq:Bethuel-Lemma-9.3-bis} is an identity of measures.
On the other hand, in view of~\eqref{eq:Z/r-B-rho}, from~\eqref{eq:Bethuel-Lemma-9.3-bis} it follows that
\begin{equation}\label{eq:Bethuel-Lemma-9.5}
	\check{\NN}_{\star,x_0} \mres \mathbb{A}_{\delta,\rho} =
	\check{\zeta}_{\star} \mres \mathbb{A}_{\delta,\rho} \geq 0.
\end{equation}
Gathering~\eqref{eq:Bethuel-Lemma-9.4} and~\eqref{eq:Bethuel-Lemma-9.5},
we obtain
\[
	\check{\NN}_{\star,x_0} \geq 0 \qquad \mbox{on } (0,\,\rho)
\]
and, by~\eqref{eq:Bethuel-Lemma-9.3}, $\frac{\d}{{\d}r}F \geq 0$ on $(\delta,\,\rho)$
for any $\delta > 0$, hence $\frac{\d}{{\d}r}F \geq 0$ on $(0,\,\rho)$.
Thus, recalling the definition of $F$, we see that
\[
	\frac{\d}{{\d}r}\left( \frac{\zeta_\star(B(x_0,\,r))}{r} \right) \geq 0
	\qquad \mbox{on } (0,\,\rho),
\]
and we have therefore obtained the same conclusion
of \cite[Proposition~1.23]{Bethuel-AC-Acta}, summarised below.
\begin{prop}\label{prop:analogue-of-Bethuel-Prop-6}
The measure $\zeta_\star$ has the following properties.
		\begin{enumerate}[(i)]
		\item\label{item:Z/r-monotone}
		Let $x_0 \in \Omega$ be any point. Then, for every $R > 0$
		and for almost every radii $r_1$, $r_2$ with $0 < r_1 < r_2 \leq R$
		and such that $B(x_0,\,r_2) \setminus B(x_0,\,r_1) \csubset \Omega \setminus \spt\mu_\star$, we have the inequality
		\begin{equation}\label{eq:Z/r-monotone}
			\frac{\zeta_\star(B(x_0,\,r_2))}{r_2} \geq \frac{\zeta_\star(B(x_0,\,r_1))}{r_1}.
		\end{equation}
		\item\label{item:z*-sptmu*} The measure $\zeta_\star$ does not charge
		the support of $\mu_\star$, that is,
		\begin{equation}\label{eq:z*-sptmu*}
			\zeta_\star(\spt \mu_\star) = 0.
		\end{equation}
		\item\label{item:zeta*-density} For every $x_0\in \Omega$, the
		function $r \mapsto \frac{\zeta_\star(B(x_0,\,r))}{r}$ has a limit when
		$r \to 0$, and we have, in fact,
	\begin{equation}\label{eq:zeta*-density}
		\mathfrak{v}_\star(x_0) = \lim_{r \to 0} \frac{\zeta_\star(B(x_0,\,r))}{2r}
		\leq \frac{\zeta_\star(B(x_0,\,R))}{2R},
	\end{equation}
		for every $R > 0$ such that
		$\left(B(x_0,\,2R) \setminus \{x_0\} \right) \subset \Omega \setminus \spt\mu_\star$.
		\item\label{item:zeta*-leq-nu*}
		There holds
		\begin{equation}\label{eq:zeta*-leq-nu*}
			\zeta_\star \leq \nu_\star
		\end{equation}
		in the sense of measures in $\Omega$.
		\item\label{item:zeta*<<lambda*} The measure $\zeta_\star$ is absolutely continuous with respect
		to $\lambda_\star = \H^1 \mres \mathfrak{S}_\star$
		in $\Omega$. In fact,
		\[
			\zeta_\star = \mathfrak{v}_\star \H^1 \mres \mathfrak{S}_\star
		\]
		in $\Omega$. 
	\end{enumerate}
\end{prop}

\begin{remark}
	We stress that $x_0$ can be \emph{any point in $\Omega$}; in particular,
	it may very well be a point of the support of $\mu_\star$.
\end{remark}

\begin{proof}
	Inequality~\eqref{eq:Z/r-monotone} has been already obtained above.
	Item~\ref{item:z*-sptmu*} follows immediately from
	inequality~\eqref{eq:Z/r-monotone}.
	Item~\ref{item:zeta*-density} is an obvious consequence of~\eqref{eq:Z/r-monotone}.
	About item~\ref{item:zeta*-leq-nu*},
	the inequality
	\[
		\zeta_\star \mres (\Omega \setminus \spt\mu_\star)
		\leq \nu_\star \mres (\Omega \setminus \spt\mu_\star)
	\]
	follows already from the definitions of $\zeta_\star$
	and $\nu_\star$ and item~\ref{item:tilde-nu-eps} of
	Theorem~\ref{thm:properties-nu*}. On the other hand,
	by~\eqref{eq:z*-sptmu*} we have also
	$\zeta_\star(\spt\mu_\star) = 0$,
	and the conclusion follows.
	Finally, item~\ref{item:zeta*<<lambda*} follows straightforwardly
	by combining
	item~\ref{item:zeta*-density}, item~\ref{item:zeta*-leq-nu*},
	and item~\ref{item:z*-sptmu*}.
\end{proof}

\subsection{Upper bounds, absolute continuity of $\nu_\star$ with respect to $\lambda_\star$, and consequences}
\label{sec:upper-bounds}

By item~\ref{item:zeta*-leq-nu*} of
Proposition~\ref{prop:analogue-of-Bethuel-Prop-6}, we know that the
inequality $\zeta_\star \leq \nu_\star$ holds in the sense
of measure in $\Omega$.
We now show
that also a converse inequality holds locally on $\Omega \setminus \spt\mu_\star$,
so that $\nu_\star$ is equivalent to $\zeta_\star$ in
$\Omega \setminus \spt\mu_\star$
and, in turn, absolutely continuous with respect to $\lambda_\star$
in $\Omega \setminus \spt\mu_\star$.

\begin{prop}\label{prop:nu*<<lambda*}
	The measure $\nu_\star$ is equivalent to $\zeta_\star$ and
	absolutely continuous with respect to
	$\lambda_\star$ in $\Omega \setminus \spt\mu_\star$.
	Moreover, we have $\nu_\star = \mathfrak{e}_\star \lambda_\star$,
	with
	\begin{equation}\label{eq:nu*-leq-zeta*}
		\mathfrak{e}_\star(x_0) \leq 2 {\rm K}_\beta \, \mathfrak{v}_\star(x_0)
		\quad \mbox{ for $\lambda_\star$-almost every }
		x_0 \in \mathfrak{S}_\star,
	\end{equation}
 	where $\mathfrak{e}_\star : \mathfrak{S}_\star \to \R^+$ denotes the
 	density of $\nu_\star$ and the constant ${\rm K}_\beta$ depends only on $\beta$
 	and on $\Omega$.
\end{prop}

\begin{proof}
The proof combines \cite[Lemma~9.7]{Bethuel-AC-Acta}
with \cite[Corollary~9.8]{Bethuel-AC-Acta} and
Proposition~\ref{prop:analogue-of-Bethuel-Prop-6}.
To keep it more transparent, we divide it into two steps.

\setcounter{step}{0}
\begin{step}[$\nu_\star \sim \zeta_\star$ in $\Omega \setminus \spt \mu_\star$]
From Proposition~\ref{prop:analogue-of-Bethuel-4.3} (in
particular, from~\eqref{eq:est-M-whole-B''}) and the logarithmic
bound on the energy $\F_\eps$ in~\eqref{eq:log-bound}, we derive
that, for any ball $B(x_0,\,R) \csubset \Omega\setminus \spt\mu_\star$,
there holds
\[
	\nu_\eps\left(B(x_0,\,R/2) \right) \leq  {\rm K}_\beta\,
	\zeta_\eps\left( \overline{B(x_0,\,3R/4)} \right)
	+ \o_{\eps \to 0}(1),
\]
where ${\rm K}_\beta$ is the constant in the right-hand side
of~\eqref{eq:est-M-whole-B''}, which depends only
on $\beta$ and $\Omega$. Thus, passing to the limit $\eps \to 0$, we get
\begin{equation}\label{eq:zeta*-geq-nu*-partial}
	\nu_\star(B(x_0,\,R/2)) \leq {\rm K}_\beta \,
	\zeta_\star\left( \overline{B(x_0,\,3R/4)} \right),
\end{equation}
which is the same conclusion as in \cite[Lemma~9.7]{Bethuel-AC-Acta}.
As a consequence of~\eqref{eq:zeta*-geq-nu*-partial}, if we fix any
compact set $K \subset \Omega \setminus \spt\mu_\star$ and we take any
Borel set $E \subset K$ which is
null for the measure $\zeta_\star$, by covering $E$ with balls $B(x_j,\,R/2)$
such that $B(x_j,\,4R/5) \subset K$
and then using the outer regularity of $\zeta_\star$, it follows that $E$ is null
also for the measure $\nu_\star$,
hence $\nu_\star \ll \zeta_\star$ in $K$. In view of~\eqref{eq:zeta*-leq-nu*},
it follows that $\nu_\star$ and $\zeta_\star$ have the same null sets
in $K$, for any compact set $K \subset \Omega \setminus \spt \mu_\star$.
By exhausting $\Omega \setminus \spt\mu_\star$ by a countable family of
nested compact sets, the same holds true
in $\Omega \setminus \spt \mu_\star$, and therefore $\nu_\star \sim \zeta_\star$
in $\Omega \setminus \spt \mu_\star$.
\end{step}

\begin{step}[{Proof of~\eqref{eq:nu*-leq-zeta*}} and conclusion]
	From~\eqref{eq:density-nu*-good-point-ac},~\eqref{eq:density-nu*-good-point-sing},~\eqref{eq:zeta*-geq-nu*-partial} and~\eqref{eq:zeta*-density}, as in \cite[Corollary~9.8]{Bethuel-AC-Acta},
	denoting $\overline{\D}_{\lambda_\star}(\nu_\star)(x_0)$ the upper density
	of $\nu_\star$ with respect to $\lambda_\star$ at $x_0$,
	it follows that, for every
	$x_0 \in \mathfrak{S}_\star \setminus \mathfrak{E}_\star$,
	\begin{equation}\label{eq:density-nu-controlled-by-zeta}
	\begin{split}
		\overline{\D}_{\lambda_\star}(\nu_\star)(x_0)
		&= \limsup_{r \to 0} \frac{\nu_\star(B(x_0,\,r))}{2r}
		= \limsup_{r \to 0} \frac{\nu_\star(B(x_0,\,r/2))}{r} \\
		&\leq {\rm K}_\beta \, \lim_{r \to 0} \frac{\zeta_\star(B(x_0,\,r))}{r}
		= 2 {\rm K}_\beta \, \mathfrak{v}_\star(x_0).
	\end{split}
	\end{equation}
	This implies that $\nu_\star$ is absolutely continuous with respect
	to $\lambda_\star$.
	Moreover,
	since $x_0 \in \mathfrak{S}_\star \setminus \mathfrak{E}_\star$,
	and, therefore, $\mathfrak{e}_\star(x_0) = {\overline \D}_{\lambda_\star}(\nu_\star)(x_0)$,
	\eqref{eq:nu*-leq-zeta*} follows as well. The proof is complete.
\end{step}

\end{proof}

\begin{remark}\label{rk:nu-controlled-by-zeta}
	A straightforward consequence of the absolute continuity of
	$\nu_\star$, $\zeta_\star$ with respect to $\lambda_\star$
	on $\mathfrak{S}_\star$ is the fact that such measures do not
	have singular part with respect to $\lambda_\star$, i.e.,
	$\nu^{\rm sing}_\star$, $\zeta^{\rm sing}_\star \equiv 0$ on
	$\mathfrak{S}_\star$. In particular, the inequality~\eqref{eq:nu*-leq-zeta*}
	implies that
	\[
		\spt \zeta_\star = \mathfrak{S}_\star
	\]
	and that
	\[
		\nu_\star \leq 2 {\rm K}_\beta \, \zeta_\star,
	\]
	as measures in $\Omega \setminus \spt\mu_\star$. Therefore, we can
	really study the properties of $\mathfrak{S}_\star$ using $\zeta_\star$
	rather than $\nu_\star$.
\end{remark}

As immediate consequences of the absolute continuity of $\nu_\star$, $\zeta_\star$ with
respect to $\lambda_\star$ on $\mathfrak{S}_\star$
(in other words, of the fact that $\nu^{\rm sing}_\star$, $\zeta^{\rm sing}_\star \equiv 0$ on $\mathfrak{S}_\star$), we have the following stronger versions of
Proposition~\ref{prop:analogue-of-Bethuel-Prop5}
and of Lemma~\ref{lemma:analogue-of-Bethuel-Lemma2}.
\begin{theorem}\label{thm:analogue-of-Bethuel-Thm5}
	We have the identities
\begin{equation}\label{eq:bethuel-thm5}
	2 \zeta_\star = \mathbbm{m}_{\star,\perp,\perp} - \mathbbm{m}_{\star,||,||},
	\qquad \mathbbm{m}_{\star,\perp,||} = 0,
\end{equation}
where the measures $\mathbbm{m}_{\star,\perp,\perp}$, $\mathbbm{m}_{\star,||,||}$,
and $\mathbbm{m}_{\star,\perp,||}$ are defined
in \eqref{eq:m*-perp-par},~\eqref{eq:m*-perp-par-bis}.
\end{theorem}

\begin{proof}
	The identities in~\eqref{eq:bethuel-thm5} follow both from
	Proposition~\ref{prop:analogue-of-Bethuel-Prop5} and the fact
	that $\nu_\star^{\rm sing}$, $\zeta_\star^{\rm sing}$ are identically
	zero on $\mathfrak{S}_\star \setminus \mathfrak{E}_\star$,
	along with the trivial observation that
	\[
		-2 \nu_\star \leq \mathbbm{m}_{\star,j,k} \leq 2 \nu_\star
	\]
	for any $(j,\,k) \in \{1,\,2\}^2$.
\end{proof}

\begin{prop}\label{prop:analogue-of-Bethuel-Lemma2-strong}
	Let $(\e_1,\,\e_2)$ be any orthonormal basis of $\R^2$.
	We have the identity
	\begin{equation}\label{eq:Bethuel-Lemma2-strong}
		\omega_\star = - 2 e^{-2i\gamma_\star} \zeta_\star,
	\end{equation}
	where $\gamma_\star(x) \in \left[-\frac{\pi}{2},\,\frac{\pi}{2}\right]$ denotes
	the angle between $\e_1$ and the tangent vector $\e_{x}$ to $\mathfrak{S}_\star$
	at $x \in \mathfrak{S}_\star \setminus \mathfrak{E}_\star$.
\end{prop}

\begin{proof}
	The lemma follows immediately by
	combining Lemma~\ref{lemma:analogue-of-Bethuel-Lemma2}
	with~\eqref{prop:hopf},~\eqref{eq:hopf-M},
	Proposition~\ref{prop:analogue-of-Bethuel-Prop-6}, and~\eqref{eq:bethuel-thm5}.
\end{proof}

%
%

\section{Proof of Theorem~\ref{mainthm:B+C} and Theorem~\ref{mainthm:balance-law}}
\label{sec:S*}

As an easy consequence of
Proposition~\ref{prop:analogue-of-Bethuel-Prop-6}, Corollary~\ref{cor:CR},
and Theorem~\ref{thm:analogue-of-Bethuel-Thm5} (with the aid
of Proposition~\ref{prop:analogue-of-Bethuel-Lemma2-strong}),
we conclude that the limiting
measure $\zeta_\star$ is the weight measure
of a countably $\H^1$-rectifiable varifold
$\mathbb{V}_\star := \v(\mathfrak{S}_\star,\,\mathfrak{v}_\star)$ in $\Omega$,
which is also stationary in $\Omega \setminus \spt\mu_\star$. This gives further structure to the set
$\mathfrak{S}_\star = \spt\nu_\star \setminus \spt\mu_\star$, which turns out
to be locally a union of segments with locally constant densities.
In addition, we are able to compute the first variation of
$\mathbb{V}_\star$.
As a consequence, we obtain the proof of Theorem~\ref{mainthm:B+C}
and of Theorem~\ref{mainthm:balance-law}.

\paragraph{Varifolds}
For the reader's convenience, we briefly recall the concepts of \emph{varifold},
\emph{rectifiable varifold},
\emph{first variation of a varifold}, and \emph{stationary varifold}.
We refer to \cite{Simon} for a comprehensive treatment of varifolds.

A $k$-varifold~$\mathbb{V}$
on an open set $U \subset \R^n$ is a Radon measure
on $G_k(U) := U \times G(n,\,k)$, where $G(n,\,k)$
denotes the Grassmanian manifold of $k$-dimensional
planes in $\R^n$. The \emph{weight measure} $\norm{\mathbb{V}}$ of $\mathbb{V}$
is the Radon measure defined by setting
\[
	\int_U \phi(x) \,{\d} \norm{\mathbb{V}}(x)
	:= \int_{G_k(U)} \phi(x)\,{\d}\mathbb{V}(x,\,S),
\]
for all $\phi \in C_c(U)$.
A $k$-varifold is \emph{$\H^k$-rectifiable}
(or \emph{$k$-rectifiable}, for short)
if there exist a
countably $\H^k$-rectifiable set $\mathfrak{S}$ and
a locally $\H^k$-integrable function
$\theta\colon\mathfrak{S} \to \R^+$ (called \emph{density function}
or \emph{multiplicity}) so that
\[
	\int_{G_k(U)} \varphi(x,\,S) \,{\d} \mathbb{V}(x,\,S) =
	\int_{\mathfrak{S}} \varphi(x,\,\T_x \mathfrak{S})\,\theta(x) \,{\d}\H^k(x)
\]
for all $\varphi \in C_c(G_k(U))$, where $\T_x \mathfrak{S}$ denotes
the approximate tangent space to $\mathfrak{S}$ at $x \in \mathfrak{S}$.
Thus, if $\mathbb{V}$ is $k$-rectifiable, then
\[
	\norm{\mathbb{V}} = \theta \H^k \mres \mathfrak{S}.
\]
Two pairs $(\mathfrak{S},\,\theta)$ and
$(\mathfrak{S}',\,\theta')$ identify the same $\H^k$-rectifiable varifold
if and only if the symmetric difference $\mathfrak{S} \Delta \mathfrak{S}'$
is a $\H^k$-null set and $\theta = \theta'$ on~$\H^k$-almost all of
$\mathfrak{S} \cap \mathfrak{S}'$.
These conditions define an equivalence relation,
so that rectifiable varifolds are actually equivalence classes of
pairs $(\mathfrak{S},\,\theta)$.
We will write $\mathbb{V} = \v(\mathfrak{S}, \, \theta)$
for the rectifiable varifold carried by~$\mathfrak{S}$
with density function~$\theta$
(see e.g.~\cite[Chapter~4]{Simon}).
The \emph{first variation} of a varifold~$\mathbb{V}$
is the distribution $\delta \mathbb{V}$ in $U$ satisfying
\[
	\delta\mathbb{V}(\X) :=
	\int_{\mathfrak{S}} \div_{\T_x \mathfrak{S}} \X(x) \,{\d} \H^k(x)
\]
for all $\X \in C^1_c(U,\,\R^n)$. If it happens that $\delta{\mathbb{V}}(\X) = 0$
for any $\X \in C^1_c(U,\,\R^n)$, then $\mathbb{V}$ is said to be \emph{stationary}
in $U$.

\begin{theorem}\label{thm:stationary-varifold}
	The pair $(\mathfrak{S}_\star,\,\mathfrak{v}_\star)$ is
	a representative of a countably
	$\H^1$-rectifiable varifold in $\Omega$ whose
	weight measure is $\zeta_\star$. In addition, the varifold
	$\mathbb{V}_\star := \v(\mathfrak{S}_\star,\,\mathfrak{v}_\star)$ is stationary
	in $\Omega \setminus \spt\mu_\star$.
\end{theorem}

\begin{proof}
We divide the proof into two steps. First, we formalise the fact, already
implied by the results in the above sections, that
$(\mathfrak{S}_\star,\,\mathfrak{v}_\star)$ is a representative of
a countably $\H^1$-rectifiable
varifold in $\Omega$. Second, following the argument
of \cite[Theorem~1.5]{Bethuel-AC-Acta}, we show that it is
indeed stationary as a varifold in $\Omega \setminus \spt\mu_\star$
(i.e., with respect to variations supported in $\Omega \setminus \spt\mu_\star$).
\setcounter{step}{0}
\begin{step}[Rectifiability]
	By Proposition~\ref{prop:S_star},
	$\mathfrak{S}_\star$ is a countably $\H^1$-rectifiable set.
	By Proposition~\ref{prop:analogue-of-Bethuel-Prop-6},
	$\zeta_\star$ is
	absolutely continuous with respect to the measure
	$\lambda_\star = \H^1 \mres \mathfrak{S}_\star$ as measures in
	$\Omega$, with associated
	density $\mathfrak{v}_\star$.
	It then follows from the very definition of countably $\H^1$-rectifiable
	varifold (e.g., \cite[Chapter~4]{Simon}) that the pair
	$(\mathfrak{S}_\star,\,\mathfrak{v}_\star)$
	is a representative of a countably $\H^1$-rectifiable
	varifold in $\Omega$ whose weight measure is precisely $\zeta_\star$.
\end{step}

\begin{step}[Stationarity]
	We argue exactly as in \cite[Section~11]{Bethuel-AC-Acta}.
	The key point is proving that the `modified Cauchy-Riemann
	relations' \eqref{eq:CR} are, in fact, equivalent to the
	stationarity condition
	\begin{equation}\label{eq:def-stationary}
		\int_{\mathfrak{S}_\star} \div_{\T_x \mathfrak{S}_\star} \X\,{\d}\zeta_\star = 0
	\end{equation}
	for any $\X = (X_1,\,X_2) \in C^1_c(\Omega \setminus \spt\mu_\star,\,\R^2)$,
	where (by definition)
	\begin{equation}\label{eq:def-div}
		\div_{\T_x \mathfrak{S}_\star} \X(x)
		= \left( \e_x \cdot \nabla \X(x) \right) \cdot \e_x
	\end{equation}
	at any $x \in \mathfrak{S}_\star \setminus \mathfrak{E}_\star$.

	To this purpose, we consider the measure
	\[
		\N_\star := 2 \zeta_\star - \mathbbm{m}_{\star,2,2} + \mathbbm{m}_{\star,1,1},
	\]
	For any $x \in \mathfrak{S}_\star \setminus \mathfrak{E}_\star$, the
	unit tangent vector $\e_{x}$ can be decomposed along the Cartesian frame
	$(\e_1,\,\e_2)$ as
	\[
		\e_{x} = \cos \gamma(x) \e_1 + \sin \gamma(x) \e_2,
	\]
	where $\gamma(x)$ denotes the angle between $\e_1$ and $\e_x$
	(possibly flipped, so that $\gamma(x) \in [-\pi/2,\,\pi/2]$).
	Plugging this identity in~\eqref{eq:def-div}, we obtain
	\[
	\begin{split}
		\div_{\T_x \mathfrak{S}_\star} \X(x) = &\cos^2 \gamma(x) \frac{\partial X_1}{\partial x_1}(x) + \sin^2 \gamma(x) \frac{\partial X_2}{\partial x_2}(x)  \\
		&+ \sin \gamma(x) \cos \gamma(x)\left[ \frac{\partial X_2}{\partial x_1}(x)
		+ \frac{\partial X_1}{\partial x_2}(x)\right],
	\end{split}
	\]
	so that the stationarity condition~\eqref{eq:def-stationary} can
	be recast in the form (cf.~\cite[(11.2)]{Bethuel-AC-Acta})
	\begin{equation}\label{eq:bethuel-11.2}
		\ip{\zeta_\star}{\cos^2 \gamma(x) \frac{\partial X_1}{\partial x_1} + \sin^2 \gamma(x) \frac{\partial X_2}{\partial x_2}
		+ \sin \gamma(x) \cos \gamma(x)\left[ \frac{\partial X_2}{\partial x_1}
		+ \frac{\partial X_1}{\partial x_2}\right]} = 0.
	\end{equation}
	On the other hand, in view of
	Proposition~\ref{prop:analogue-of-Bethuel-Lemma2-strong},
	we may rewrite
	\[
		\N_\star = 4 \sin^2(\gamma) \zeta_\star,
	\]
	so that, by Theorem~\ref{thm:analogue-of-Bethuel-Thm5},
	the modified Cauchy-Riemann relations~\eqref{eq:CR} can be rewritten in the form
	\begin{equation}\label{eq:bethuel-11.1}
	\begin{cases}
		&-\frac{\partial}{\partial x_2} (\sin(2\gamma) \zeta_\star) = \frac{\partial}{\partial x_1}( 1 + \cos(2 \gamma) \zeta_\star), \\
		&-\frac{\partial}{\partial x_1} (\sin(2\gamma) \zeta_\star) = \frac{\partial}{\partial x_2}( 1 - \cos(2 \gamma) \zeta_\star).
	\end{cases}
	\end{equation}
	Integrating~\eqref{eq:bethuel-11.2} by parts in the sense of distributions
	(recalling that $X_1$, $X_2$ can be chosen independently of each other)
	and combining the
	result with~\eqref{eq:bethuel-11.1}, the conclusion follows.
	\qedhere
\end{step}
\end{proof}

We now take advantage of the stationarity of
$\mathbb{V}_\star$
in $\Omega \setminus \spt\mu_\star$ to compute 
its first variation as a varifold in $\Omega$.
Before stating the result, we establish some further notation. First, we shall
denote $\delta \mathbb{V}_\star$ the first variation of
$\mathbb{V}_\star$, taken in the sense of distributions
in $\Omega$. Next, we recall 
that $\spt\mu_\star$ is a finite set, with cardinality $N_\star \in \mathbb{N}$
and we denote $a_j$, for $j \in \{1,\,\dots,\,n_\star\}$
and $n_\star \leq N_\star$, the points of $\spt\mu_\star$ lying in $\Omega$.
Finally, we write $\nabla_{a_j} \mathbb{W}_\star(a_1,\,\dots,\,a_{n_\star})$
for the gradient of the renormalised energy with respect to its $j$-th variable.
We recall that, from \cite[Theorem~VII.4 and Theorem~VIII.3]{BBH}
and \cite[Theorem~5.1]{Lin96},
$\nabla_{a_j} \mathbb{W}_\star(a_1,\,\dots,\,a_{n_\star})$ satisfies
\begin{equation}\label{eq:ren-energy-BBH}
	\e\cdot \nabla_{a_j} \mathbb{W}_\star(a_1,\,\dots,\,a_{n_\star}) =
	\int_{\partial B_\rho(a_j)} \left( \partial_\e \Q_\star \cdot \partial_\nnu \Q_\star - \frac{1}{2}(\nnu\cdot\e) \abs{\nabla \Q_\star}^2 \right) \,{\d}s
\end{equation}
for any constant vector field $\e \in \R^2$ and
where $\rho > 0$ is any radius small enough that,
for any $j$, $k \in \{1,\,\dots,\,n_\star\}$, there holds
$\overline{B_\rho(a_j)} \cap \overline{B_\rho(a_k)} = \emptyset$ whenever $j \neq k$.

We are now ready for the proof of Theorem~\ref{mainthm:balance-law}. For the 
reader's convenience, we recall its statement below.
\begin{theorem}\label{thm:first-variation}
	The first variation $\delta \mathbb{V}_\star$ of
	$\mathbb{V}_\star := \v(\mathfrak{S}_\star,\,\mathfrak{v}_\star)$
	as a varifold in $\Omega$ is the
	distribution in $\Omega$ defined by
	\begin{equation}\label{eq:first-variation}
		\delta \mathbb{V}_\star(\X) = \sum_{j=1}^{n_\star} z_j \cdot \X(a_j),
	\end{equation}
	for any $\X \in C_c^\infty(\Omega,\,\R^2)$,
	where
	\begin{equation}\label{eq:zj}
		z_j := -\frac{1}{2} \nabla_{a_j} \mathbb{W}_\star(a_1,\,\dots,\,a_{n_\star}),
	\end{equation}
	for any $j \in \{1,\,\dots,\,n_\star\}$.
\end{theorem}

\begin{proof}
We begin by observing that, since $\spt\mu_\star$ is a finite set and
$\mathbb{V}_\star := \v(\mathfrak{S}_\star,\,\mathfrak{v}_\star)$ is stationary in
$\Omega \setminus \spt\mu_\star$ by Theorem~\ref{thm:stationary-varifold},
it suffices to work at fixed $j \in \{1,\,\dots,\,n_\star\}$
and it is enough to consider vector fields $\X = (X_1,\,X_2)$
compactly supported near $\spt \mu_\star \cap \Omega$, i.e.,
$\X \in C_c^\infty\left(\cup_{j=1}^{n_\star} B_s(a_j),\,\R^2\right)$,
for some small $s > 0$.
In fact, by linearity of the first variation it
is enough to consider vector fields of the form
$\X = \e \varphi$, for $\e \in \R^2$ a constant vector field and
$\varphi \in C_c^\infty(B_s(a_j))$ such that
\begin{equation}\label{eq:varphi}
	\varphi(x) = \varphi(\abs{x}), \qquad
	\varphi(a_j) = 1  \quad \mbox{in } B_{s/2}(a_j), \qquad
	\abs{\varphi'(x)} \lesssim 1/s,
\end{equation}
for any $x \in \spt\varphi$.
Here and in the rest of the proof, we assume that
$s > 0$ is any radius small enough that $B_s(a_j) \csubset \Omega$
and $B_s(a_j)$ does not contain points of $\spt\mu_\star \cap \Omega$ other than
$a_j$.

\setcounter{step}{0}
\begin{step}\label{step:dV*=I*}
\emph{
We let
\begin{multline}\label{eq:I*}
	\mathbb{I}_\star(\X) := \lim_{\eps \to 0} \int_\Omega \left( \eps \abs{\partial_1 \M_\eps}^2 \partial_1 X_1 + \eps \partial_1 \M_\eps \cdot \partial_2 \M_\eps (\partial_1 X_2 + \partial_2 X_1)\right. \\
	\left.+ \eps \abs{\partial_2 \M_\eps}^2 \partial_2 X_2 - \frac{\eps}{2} \abs{\nabla \M_\eps}^2 \div \X - \frac{1}{\eps} V(x,\,\M_\eps) \div \X \right)\,{\d}x
\end{multline}
for any $\X \in C^\infty_c(\Omega,\,\R^2)$, and we claim that
\begin{equation}\label{eq:deltaV=I*}
	\delta \mathbb{V}_\star = -\frac{1}{2} \mathbb{I}_\star \qquad
	\mbox{in } \mathscr{D}'(\Omega,\,\R^2).
\end{equation}
}
\begin{proof}[Proof of~\eqref{eq:deltaV=I*}]\renewcommand{\qedsymbol}{\ensuremath{\blacksquare}}
To prove the claim, we first use that, by~\eqref{eq:m*},
the various terms on the right-hand side of the
definition of $\mathbb{I}_\star$ have a limit in the sense of distributions in
$\Omega$ as $\eps \to 0$, so that,
when $\mathbb{I}_\star$ is applied to $\X$, we can write (dropping the subscript $\star$
for the measure $\mathbbm{m}_{\star,j,k}$ for ease)
\[
\begin{split}
	\mathbb{I}_\star(\X) = &\int_\Omega \partial_1 X_1 \, {\d}\mathbbm{m}_{1,1}
	+ \int_\Omega (\partial_1 X_2 + \partial_2 X_1) \, {\d}\mathbbm{m}_{1,2}
	+\int_\Omega \partial_2 X_2 \, {\d}\mathbbm{m}_{2,2} \\
	&-\frac{1}{2} \int_\Omega (\partial_1 X_1 + \partial_2 X_2) \, {\d}(\mathbbm{m}_{1,1} + \mathbbm{m}_{1,2})
	-\int_\Omega (\partial_1 X_1 + \partial_2 X_2) \, {\d}\zeta_\star.
\end{split}
\]
Now, we pass to the moving frame associated with $\mathfrak{S}_\star$ and we
rewrite in this basis the right-hand side of $\mathbb{I}_\star$.
To this purpose, for any good point $x \in \mathfrak{S}_\star$,
we let $\ttau(x) = \e_{x}$ be a
unit tangent vector to $\mathfrak{S}_\star$ at $x$ and we write $\gamma(x)$ for the
angle formed by $\ttau(x)$ with the
canonical orthonormal frame $(\e_1,\,\e_2)$. Up to flipping $\ttau(x)$ if necessary, we
may assume that $\gamma(x) \in [-\pi/2,\,\pi/2]$, so that we can write
\[
\begin{cases}
	\ttau(x) = \e_1 \cos\gamma(x) + \e_2 \sin \gamma(x), \\
	\nnu(x) = -\e_1 \sin\gamma(x) + \e_2 \cos \gamma(x),
\end{cases}
\]
where $\nnu = \ttau^\perp$. By elementary computations and
recalling from Theorem~\ref{thm:analogue-of-Bethuel-Thm5} that
\[
\begin{cases}
	2 \zeta_\star = \mathbbm{m}_{\perp,\perp} - \mathbbm{m}_{||,||}
	= \mathbbm{m}_{\nnu,\nnu} - \mathbbm{m}_{\ttau,\ttau}, \\
	\mathbbm{m}_{\perp,||,||} = \mathbbm{m}_{\nnu,\ttau} = 0,
\end{cases}
\]
it follows that
\begin{equation}\label{eq:first-variation-compu1}
	\mathbb{I}_\star(\X) = -2 \int_\Omega \ttau\cdot \partial_\ttau \X \,{\d} \zeta_\star
	= -2 \int_{\mathfrak{S}_\star} \div_{\T_x \mathfrak{S}_\star} \X(x)
	\,{\d}\zeta_\star(x),
\end{equation}
because, in view of Proposition~\ref{prop:nu*<<lambda*} and of item~\ref{item:zeta*<<lambda*}
of Proposition~\ref{prop:analogue-of-Bethuel-Prop-6},
$\zeta_\star \ll \H^1 \mres \mathfrak{S}_\star$ in $\Omega$. Thus, the claim follows
from~\eqref{eq:first-variation-compu1} and the definition of first variation of
a varifold.
\end{proof}
\end{step}

\begin{step}\label{step:balance-1}
\emph{For any $\e \in \R^2$ and for any $j \in \{1,\,\dots,\,n_\star\}$,
there holds
\begin{equation}\label{eq:balance-1}
\begin{split}
	-\e \cdot \nabla_{a_j} &\mathbb{W}_\star(a_1,\dots,a_{n_\star}) \\
	&= \lim_{\eps \to 0}
	\int_{\partial B_\rho(a_j)} \left( \eps \partial_\e \M_\eps \cdot \partial_\nnu  \M_\eps - \frac{\eps}{2} (\e \cdot \nnu) \abs{\nabla \M_\eps}^2 - \frac{1}{\eps}(\e \cdot \nnu) V(\M_\eps) \right).
\end{split}
\end{equation}
}
\begin{proof}[Proof of~\eqref{eq:balance-1}]\renewcommand{\qedsymbol}{\ensuremath{\blacksquare}}
The proof of~\eqref{eq:balance-1} is almost immediate.
Indeed, since
\[
	\zeta_* \mres (\Omega \setminus \spt\mu_\star)
	= {\mbox (weak)}^*{\mbox -}\lim_{\eps \to 0}
	\frac{1}{\eps^2} f_\eps(\Q_\eps,\,\M_\eps)
\]
by Lemma~\ref{lemma:limit-f-minus-V},~\eqref{eq:balance-1} follows
from the stress-energy identity~\eqref{stren4}
with the choices $G = B_\rho(a_j)$ and $\X = \e \chi$, where $\chi$
is a smooth-cut off function such that $\chi \equiv 1$ in a neighbourhood
of $B_\rho(a_j)$,
and from~\eqref{eq:ren-energy-BBH}.
\end{proof}
\end{step}

\begin{step}\label{step:zj}
\emph{There exist $z_1,\,\dots,z_{n_\star} \in \R^2$ such
that~\eqref{eq:first-variation} holds, i.e., such that
\[
	\delta \mathbb{V}_\star(\X) = \sum_{j=1}^{n_\star}z_j \cdot \X(a_j)
\]
for any vector field $\X \in C_c^\infty(\Omega,\,\R^2)$.
}

\begin{proof}[Proof of~\eqref{eq:first-variation}]\renewcommand{\qedsymbol}{\ensuremath{\blacksquare}}
We begin by observing that
for any $\X \in C^\infty_c(\Omega,\,\R^2)$ and any
$\mathbf{Y} \in C^\infty_c(\Omega,\,\R^2)$ vector fields
such that
\begin{equation}\label{eq:XY}
	\X(a_j) = \mathbf{Y}(a_j), \qquad
	\spt \mu_\star \cap \spt \X = \spt \mu_\star \cap \spt \mathbf{Y}
	= \{a_j\}.
\end{equation}
there holds
\begin{equation}\label{eq:dVX=dVY}
	\delta \mathbb{V}_\star(\X) = \delta \mathbb{V}_\star(\mathbf{Y}).
\end{equation}
To prove~\eqref{eq:dVX=dVY}, we first notice that,
by linearity, we may assume $\mathbf{Y} = 0$. Let $s > 0$ be as small
as specified at the beginning of the proof, and
assume that
$\varphi \in C_c^\infty([0,\,s))$ satisfies~\eqref{eq:varphi}.
Then, by linearity of the first variation,
\[
	\delta \mathbb{V}_\star(\X)
	= \delta \mathbb{V}_\star((1-\varphi)\X)
	+ \delta \mathbb{V}_\star(\varphi \X)
	= \delta \mathbb{V}_\star(\varphi \X),
\]
where the second equality follows
because $(\mathfrak{S}_\star,\,\mathfrak{v}_\star)$ is stationary
in $\Omega \setminus \spt\mu_\star$. On the other hand, by definition,
\[
	\delta \mathbb{V}_\star(\X)
	= \int_{\mathfrak{S}_\star} \div_{\T_x \mathfrak{S}_\star}(\varphi \X) \,{\d}\zeta_\star
\]
and, by the bound on $\varphi'$ in~\eqref{eq:varphi},
\[
	\norm{\nabla(\varphi \X)}_{L^\infty(B_s(a_j))} \lesssim
	\frac{1}{s} \norm{\X}_{L^\infty(B_s(a_j))} + \norm{\nabla \X}_{L^\infty(B_s(a_j))}
	\lesssim 1,
\]
because $\X(a_j) = 0$ by~\eqref{eq:XY}. Thus,
\[
	\abs{\delta \mathbb{V}_\star(\X)} \lesssim
	\int_{\mathfrak{S}_\star \cap B_s(a_j)} \,{\d}\zeta_\star \to 0
	\qquad \mbox{as } s \to 0
\]
because $\zeta_\star(\spt\mu_\star) = 0$ by~\eqref{eq:z*-sptmu*}.
This shows that~\eqref{eq:dVX=dVY} holds for any two vector fields
$\X$, $\mathbf{Y} \in C_c^\infty(\Omega,\,\R^2)$ satisfying~\eqref{eq:XY}.
On the other hand,
by letting $a_j$ vary in $\spt\mu_\star \cap \Omega$ and using appropriate
cut-off functions as in the above, it follows from the stationarity
of $\mathbb{V}_\star$
in $\Omega \setminus \spt\mu_\star$ that
\begin{equation}\label{eq:XY-bis}
	\X(a_j) = \mathbf{Y}(a_j) \quad \mbox{for any }
	j \in \{1,\,\dots,\,n_\star\} \implies
	\delta \mathbb{V}_\star(\X) = \delta \mathbb{V}_\star(\mathbf{Y}).
\end{equation}
Finally, let $s > 0$ be again as in the above, let
$j$, $k \in \{1,\,\dots,\,n_\star\}$ be any two indexes such that
$j \neq k$. Let $\mathbf{V}_1$, $\mathbf{V}_2 \in C_c^\infty(B_s(0),\,\R^2)$
be any two vector fields with $\mathbf{V}_1(0) = \e_1$
and $\mathbf{V}_2(0) = \e_2$. Given $\X \in C_c^\infty(\Omega,\,\R^2)$,
define
\[
	\overline{\X}(x) := \X(x) - \sum_{j=1}^{n_\star} \sum_{k=1}^2 X_k(a_j) \mathbf{V}_k(x-a_j).
\]
Clearly, $\overline{\X}(a_j) = 0$ for any $j \in \{1,\,\dots,\,n_\star\}$
and thus, from~\eqref{eq:XY-bis}, we have
\[
	0 = \delta \mathbb{V}_\star(\overline{\X}) = \delta \mathbb{V}_\star(\X) - \sum_{j=1}^{n_\star} \sum_{k=1}^2 X_k(a_j) \mathbf{V}_k(x-a_j)
\]
Therefore, defining $z_j$ by setting
\begin{equation}\label{eq:def-zj}
	(z_j)_k := \delta \mathbb{V}_\star(\mathbf{V}_k(\cdot - a_j))
\end{equation}
for any $k \in \{1,\,2\}$ and any $j \in \{1,\,\dots,\,n_\star\}$
concludes the proof of~\eqref{eq:first-variation}.
\end{proof}
\end{step}

\begin{step}\label{step:nablaW-zj}
\emph{
For the same vectors $z_j$ defined by~\eqref{eq:def-zj}, there holds
\begin{equation}\label{eq:nablaW-zj}
	\nabla_{a_j} \mathbb{W}_\star(a_1,\dots,a_{n_\star}) = - 2 z_j
\end{equation}
for any $j \in \{1,\,\dots,\,n_{\star}\}$.
}
\begin{proof}[Proof of~\eqref{eq:nablaW-zj}]\renewcommand{\qedsymbol}{\ensuremath{\blacksquare}}
As observed at the beginning of the proof, we can restrict
to vector fields of the form $\X = \e \varphi$, where
$\varphi \in C^\infty([0,\,s)$ satisfies $\varphi \equiv 1$
for $\rho$, say, in $[0,\,s/2)$ and $\abs{\varphi'} \lesssim 1/s$
and where $s > 0$ is small enough
that $B_s(a_j) \csubset \Omega$ does not contain points of
$\spt\mu_\star$ other than $a_j$.

Keeping the above in mind, for any fixed $\eps > 0$ and any $\e \in \R^2$, we let
\[
	J_\eps(\rho) := \int_{\partial B_\rho(a_j)} \left( \eps \partial_\e \M_\eps \cdot \partial_\nnu  \M_\eps - \frac{\eps}{2} (\e \cdot \nnu) \abs{\nabla \M_\eps}^2 - \frac{1}{\eps}(\e \cdot \nnu) V(\M_\eps) \right).
\]
We already know from Step~\ref{step:dV*=I*} and Step~\ref{step:zj}
that $\mathbb{I}_\star(\X) = -2 \e \cdot z_j$ if $\X = \e \varphi$.
Upon using the explicit expression for $\X$, we can compute
\[
	\partial_i X_k = e_k \varphi'(\abs{x-a_j}) \nu_i, \qquad
	\div \X = (\e \cdot \nnu) \varphi'(\abs{x-a_j}),
\]
where $\nnu(x) = \frac{x - a_j}{\abs{x- a_j}}$, we may
write the integral on the right-hand side of~\eqref{eq:I*}
in a more explicit form, that is,
\begin{equation}\label{eq:I*-J*-compu1}
\begin{split}
	\mathbb{I}_\star(\X) &= \lim_{\eps \to 0}
	\int_{B_s(a_j)} \varphi'(\abs{x-a_j})
	\left( \eps \partial_\e \M_\eps \cdot \partial_\nnu \M_\eps - \eps(\e \cdot \nnu) \left( \frac{1}{2} \abs{\nabla \M_\eps}^2 + \frac{1}{\eps} V(\M_\eps) \right)\right) \\
	&= \lim_{\eps \to 0} \int_0^s \varphi'(\rho) J_\eps(\rho) \,{\d}\rho,
\end{split}
\end{equation}
where the second line follows from Fubini's theorem and the
definition of $J_\eps$. From Step~\ref{step:balance-1}, we
already know that
\[
	\lim_{\eps \to 0} J_\eps(\rho) = J_\star
\]
pointwise for $\rho \in [0,\,s)$, where $J_\star$ does not
depend on $\rho$ and, in fact,
\begin{equation}\label{eq:J*}
	J_\star = - \e \cdot \nabla_{a_j} \mathbb{W}_\star(a_1,\,\dots,\,a_{n_\star}),
\end{equation}
as it is easily seen from~\eqref{eq:balance-1}. On the other hand,
recalling item~\ref{item:mainthm-asymp-strong-conv} of Theorem~\ref{mainthm:asymp},
Lemma~\ref{lemma:V-f-limit-Lp},
and~\eqref{eq:strong-p-conv-rhoeps}, 
we have 
\[
	\nabla \Q_\eps \to \nabla \Q_\star, \qquad
	\frac{1}{\eps}(\abs{\Q_\eps} - 1) \to \kappa_\star, \qquad
	\frac{1}{\eps^2}f_\eps(\Q_\eps,\,\M_\eps) - \frac{1}{\eps}V(\M_\eps) \to 0
\]
strongly in $L^p(B_s(a_j) \setminus B_{s/2}(a_j))$,
for any $p$ with $1 \leq p < \infty$.
This implies that $J_\eps \to J_\star$
strongly in $L^p((s/2,\,s))$ for any $p$ with $1 \leq p < \infty$,
and recalling that $\varphi(\rho) = 1$
for any $\rho \in (0,\,s/2)$, we can pass to the limit
inside the integral on the right-hand side of the first line
in~\eqref{eq:I*-J*-compu1} to obtain
\[
	\mathbb{I}_\star(\X) = J_\star \int_0^s \varphi'(s) \,{\d}s
	= - J_\star.
\]
Thus, from~\eqref{eq:J*} it follows that
\begin{equation}
	\mathbb{I}_\star(\X) = \mathbb{I}_\star(\e \varphi) = -2 \e \cdot z_j
	= \e \cdot \nabla_{a_j} \mathbb{W}_\star(a_1,\,\dots,\,a_{n_\star}),
\end{equation}
and we obtain~\eqref{eq:nablaW-zj} because $\e \in \R^2$ was
arbitrary.
\end{proof}
\end{step}

\begin{step}[Conclusion]
	Since~\eqref{eq:first-variation} was already proved
	in Step~\ref{step:zj} and~\eqref{eq:zj} is just a
	restatement of~\eqref{eq:nablaW-zj}, the desired
	conclusion follows by combining~\eqref{eq:first-variation}
	and~\eqref{eq:nablaW-zj}.
	\qedhere
\end{step}
\end{proof}

Joining Theorem~\ref{thm:stationary-varifold}, Theorem~\ref{thm:properties-nu*},
Proposition~\ref{prop:nu*<<lambda*}, a classical
results of Allard and Almgren \cite{AllardAlmgren} on the structure of stationary
$1$-varifolds with strictly positive density, and \cite[Theorem~1.3]{Bethuel-AC-Acta},
we obtain in Theorem~\ref{thm:structure-S*} the refined structure properties of 
$\mathfrak{S}_\star$ contained in item~\ref{item:S_*} of Theorem~\ref{mainthm:B+C}. 

\begin{theorem}\label{thm:structure-S*}
	The set $\mathfrak{S}_\star = \spt\nu_\star \setminus \spt\mu_\star$
	is locally around $\H^1$-almost every point a union of segments with
	locally constant densities.
	In fact, for every compact set $K \subset \Omega \setminus \spt\mu_\star$
	and any $x_0 \in \mathfrak{S}_{\star, K} \setminus \mathfrak{E}_\star$,
	there exist a unit vector $\e_{x_0}$ and positive numbers $\rho_0$, $c_{x_0}$
	so that $B(x_0,\, \rho_0) \csubset K$,
	\begin{equation}\label{eq:structure-S*-1}
		\mathfrak{S}_\star \cap B(x_0,\,\rho_0)
		= (x_0 - \rho_0 \e_{x_0},\,x_0 + \rho_0 \e_{x_0})
	\end{equation}
	and
	\begin{equation}\label{eq:structure-S*-2}
		\zeta_\star \mres B(x_0,\,\rho_0)
		= c_{x_0}\left(\H^1 \mres (x_0 - \rho_0 \e_{x_0},\,x_0 + \rho_0 \e_{x_0}) \right).
	\end{equation}
	The number $c_{x_0} = \mathfrak{v}_\star(x_0)$ is bounded below by the
	ratio $\eta_{\star, K} \,/\, 2{\rm K}_\beta$, where
	$\eta_{\star,K}$ is given by~\eqref{eq:eta*K}
	and ${\rm K}_\beta$ is the number in~\eqref{eq:density-nu-controlled-by-zeta}.
\end{theorem}

\begin{proof}
Let $K \subset \Omega \setminus \spt\mu_\star$ be any compact set and
let $\mathfrak{S}_{\star, K} = \mathfrak{S}_\star \cap K$.
Recall that, because of item~\ref{item:S_*} of
Theorem~\ref{thm:properties-nu*}, $\mathfrak{S}_{\star,K}$ is also characterised
as the set on which $\mathfrak{e}_\star \geq \eta_{\star,K}$, where
$\eta_{\star,K}$ is given by~\eqref{eq:eta*K}.
\setcounter{step}{0}
\begin{step}[General local structure]\label{step:general-struct}
	From item~\ref{item:cl-o} of Theorem~\ref{thm:properties-nu*} and from
	Proposition~\ref{prop:nu*<<lambda*}, we know that,
	on $\mathfrak{S}_{\star,K}$, the density $\mathfrak{v}_\star$ is bounded
	both from above and from below, by a strictly positive number. More specifically,
	the lower bound is given by $\eta_{\star,K} / 2 K_\beta$,
	by~\eqref{eq:S-gotico} and~\eqref{eq:density-nu-controlled-by-zeta}.
	We are therefore in a position to
	apply \cite[Theorem in Section~3]{AllardAlmgren} with
	$V = (\mathfrak{S}_{\star, K},\,\mathfrak{v}_\star)$.
	Calling $S_{V,K}$ the set of points of $\mathfrak{S}_{\star,K}$
	around which the density $\mathfrak{v}_\star$ is not constant,
	we have $\| V \| = \zeta_\star(S_{V,K}) = 0$
	by \cite[Remark below the Theorem in Section~3]{AllardAlmgren},
	so that we have as well $\H^1(S_{V,K}) = 0$,
	and that
	$\mathfrak{S}_{\star,K} \setminus S_V$ is, in fact, a union of
	segments open relative to $\mathfrak{S}_{\star,K}$,
	on each of which $\mathfrak{v}_\star$ is constant, as claimed.
	Since $\mathfrak{S}_{\star}$ can be look at as the countable union
	of the sets $\mathfrak{S}_{\star, K}$ over an(y) ascending sequence of compact sets
	$K_n \subset \Omega \setminus \spt \mu_\star$,
	it follows that
	$\H^1(S_V) = 0$, where $S_V := \cup_{n \in \mathbb{N}} S_{V,K}$
	does not depend on the chosen sequence $\{K_n\}$.
\end{step}

\begin{step}[Refined local structure]
	Let $x_0 \in \mathfrak{S}_{\star,K} \setminus \mathfrak{E}_\star$.
	Reasoning word-for-word as in \cite[Section~10]{Bethuel-AC-Acta},
	one may improve on Step~\ref{step:general-struct} to show
	that~\eqref{eq:structure-S*-1} and~\eqref{eq:structure-S*-2} hold.
	Being the argument totally analogous to that in \cite{Bethuel-AC-Acta},
	with no but notational changes, we address the reader
	to~\cite{Bethuel-AC-Acta} for full details.
	\qedhere
\end{step}
\end{proof}

\begin{remark}\label{rk:integrality}
	By applying the change of variables $\M_\eps \to \u_\eps$ 
	described in~\cite[Section~3.1]{CDS1} (which 
	amounts to project $\M_\eps$ on the eigenframe 
	of $\Q_\eps$) and then blowing-up, the system~\eqref{EL-M} becomes
	\begin{equation}\label{eq:entire-AC}
		-\Delta \u + \nabla_{\u} h(\u) = 0,
	\end{equation}
	where $\u\colon \R^2 \to \R^2$ and $h\colon \R^2 \to \R$ is exactly the same
	potential as defined in~\cite[Section~3.1]{CDS1}. In particular, the equation
	for the component $u_2$ reads
	\begin{equation}\label{eq:entire-AC-u2}
		-\Delta u_2 + \left(\abs{\u}^2 - 1 + \sqrt{2}\beta\right)u_2 = 0.
	\end{equation}
	Multiplying~\eqref{eq:entire-AC-u2} by $u_2$ and assuming that
	$\sqrt{2}\beta \geq 1$, it is readily seen that $\abs{u_2}^2$ must be a subharmonic function. Therefore, bounded solutions of~\eqref{eq:entire-AC} with~$\sqrt{2}\beta \geq 1$ must have constant~$u_2$, and in fact
	$u_2 \equiv 0$
	(because 0 is the only constant solving~\eqref{eq:entire-AC-u2}).
	Consequently, every bounded entire solution of~\eqref{eq:entire-AC}
	is, in fact, one-dimensional. Recalling the well-known integrality
	result in \cite{HutchinsonTonegawa}, this suggests, at least formally
	and at least if $\sqrt{2}\beta \geq 1$, that the limiting varifold
	$\mathbb{V}_\star$ is \emph{integral}, i.e., the density function
	$\mathfrak{v}_\star$ is integer-valued. However,
	we will not further
	pursue this issue in this paper, deferring rigorous
	analysis of it to future work.
\end{remark}

We have now all the ingredients to complete the proof of Theorem~\ref{mainthm:B+C}.
\begin{proof}[{Proof of Theorem~\ref{mainthm:B+C}}]
	The statement $\zeta_\star = \mathfrak{v}_\star \H^1 \mres \mathfrak{S}_\star$
	is already contained in Proposition~\ref{prop:analogue-of-Bethuel-Prop-6}.
	Combining
	Theorem~\ref{thm:properties-nu*} and Proposition~\ref{prop:nu*<<lambda*},
	we obtain the boundedness from below of the density
	$\mathfrak{v}_\star$ of $\zeta_\star$ on $\mathfrak{S}_\star \cap K$,
	for any compact set $K \subset \Omega \setminus \spt\mu_\star$, with a constant
	$c_K$ depending only on $K$, $\beta$, and the energy bound $\mathcal{E}_0$
	in Proposition~\ref{prop:Eeps-bounded}.
	The boundedness from above, again with a constant
	$C_K$ depending on $K$ and $\mathcal{E}_0$,
	follows from~\eqref{eq:zeta*-density}
	and Proposition~\ref{prop:Eeps-bounded}.  
	The conclusion of
	Theorem~\ref{mainthm:B+C} now follows from
	Theorem~\ref{thm:first-variation} and Theorem~\ref{thm:structure-S*}.
\end{proof}

\bibliographystyle{plain}
\bibliography{UnifConv}

\begin{flushright}
\Addresses
\end{flushright}

\end{document}